\documentclass[a4paper,11pt,fleqn]{article} 
\usepackage{amsmath}
\usepackage{amsthm}
\usepackage{amssymb}
\usepackage[T1]{fontenc} 
\usepackage[utf8]{inputenc} 
\usepackage{mathrsfs} 
\usepackage{euscript}
\usepackage{mathtools}
\usepackage[dvipsnames]{xcolor} 
\usepackage[shortlabels]{enumitem} 
\usepackage{bm} 
\usepackage{charter} 
\usepackage[scaled=0.95]{helvet} 
\usepackage{graphicx}
\usepackage[normalem]{ulem} 
\usepackage[misc]{ifsym} 
\usepackage[toc,titletoc,title]{appendix}

\usepackage{caption}
\usepackage{subcaption}
\definecolor{labelkey}{rgb}{0,0.08,0.45}
\definecolor{refkey}{rgb}{0,0.6,0.0}
\definecolor{Brown}{rgb}{0.45,0.0,0.05}
\definecolor{dgreen}{rgb}{0.00,0.49,0.00}
\definecolor{dblue}{rgb}{0,0.08,0.75}
\definecolor{overleaf_mostly_pure_blue}{rgb}{0, .502, 1} 
\definecolor{overleaf_dark_moderate_blue}{rgb}{.353, .361, .678} 
\definecolor{overleaf_dark_moderate_cyan}{rgb}{.247, .498, .498} 
\definecolor{colorhexa_dark_blue}{rgb}{0, .4, .6} 
\definecolor{colorhexa_dark_orange}{rgb}{.6, .2, 0} 
\RequirePackage[colorlinks,hyperindex]{hyperref} 
\hypersetup{linktocpage=true,citecolor=dblue,linkcolor=dgreen}

\newtheorem{theorem}{Theorem}[section]
\newtheorem{corollary}[theorem]{Corollary}
\newtheorem{lemma}[theorem]{Lemma}
\newtheorem{fact}[theorem]{Fact}
\newtheorem{proposition}[theorem]{Proposition}

\theoremstyle{definition}

\newtheorem{algorithm}[theorem]{Algorithm}
\newtheorem{remark}[theorem]{Remark}

\numberwithin{equation}{section}

\providecommand{\norm}[1]{\lVert#1\rVert}
\providecommand{\scalarp}[1]{\langle#1\rangle}
\providecommand{\abs}[1]{\lvert#1\rvert}
\newcommand{\EE}{\ensuremath{\mathsf E}}
\newcommand{\Js}{\ensuremath{\mathsf{J}}}

\newcommand{\As}{\ensuremath{\mathsf{A}}}
\newcommand{\Bs}{\ensuremath{\mathsf{B}}}
\newcommand{\Id}{\ensuremath{\mathsf{Id}}}
\newcommand{\PP}{\ensuremath{\mathsf P}}
\newcommand{\R}{\ensuremath \mathbb{R}}
\newcommand{\HH}{\ensuremath \mathsf{H}}
\newcommand{\GG}{\ensuremath \mathsf{G}}
\newcommand{\WW}{\ensuremath \mathsf{W}}
\newcommand{\VV}{\ensuremath \mathsf{V}}
\newcommand{\XX}{\ensuremath \mathsf{X}}
\newcommand{\KK}{\ensuremath \mathsf{K}}

\newcommand{\Gammas}{\ensuremath \mathsf{\Gamma}}

\newcommand{\pp}{\ensuremath \mathsf{p}}
\newcommand{\xs}{\ensuremath \mathsf{x}}
\newcommand{\bxs}{\ensuremath \bm{\mathsf{x}}}
\newcommand{\ds}{\ensuremath \mathsf{d}}
\newcommand{\bds}{\ensuremath \bm{\mathsf{d}}}

\newcommand{\xx}{\ensuremath \mathsf{x}}

\newcommand{\bs}{\ensuremath \mathsf{b}}
\newcommand{\ww}{\ensuremath \mathsf{w}}
\newcommand{\zz}{\ensuremath \mathsf{z}}

\newcommand{\uu}{\ensuremath \mathsf{u}}
\newcommand{\yy}{\ensuremath \mathsf{y}}
\newcommand{\uuu}{\alpha}
\newcommand{\N}{\ensuremath \mathbb{N}}
\newcommand{\iter}{\ensuremath k}
\newcommand{\prox}{\ensuremath{\text{\textsf{prox}}}}

\DeclareMathOperator*{\dom}{\ensuremath{\text{\textrm{dom}}}}

\DeclareMathOperator*{\argmin}{\text{\rm{argmin}}}

\newcommand{\bHH}{\ensuremath{\bm{\mathsf{H}}}}
\newcommand{\bSS}{\ensuremath{\mathsf{S}}}
\newcommand{\bWW}{\ensuremath{{\WW}}}

\newcommand{\bVV}{\ensuremath{{\VV}}}

\newcommand{\bXX}{\ensuremath{{\XX}}}
\newcommand{\bAs}{\ensuremath{{\As}}}
\newcommand{\bxx}{\ensuremath \bm{\xx}}

\newcommand{\byy}{\ensuremath \bm{\yy}}
\newcommand{\bzz}{\ensuremath \bm{\zz}}

\newcommand{\buu}{\ensuremath \bm{\uu}}

\newcommand{\bbs}{\ensuremath \bs}

\newcommand{\bx}{\ensuremath \bm{x}}
\newcommand{\by}{\ensuremath \bm{y}}

\newcommand{\bw}{\ensuremath \bm{w}}

\newcommand{\bvarphi}{\ensuremath{{\varphi}}}
\newcommand{\bDelta}{\ensuremath{{\bm{\Delta}}}}
\newcommand{\bGammas}{\ensuremath{{\Gammas}}}
\newcommand{\bbar}{\bm \bar}
\newcommand{\bhat}{\bm \hat}

\begin{document}

\title{ {\sffamily Convergence of an Asynchronous Block-Coordinate Forward-Backward Algorithm for Convex Composite Optimization}}
\author{Cheik Traor\'e$\,^{\textrm{\Letter}}\,$\thanks{Malga Center, DIMA, Universit\`a degli Studi di Genova, Genova, Italy (traore@dima.unige.it). {\scriptsize This project has received funding from the European Union’s Horizon 2020 research and innovation programme under the Marie Skłodowska-Curie grant agreement No 861137.}}\,,\ \
Saverio Salzo\thanks{Istituto Italiano di Tecnologia, Genova, Italy (saverio.salzo@iit.it).}\ \
and Silvia Villa\thanks{Malga Center, DIMA, Universit\`a degli Studi di Genova, Genova, Italy (silvia.villa@unige.it). {\scriptsize She acknowledges the financial support of the European
Research Council (grant SLING 819789), the AFOSR projects FA9550-17-1-0390, FA8655-22-1-7034, and BAAAFRL-
AFOSR-2016-0007 (European Office of Aerospace Research and Development), and the EU H2020-MSCA-RISE project NoMADS - DLV-777826.}}
        }
\date{}
\maketitle
\begin{abstract}
In this paper, we study the convergence properties of a randomized block-coordinate descent
algorithm for the minimization of a composite convex objective function,
where the block-coordinates are updated asynchronously and randomly according to an arbitrary
probability distribution. We prove
that the iterates generated by the algorithm form a stochastic quasi-Fej\'er sequence and thus
converge almost surely to a minimizer of the objective function.
Moreover, we prove a general sublinear rate of convergence in expectation for the function values
and a linear rate of convergence in expectation under an error bound condition of Tseng type. Under the same condition  strong convergence of the iterates is provided as well as their linear convergence rate.
\end{abstract}

\vspace{1ex}

\section{Introduction}

We consider the composite minimization problem
\begin{align}
\label{mainprob}
\underset{\bxx \in \bHH}{\text{minimize }} F(\bxx) := f(\bxx) + g(\bxx),\qquad 
g(\bm{\bxx}) := \sum_{i=1}^m g_i(\xx_i),
\end{align}
where $\bHH$ is the direct sum of $m$ separable real Hilbert spaces 
$(\HH_i)_{1 \leq i \leq m}$, that is,
$\bHH = \bigoplus_{i=1}^{m} \HH_i$ and 
the following assumptions are satisfied unless stated otherwise.
\begin{enumerate}[label=A\arabic*]
\item \label{eq:A1} $f \colon \bHH \to \mathbb{R} $ is convex and differentiable.
\item \label{eq:A2} For every $i \in \{1, \cdots, m\}$, 
$g_i \colon \HH_i \to \left]-\infty,+\infty\right]$ 
is proper convex and lower semicontinuous.
\item \label{eq:A3} For all $\bxx \in \bHH$ and $i \in \{1, \cdots, m\}$, the map  
$\nabla f(\xx_1, \dots,\xx_{i-1}, \cdot, \xx_{i+1}, \dots, \xx_m)\colon \HH_i \to \bHH$ is Lipschitz continuous 
with constant $L_\mathrm{res}>0$ and the map
$\nabla_i f(\xx_1, \dots,\xx_{i-1}, \cdot, \xx_{i+1}, \dots, \xx_m)\colon \HH_i \to \HH_i$ is Lipschitz continuous 
with constant $L_i$. Note that $L_{\max}\coloneqq \max_i L_i   \leq L_\mathrm{res}$ and $L_{\min} \coloneqq \min_i L_i$.
\item \label{eq:A4} $F$ attains its minimum $F^*:= \min F$ on \bHH.
\end{enumerate}

To solve problem \ref{mainprob}, we use the following asynchronous block-coordinate descent algorithm. It is an extension of the parallel block-coordinate proximal gradient method considered in \cite{salzo2021parallel} to the asynchronous setting, where an {\em inconsistent} delayed gradient vector may be processed at each iteration.

\begin{algorithm}
\label{algoAsymain}

Let $(i_k)_{k \in \N}$ be a sequence of i.i.d.~random variables with values
in $[m]:=\{1, \dots, m\}$ and $\pp_i$ be the probability of the event $\{i_k = i\}$, for every $i \in [m]$. Let $(\bds^k)_{k \in \N}$ be a sequence of 
integer delay vectors, $\bds^k = (\ds_1^k, \dots, \ds_m^k) \in \N^m$ such that 
$\max_{1\leq i\leq m} \ds^k_i \leq \min\{k,\tau\}$ for some $\tau \in \N$. This delay vector is deterministic and independent from the block coordinates selection process $(i_k)_{k \in \N}$.
Let $(\gamma_i)_{1 \leq i \leq m} \in \R_{++}^m$ and
$\bx^0 =(x^0_{1}, \dots, x^0_{m}) \in \bHH$
be a constant random variable. Iterate
 
\vspace{-2.5ex}
\begin{equation}
\label{eq:algoPRCD2}
\begin{array}{l}
\text{for}\;k=0,1,\ldots\\
\left\lfloor
\begin{array}{l}
\text{for}\;i=1,\dots, m\\[0.7ex]
\left\lfloor
\begin{array}{l}
x^{k+1}_i = 
\begin{cases}
\prox_{\gamma_{i_{k}} g_{i_{k}}} \big(x^k_{i_{k}} - \gamma_{i_{k}} \nabla_{i_{k}} f (\bx^{k-\bds^k})\big) &\text{if } i=i_{k}\\
x^k_i &\text{if } i \neq i_{k},
\end{cases}
\end{array}
\right.
\end{array}
\right.
\end{array}
\end{equation}
where $\bx^{k-\bds^k} = (x_1^{k - \ds^k_1}, \dots, x_m^{k - \ds^k_m})$.
\end{algorithm}
In this work, we assume the following stepsize rule
\begin{equation}
(\forall\, i \in [m])\quad\gamma_i( L_{i} 
+ 2\tau L_{\mathrm{res}}\pp_{\max}/ \sqrt{\pp_{\min}})< 2,
\end{equation}
where $\pp_{\max} := \max_{1 \leq i \leq m} \pp_i$ and $\pp_{\min} 
:= \min_{1 \leq i \leq m} \pp_i$.
If there is no delay, namely $\tau = 0$, the usual stepsize rule $\gamma_i < 2/L_i$ is obtained \cite{ComWaj05,salzo2021book}. 

The presence of the delay vectors in the above algorithm 
allows to describe a parallel computational model on multiple cores, as we explain below. 
\subsection{Asynchronous models}
In this section we discuss an example of a parallel computational model, occurring in 
shared-memory system architectures, which can be covered by the proposed algorithm.
Consider a situation where we have a machine with multiple cores. They all have access to 
a shared data $\bx = (x_1, \dots, x_m)$ and each core updates a block-coordinate $x_i$, $i \in[m]$,
asynchronously without waiting for the others. The iteration's counter $k$ is increased any time 
a component of $\bx$ is updated. When a core is given a coordinate to update, it has to read 
from the shared memory and compute a partial gradient. While performing these two operations, 
the data $\bx$ may have been updated by other cores. So, when the core is updating its assigned
coordinate at iteration $k$, the gradient might no longer be up to date. This phenomenon is
modelled by using a delay vector $\bds^k$ and evaluating the partial gradient at $\bx^{k-\bds^k}$
as in Algorithm \ref{algoAsymain}.
Each component of the delay vector reflects how many times the
corresponding coordinate of $\bx$ have been updated since the core has read this particular coordinate from the shared memory. Note that different delays among the coordinates may arise since the shared data 
may be updated during the reading phase, so that the partial gradient ultimately is computed at 
a point which may not be consistent with any past instance of the shared data.
This situation is called \emph{inconsistent read} \cite{bertsekas1989parallel} and, in practice, allows a  reading phase without any lock.
By contrast, in a \emph{consistent read} model \cite{liu2014asynchronous,niu2011hogwild}, 
a lock is put during the reading phase and the delay originates only while computing the 
partial gradient. The delay is the same for all the block-coordinates, so that the value read 
by any core is a past instance of the shared data. However,  for our theoretical study it does not make any difference considering an inconsistent or a consistent reading setting, because in the end only the maximum delay matters. In the literature other paper also consider the inconsistent read, see \cite{liu2015asynchronous, cannelli2019asynchronous, davis2016asynchronous}.

We remark that, in our setting, for all $k \in \N$, the delay vector $\bds^k$ is considered to be
a parameter that does not dependent on the random variable $i_k$, similarly to the works
\cite{liu2015asynchronous, liu2014asynchronous, davis2016asynchronous, hannah2018unbounded}. In this way, the stochastic attribute of the sequence $(\bx_k)_{k \in \N}$ is not determined by the delay, but it only comes from the stochastic selection of the block-coordinates. Some papers consider the case where the delay vector is a stochastic variable that may depend on 
$i_k$ \cite{sun2017asynchronous,cannelli2019asynchronous} or that it is unbounded
\cite{sun2017asynchronous,hannah2018unbounded}. Those setting are natural extensions to our work that we are considering for future work. Finally, a completely deterministic model, both in 
the block's selection and delays is studied in \cite{combettes2018asynchronous}.

 \subsection{Related work}

The topic on parallel asynchronous algorithm is not a recent one. In 1969, Chazan and Miranker
\cite{chazan1969chaotic} presented an asynchronous method for solving linear equations. Later on,
Bertsekas and Tsitsiklis \cite{bertsekas1989parallel} proposed an \emph{inconsistent read} 
model of asynchronous computation.   Due to the availability of large amount of data and the
importance of large scale optimization, in recent years we have witnessed a surge of interest 
in asynchronous algorithms. They have been studied and adapted to many optimization problems 
and methods such as stochastic gradient descent
\cite{agarwal2011distributed,niu2011hogwild,feyzmahdavian2016asynchronous,paine2013gpu,lian2016comprehensive}, 
randomized Kaczmarz algorithm \cite{liu2014asynchronous2}, and stochastic coordinate descent
\cite{avron2015revisiting, liu2014asynchronous, peng2016arock, um2020solver,sun2017asynchronous}.

In general, stochastic algorithms can be divided in two classes. The first one is when the 
function $f$ is an expectation i.e., $f(\xx) = \EE[h(\xx ; \xi)]$. At each iteration 
$k$ only a stochastic gradient  $\nabla h(\cdot ; \xi_k)$ is computed based on the current 
sample $\xi_k$. In this setting, many asynchronous versions have been proposed, where delayed
stochastic gradients are considered, see \cite{nedic2001distributed, feyzmahdavian2016asynchronous,DBLP:journals/corr/abs-1911-03444, Chen_2021, lian2015asynchronous, mai2020convergence}. 
The second class, which is the one we studied, is that of randomized block-coordinate methods. Below we describe the related literature.

The work  \cite{liu2015asynchronous} studied a problem and a model of asynchronicity which is
similar to ours, but the proposed algorithm AsySPCD requires that the random variables 
$(i_k)_{k \in \N}$ are uniformly distributed (i.e, $p_i =1/m$) and that the stepsize is the 
same for all the  block-coordinates. This latter assumption is an important limitation, since 
it does not exploit the possibility of adapting the stepsizes to the block-Lipschitz constants 
of the partial gradients, hence allowing longer steps along block-coordinates. A linear rate of
convergence is also obtained by exploiting a quadratic growth condition which is essentially
equivalent to our error bound condition \cite{drusvyatskiy2018error}. For a discussion on the limitations of \cite{liu2015asynchronous} and the improvements we bring, see Remark \ref{rmk:20211216a} point \ref{item:20211216a} and Section~\ref{sect:20220825a} on numerical experiments.

In the nonconvex case, \cite{davis2016asynchronous} considers an asynchronous algorithm which may
select the blocks both in an almost cyclic manner or randomly with a uniform probability. In the
latter case, it is proved that the cluster points of the sequence of the iterates are almost surely
stationary points of the objective function. However, the convergence of the whole sequence is not
provided, nor is given any rate of  convergence for the function values. Moreover, under the 
Kurdyka-{\L}ojasiewicz (KL) condition \cite{drusvyatskiy2018error, bolte2017error}, linear
convergence is also derived, but it is restricted  to the deterministic case.

To conclude, we note that our results, when specialized to the case of zero delays,  fully 
recover the ones given in  \cite{salzo2021parallel}.
\subsection{Contributions}
The main contributions of this work are summarized below:
\begin{itemize}
\item We first prove the almost sure weak convergence of the iterates $(\bx^k)_{k \in \N}$,
generated by Algorithm \ref{algoAsymain}, to a  random variable $\bx^*$ taking values in 
$\argmin F$. At the same time, we prove a sublinear rate of convergence of the function values 
in expectation, i.e, $\displaystyle \EE[F(\bx^k)] - \min F = o(1/k)$. We also provide for the 
same quantity an explicit rate  of $\mathcal{O}(1/k)$, see Theorem~\ref{thm:main1}. 
\item Under an error bound condition of Luo-Tseng type, on top of the strong convergence a.s of 
the iterates, we prove linear convergence in expectation of the function values and in mean of 
the iterates, see Theorem \ref{thm:main3}. 
\end{itemize}
We improve the state-of-the-art under several aspects: we consider an arbitrary probability for 
the selection of the blocks; the adopted stepsize rule improves over the existing ones, and
coincides with the one in \cite{davis2016asynchronous} in the special case of uniform selection 
of the blocks --- in particular, it allows for larger stepsizes when the number of blocks grows;
the  almost sure convergence of the iterates in the convex and stochastic setting is new and relies
on a stochastic quasi-Fejerian analysis; linear convergence under an error bound condition is also
new in the asynchronous stochastic scenario. 

The rest of the paper is organized as follows. In the next subsection we set up basic notation.
In Section~\ref{sec:preliminaries} we recall few facts and we provide some preliminary results. 
The general convergence analysis is given in Section~\ref{sec:convergence} where the main
Theorem~\ref{thm:main1} is presented. Section~\ref{sec:errorbound} contains the 
convergence theory under an additional error bound condition, while applications are discussed in
Section~\ref{sec:applications}. The majority of proofs are postponed to Appendices 
\ref{AppendixA} and \ref{sec:appB}.

\subsection{Notation}
We set $\R_+ = \left[0,+\infty\right[$
and $\R_{++} = \left]0,+\infty\right[$.
For every integer $\ell \geq 1$ we define $[\ell] = \{1, \dots, \ell\}$. For all $i \in [m]$, we denote indifferently the scalar products of $\bHH$ and $\HH_i$ by $\langle\cdot, \cdot\rangle$ and:
\begin{align*}
    (\forall \bm{\bxx} =(\xx_1, \cdots, \xx_m), \bm{\byy} =(\yy_1, \cdots, \mathrm{y}_m) \in \bHH) \quad \langle \bm{\bxx}, \bm{\byy} \rangle = \sum_{i=1}^{m} \langle \xx_i, \yy_i \rangle.
\end{align*}
$\|\cdot\|$ and $|\cdot|$ represent the norms associated to their scalar product
in $\bHH$ and in any of $\HH_i$ respectively. We also consider the canonical embedding, for all $i = 1,2, \cdots, m$,
$\Js_i\colon\HH_i \rightarrow \bHH$, $\xx_i \mapsto (0,\cdots, 0, \xx_i,0,\cdots,0)$, with $x_i$ in the $i^{th}$ position.
Random vectors and variables are defined on the underlying probability space $(\Omega, \mathfrak{A}, \PP)$. The default font is used for random variables while sans serif font is used for their realizations or deterministic variables.
Let $(\alpha_i)_{1 \leq i \leq m} \in \R^m_{++}$.  The direct sum operator $\bAs = \bigoplus_{i=1}^m \alpha_i \Id_i$, where $\Id_i$ is the identity operator on $\HH_i$, is 
\begin{align*}
    \bAs \colon \bHH &\rightarrow \bHH \\
    \bxx=(\xx_i)_{1 \leq i \leq m} &\mapsto (\alpha_i \xx_i)_{1 \leq i \leq m}
\end{align*}
This operator defines an equivalent scalar product on $\bHH$  as follows
\begin{equation*}
(\forall\, \bxx \in \bHH)(\forall\, \byy \in \bHH)\qquad\scalarp{\bxx, \byy}_{\bAs} = \scalarp{\bAs \bxx, \byy} 
= \sum_{i=1}^m \alpha_i \scalarp{\xx_i, \yy_i},
\end{equation*}
which gives the norm $\norm{\bxx}_{\bAs}^2 
= \sum_{i=1}^m \alpha_i |\xx_i|^2$. 
We let 
\begin{equation*}
\bVV=\bigoplus_{i=1}^m \pp_i \Id_i, \quad \bGammas^{-1} = \bigoplus_{i=1}^m \frac{1}{\gamma_i} \Id_i,
\quad\text{and}\quad \bWW = \bigoplus_{i=1}^m \frac{1}{\gamma_i\pp_i} \Id_i,
\end{equation*}
where for all $i \in [m]$, $\gamma_i$ and $\pp_i$ are defined in Algorithm \ref{algoAsymain}. We set  $\pp_{\max }:=\max_{1 \leq i \leq m} \pp_{i}$ and $\pp_{\min }:=\min_{1 \leq i \leq m} \pp_{i}.$
Let $\bvarphi\colon \bHH \to \left]-\infty,+\infty\right]$ be proper, convex, and lower semicontinuous. The domain of $\bvarphi$ is $\dom \bvarphi 
= \{\bxx \in \HH \,\vert\, \bvarphi(\bxx)<+\infty\}$ and the set 
of minimizers of $\bvarphi$ is $\argmin \bvarphi 
= \{\bxx \in \bHH \,\vert\, \bvarphi(\bxx) = \inf \bvarphi\}$. We recall that the proximity operator of $\varphi$ is $\prox_{\varphi}(\bx) = \argmin_{\by \in \bHH} \varphi(\by) + \frac{1}{2} \norm{\by-\bx}^2$.
If the function $\bvarphi\colon \bHH\to \R$ is differentiable, 
then for all $\buu, \bxx \in \bHH$ and any symmetric positive definite operator $\bAs$, we have $\scalarp{\nabla^{\bAs} \bvarphi(\bxx), \buu}_{\bAs} = \scalarp{\nabla \varphi(\bxx), \buu}$, where $\nabla^{\bAs}$
denotes the gradient operator in the norm $\norm{\cdot}_{\bAs}$. If $\bSS \subset \bHH$ and $\bxx \in \bHH$, we set $\mathrm{dist}_{\bAs}(\bxx,\bSS) = \inf_{\bm{\zz} \in \bSS} \norm{\bxx - \bm{\zz}}_{\bAs}$. We also denote by $\prox_{\varphi}^{\bAs}$ the proximity operator of $\varphi$ with the norm $\norm{\cdot}_{\bAs}$.
\section{Preliminaries}
\label{sec:preliminaries}
In this section we present basic definitions and facts that are used in the rest of the paper. Most of them are already known, and we include them for clarity.

In the rest of the paper, we extend the definition of $\bx^k$ by setting
$\bx^k = \bx^0$ for every $k \in \{-\tau, \dots, -1\}$. Using the notation of Algorithm \ref{algoAsymain}, we also set, for any $k\in\mathbb{N}$
\begin{equation}
\label{eq:20170921b}
\left\{
\begin{aligned}
\bhat{\bx}^k &= \bx^{k - \bds^k}\\
\bbar{x}^{k+1}_i &= \prox_{\gamma_i g_i} \big( x_i^{k}
- \gamma_i \nabla_i f(\bhat{\bx}^k) \big) \text{ for all } i \in [m]\\
\bx^{k+1} &= \bx^k + \Js_{i_k}\big[\prox_{\gamma_{i_{k}} g_{i_{k}}} \big( x^{k}_{i_{k}} 
- \gamma_{i_k} \nabla_{i_{k}} f(\bhat{\bx}^k) \big) - x^k_{i_{k}} \big] \\
\bDelta^k &=  \bx^k - \bbar{\bx}^{k+1}.
\end{aligned}
\right.
\end{equation}
With this notation, we have
\begin{equation}
\label{eq:20170918a}
\bar{x}^{k+1}_{i_{k}} = \prox_{\gamma_{i_{k}} g_{i_{k}}} \big( x^{k}_{i_{k}} 
- \gamma_{i_{k}} \nabla_{i_{k}} f(\bhat{\bx}^k) \big) = x^{k+1}_{i_{k}};
\qquad \Delta^k_{i_{k}} = x^k_{i_{k}} - x^{k+1}_{i_{k}}.
\end{equation}
We remark that the random variables $\bx^k$ and $\bbar{\bx}^{k+1}$ depend on the previously selected blocks, and related delays. More precisely, we have
\begin{equation}
\label{eq:20170918b}
\begin{aligned}
\bx^k &= \bx^k(i_0,\dots, i_{k-1}, \bds^{0}, \dots, \bds^{k-1})\\
\bbar{\bx}^{k+1} &= \bbar{\bx}^{k+1}(i_0,\dots, i_{k-1}, \bds^{0}, \dots, \bds^{k}).
\end{aligned}
\end{equation}
From  \eqref{eq:20170921b} and \eqref{eq:20170918a}, we derive
\begin{equation}
\label{eq:20170921l}
\frac{x^k_{i_{k}} - x^{k+1}_{i_{k}}}{\gamma_{i_{k}}} - \nabla_{i_{k}} f (\bhat{\bx}^k) 
\in \partial g_{i_{k}}(x^{k+1}_{i_{k}})
\quad\text{and}\quad
\frac{x^k_{i} - \bbar{x}^{k+1}_{i}}{\gamma_{i}} - \nabla_{i} f (\bhat{\bx}^k) 
\in \partial g_{i}(\bbar{x}^{k+1}_{i})
\end{equation}
and therefore, for every $\bxs \in \bHH$
\begin{align}
\label{eq:20170925a}&\scalarp{\nabla_{i_{k}} f(\bhat{\bx}^k) - \frac{\Delta^k_{i_{k}}}{\gamma_{i_{k}}}, x_{i_{k}}^{k+1} - \xs_{i_{k}} }
+  g_{i_{k}}(x^{k+1}_{i_{k}}) - g_{i_{k}}(\xs_{i_{k}}) \leq 0.
\end{align}
Suppose that $\bxs$ and $\bxs^\prime$ in $\bHH$ differ  only for one component, say that of index $i$, then it follows from Assumption \ref{eq:A3} and the Descent Lemma \cite[Lemma 1.2.3]{nesterov2003introductory}, that
\begin{align}
\nonumber f(\bxx^\prime) &= f(\xs_1,\dots, \xs_{i-1}, \xs^\prime_{i}, \xs_{i+1}, \cdots, \xs_{m})\\
 \label{eq:20170921m}
&\leq f(\bxs) + \langle\nabla_{i}f(\bxs),\xs^\prime_{i} - \xs_i\rangle + \frac{L_{i}}{2} \abs{\xs^\prime_i - \xs_i}^2 \\
\label{eq:20170925d}&\leq f(\bxs) + \scalarp{\nabla f(\bxs), \bxs^\prime - \bxs}+
\frac{L_{\max}}{2} 
\norm{\bxs^\prime - \bxs}^2.
\end{align}

We finally need the following results on 
the convergence of stochastic quasi-Fej\'er sequences and monotone summable positives sequences.
\begin{fact}[\cite{combettes2015stochastic}, Proposition 2.3]\label{prop:stochfej}
Let $\mathsf{S}$ be $a$ nonempty closed subset of a real Hilbert space $\bHH$. 
Let $\mathscr{F}=\left(\mathcal{F}_{n}\right)_{n \in \mathbb{N}}$ be a sequence of 
sub-sigma algebras of $\mathcal{F}$ such that $(\forall n \in \mathbb{N})\  \mathcal{F}_{n} 
\subset \mathcal{F}_{n+1}$. We denote by $\ell_{+}(\mathscr{F})$ the set of sequences of 
$\R_+$-valued random variables $\left(\xi_{n}\right)_{n \in \mathbb{N}}$ 
such that, for every $n \in \mathbb{N}, \xi_{n}$ is $\mathcal{F}_{n}$-measurable.
We set 
\begin{equation*}
\ell_{+}^{1}(\mathscr{F})=
\bigg\{\left(\xi_{n}\right)_{n \in \mathbb{N}} 
\in \ell_{+}(\mathscr{F}) \,\bigg\vert\, \sum_{n \in \mathbb{N}} \xi_{n}<+\infty \quad \PP\text{-a.s.}\bigg\}.
\end{equation*}
Let $\left(x_{n}\right)_{n \in \mathbb{N}}$ be a
sequence of $\bHH$-valued random variables. Suppose that, for every $\mathsf{z} \in \mathsf{S}$,
there exist $\left(\chi_{n}(\mathsf{z})\right)_{n \in \mathbb{N}} \in \ell_{+}^{1}(\mathscr{X}),
\left(\vartheta_{n}(\mathsf{z})\right)_{n \in \mathrm{N}} \in \ell_{+}(\mathscr{X}),$ and
$\left(\eta_{n}(\mathsf{z})\right)_{n \in \mathbb{N}} \in \ell_{+}^{1}(\mathscr{X})$ such that the 
stochastic quasi-F\'ejer property is satisfied $\PP$-a.s.:
\begin{equation*}
(\forall n \in \mathbb{N}) \quad 
\mathsf{E} \big[\norm{x_{n+1}-\mathsf{z}}^2 \mid \mathcal{F}_{n}\big] +\vartheta_{n}(\mathsf{z}) \leqslant\left(1+\chi_{n}(\mathsf{z})\right) \left\|x_{n}-\mathsf{z}\right\|^2 + \eta_{n}(\mathsf{z}).
\end{equation*}
Then the following hold:
\begin{enumerate}[label={\rm (\roman*)}]
\item $\left(x_{n}\right)_{n \in \mathbb{N}}$ is bounded $\PP$-a.s.
\item Suppose that the set of weak cluster points of the sequence  $\left(x_{n}\right)_{n \in \mathbb{N}}$ is $\PP$-a.s.~contained in $\mathsf{S}$.
Then $\left(x_{n}\right)_{n \in \mathbb{N}}$ weakly converges $\PP$-a.s. to an $\mathsf{S}$-valued random variable.
\end{enumerate}
\end{fact}

\begin{fact}[{\cite[Example~5.1.5]{durrett2019probability}}]
\label{fact:durrett}
Let $\zeta_1$ and $\zeta_2$ be independent random variables with values in 
the measurable spaces $\mathcal{Z}_1$ and $\mathcal{Z}_2$ respectively. 
Let $\varphi\colon \mathcal{Z}_1\times \mathcal{Z}_2 \to \R$ be measurable
and suppose that $\EE[\abs{\varphi(\zeta_1,\zeta_2)}]<+\infty$.
Then $\EE[\varphi(\zeta_1,\zeta_2) \,\vert\, \zeta_1] = \psi(\zeta_1)$,
where for all $z_1 \in \mathcal{Z}_1$, $\psi(z_1) = \EE[\varphi(z_1, \zeta_2)]$.
\end{fact}

\begin{fact}
\label{fact:convmonseq}
Let $(a_k)_{k \in \N} \in \R_+^{\N}$ be a
decreasing sequence of positive numbers and let $b \in \R_+ $ such that $\sum_{k \in \N} a_k \leq b<+\infty$.
Then $a_{k} = o(1/(k+1))$ and for every $k \in \N$, $a_k \leq b/(k+1)$.
\end{fact}

\begin{fact}\label{fact:decomp}
Let $(a_k)_{k \in \N} \in \R_+^{\N}$ be a sequence of positive numbers. 
$(\forall\, n, k \in \mathbb{Z}, k \geq n)$, 
$$\sum_{h=n}^{k-1} a_h = \sum_{h=n}^{k-1} (h- n + 1) a_h - \sum_{h=n+1}^k (h - n) a_h  + (k-n) a_k.$$
\end{fact}
\subsection{Auxiliary lemmas}
Here we collect technical lemmas needed for our analysis, using the notation given in \eqref{eq:20170921b}. For reader's convenience, we provide all the proofs in Appendix~\ref{app:proof}.  

The following result appears in \cite[page 357]{liu2015asynchronous}.
\begin{lemma}
\label{p:20170918c}  Let $(\bx_k)_{k \in \N}$ be the sequence generated by Algorithm \ref{algoAsymain}.
We have
\begin{equation}
\label{eq:20170918d}
(\forall\, k \in \N)\quad  \bx^k = \bhat{\bx}^k - \sum_{h \in J(k)} (\bx^h - \bx^{h+1}),
\end{equation}
where $J(k) \subset \{k-\tau, \dots, k-1\}$
is a random set.
\end{lemma}

The next lemma bounds the difference between the delayed and the current gradient in terms of the steps along the block coordinates, see \cite[equation A.7]{liu2015asynchronous}.
\begin{lemma}
\label{p:20170921n}
Let $(\bx_k)_{k \in \N}$ be the sequence generated by Algorithm \ref{algoAsymain}. It follows
\begin{equation*}
(\forall\, k \in \N)\quad\norm{\nabla f(\bx^k) - \nabla f(\bhat{\bx}^{k})} \leq L_{\mathrm{res}}
\sum_{h \in J(k)} \norm{\bx^{h+1} - \bx^{h}}.
\end{equation*}
\end{lemma}

\begin{remark} 
\label{rmk:20210618a}
Since $\|\cdot\|^2_{\bVV} \leq \pp_{\max}\|\cdot\|^2$ and $\|\cdot\|^2 \leq \pp_{\min}^{-1}\|\cdot\|^2_{\bVV}$, Lemma~\ref{p:20170921n} yields
\begin{align*}
\norm{\nabla f (\bx^k) - \nabla f(\bhat{\bx}^{k})}_{\bVV}
&\leq \sqrt{\pp_{\max}} \norm{\nabla f (\bx^k) - \nabla f(\bhat{\bx}^{k})}\\[1ex]
&\leq L_{\mathrm{res}} \sqrt{\pp_{\max}}
\sum_{h \in J(k)} \norm{\bx^{h + 1} - \bx^{h}}\\
&\leq L_{\mathrm{res}} \frac{\sqrt{\pp_{\max}}}{\sqrt{\pp_{\min}}}
\sum_{h \in J(k)} \norm{\bx^{h + 1} - \bx^{h}}_{\bVV}.
\end{align*}
We set $\displaystyle L_{\mathrm{res}}^{\bVV} = L_{\mathrm{res}} \frac{\sqrt{\pp_{\max}}}{\sqrt{\pp_{\min}}}$.
\end{remark}


The result below yields a kind of inexact convexity inequality due to the presence of the delayed gradient vector. It is our variant of \cite[Equation A.20]{liu2015asynchronous}.
\begin{lemma}\label{eq:20210308a} Let $(\bx_k)_{k \in \N}$ be a sequence generated by 
Algorithm \ref{algoAsymain}. Then, for every $k \in \N$,
\begin{equation*}
(\forall\, \bxs \in \bHH)\quad \scalarp{\nabla f(\bhat{\bx}^k), \bxs - \bx^k} \leq f(\bxs) - f(\bx^k) 
+ \frac{\tau L_\mathrm{res}}{2 } \sum_{h \in J(k)} \norm{\bx^{h} - \bx^{h+1}}^2.
\end{equation*} 
\end{lemma}
The result below generalizes to the asynchronous case
Lemma 4.3 in \cite{salzo2021parallel}.
\begin{lemma}
\label{lem:20200126a}
 Let $\bHH$ be $a$ real Hilbert space. Let $\varphi\colon\bHH \to \mathbb{R}$ be differentiable and convex, and $\left.\left.\psi\colon \bHH \to\right]-\infty,+\infty\right]$ be proper, lower semicontinuous and convex. 
 Let $\bxx, \hat{\bxx} \in \bHH$ and set $\bxx^{+}=\operatorname{prox}_{\psi}(\bxx-\nabla \varphi(\hat{\bxx})) .$ Then, for every $\bzz \in \bHH$,
\begin{align*}
\left\langle\bxx-\bxx^{+}, \bzz-\bxx\right\rangle 
&\leq \psi(\bzz)-\psi(\bxx)+\langle\nabla \varphi(\hat{\bxx}), \bzz-\bxx\rangle \\[1ex]
&\quad +\psi(\bxx)-\psi\left(\bxx^{+}\right) +\left\langle\nabla \varphi(\hat{\bxx}), \bxx-\bxx^{+}\right\rangle-\lVert \bxx-\bxx^{+}\rVert^{2}.
\end{align*}
\end{lemma}

\section{Convergence analysis}
\label{sec:convergence}
In this section we assume just convexity of the objective function and we provide worst case convergence rate as well as almost sure weak convergence of the iterates. 

Throughout the section we set
\begin{equation}
\label{eq:20210618b}
\delta  = \max_{i \in [m]} \left( L_{i}\gamma_i + 2\gamma_i\tau L_{\mathrm{res}}^{\bVV} \sqrt{\pp_{\max}}\right) = \max_{i \in [m]} \left( L_{i}\gamma_i + 2\gamma_i\tau L_{\mathrm{res}} \frac{\pp_{\max}}{\sqrt{\pp_{\min}}}\right),
\end{equation}
where the constants $L_i$'s and $L_{\mathrm{res}}$ are defined in Assumption~\ref{eq:A3} and the constant $L_{\mathrm{res}}^{\bVV}$ is defined in Remark~\ref{rmk:20210618a}.
The main convergence theorem is as follows.
\begin{theorem}
\label{thm:main1}
 Let $(\bx^k)_{k \in \mathbb{N}}$ be the sequence generated by Algorithm \ref{algoAsymain} and suppose that $\delta < 2$. Then the following hold.
\begin{enumerate}[label={\rm (\roman*)}]
\item\label{thm:main1_i} 
The sequence $(\bx^k)_{k\in \mathbb{N}}$ 
weakly converges $\PP$-a.s.~to a random variable that takes values in $\argmin F$.
\item\label{thm:main1_ii} 
$\EE[F(\bx^k)] - F^* = o(1/k)$. Furthermore, for every integer $k \geq 1$,
\begin{align*}
\EE[&F(\bx^k)] - F^* \leq \frac{1}{k} \left(\frac{\mathrm{dist}^2_{\bWW}
(\bx^0, \argmin F)}{2} + C\left(F(\bx^0) - F^*\right)\right),
\end{align*}
where $\displaystyle C = \frac{\max\left\{1,(2-\delta)^{-1}\right\}}{\pp_{\min}} -1
+ \tau\frac{1}{\sqrt{\pp_{\min}}(2-\delta)}
\left( 1 + \frac{\pp_{\max}}{\sqrt{\pp_{\min}}}\right)
$.
\end{enumerate}
\end{theorem}
\begin{remark}\label{rmk:20211216a}\
 \begin{enumerate}[label={\rm (\roman*)}]
 \item Theorem~\ref{thm:main1} extends classical results about the forward-backward algorithm to the asynchronous and stochastic block-coordinate setting. See \cite{salzo2021book} and reference therein. Moreover, we note that the above results, when specialized to the synchronous case, that is, $\tau=0$, yield exactly 
 \cite[Theorem~4.9]{salzo2021parallel}. The $o(1/k)$ was also proven in \cite{LeeWri19}.
 \item The almost sure weak convergence of the iterates for the asynchronous stochastic forward-backward algorithm is new.
 In general only convergence in value is provided
 or, in the nonconvex case, cluster points of the sequence of the iterates are proven to be almost surely stationary points \cite{davis2016asynchronous,cannelli2019asynchronous}.
 \item As it can be readily seen from statement \ref{thm:main1_ii} in Theorem~\ref{thm:main1}, our results depend only on the maximum possible delay, and therefore apply in the same way to the consistent and inconsistent read model.
 \item If we suppose that the random variables $(i_k)_{k \in \N}$ are uniformly distributed over $[m]$, the stepsize rule reduces to $\gamma_i< 2/(L_{i} + 2\tau L_{\mathrm{res}}/\sqrt{m})$, which agrees with that given in \cite{davis2016asynchronous} and gets better when the number of blocks $m$ increases. In this case, we see that the effect of the delay on the stepsize rule is mitigated by the number of blocks. In \cite{cannelli2019asynchronous} 
 the stepsize is not adapted to the blockwise Lipschitz constants $L_i$'s, but it is chosen for each block as $\gamma < 2/(2L_{f} + \tau^2 L_{f})$ with $L_f \geq L_{\mathrm{res}}$, leading, in general, to smaller stepsizes. In addition, this rule has a worse dependence on the delay $\tau$ and lacks of any dependence on the number of blocks.
 \item The framework of \cite{cannelli2019asynchronous}
 is nonconvex and considers more general types of algorithms, in the flavour of majorization-minimization approaches \cite{HunLan04}. On the other hand the   assumptions are stronger (in particular, they assume $F$ to be coercive) and the rate of convergence is given with respect to $\|\bx^k - \prox_{g}(\bx^k - \nabla f(\bx^k))\|^2$,
 a quantity which is hard to relate to $F(\bx^k) - F^*$. They also prove 
 that the cluster points of the sequence of the iterates
 are almost surely stationary points.
 \item \label{item:20211216a}
The work \cite{liu2015asynchronous}
was among the first to study an asynchronous 
version of the randomized coordinate 
gradient descent method. There, the
coordinates were selected at random with uniform probability and the stepsize was chosen the same for every coordinate. However, the stepsize was chosen to depend  exponentially on $\tau$, i.e as $\mathcal{O}(1/\rho^{\tau})$ with $\rho > 1$,
which is much worse than our $\mathcal{O}(1/\tau)$. The same problem affects the constant in front of the 
bound of the rate of convergence which indeed is  of the form $\mathcal{O}(\rho^{\tau})$.

To circumvent these limitations above they put a condition in Corollary $4.2$ that bounds how big the maximum delay $\tau$ can be:
\begin{equation}
\label{eq:20220902a}
 4 e \Lambda(\tau+1)^{2} \leq \sqrt{m},\quad   \Lambda = \frac{L_{\mathrm{res}}}{L_{\max }},
\end{equation}
where $m$ is the dimension of the space. However, this inequality is never satisfied if $\Lambda >\sqrt{m}/(4e)$, since this would imply
\[
(\tau+1)^{2} < 1,
\]
contradicting the fact that $\tau$ is a non-negative integer. An example where this happens is when we are dealing with a quadratic function with positive semidefinite Hessian $\mathsf{Q}\in \R^{n\times n}$. In this case
\[
L_{\mathrm{res}}=\max_{i}\|\mathsf{Q}_{\cdot i}\|_{2} \text{ and } L_{\max } =\max _{i}\left\|\mathsf{Q}_{\cdot i}\right\|_{\infty} \text{ with } \mathsf{Q}_{\cdot i} \text{ the } i\text{th column of } \mathsf{Q}.
\]
Say one column of $\mathsf{Q}$ has constant entries equal to $p > 0$, while the absolute value of all the other entries of $\mathsf{Q}$ are less than $p$. Then, 
\[
\Lambda = \frac{p\sqrt{m}}{p} = \sqrt{m}>\frac{\sqrt{m}}{4e}.
\]
In Section \ref{sect:20220825a}, we show two experiments
on real datasets for which condition \eqref{eq:20220902a} is not verified.

\end{enumerate}
\end{remark}
Before giving the proof of Theorem~\ref{thm:main1}, we present few preliminary results. The first one is a proposition showing that the function values are decreasing in expectation. The proof of this proposition, as well as those of the next intermediate results, are given in Appendix~\ref{sec:appB}.
\begin{proposition}\label{prop:20211112a}
Assume that $\delta < 2$ and let $(\bx^k)_{k \in \mathbb{N}}$ be the sequence generated by Algorithm \ref{algoAsymain}. Then, for every $k \in \N$,
\begin{align}\label{eq:20210106a}
(2-\delta) \frac{\pp_{\min}}{2}\norm{\bbar{\bx}^{k+1} - \bx^k}^2_{\bGammas^{-1}}
\leq 
 F(\bx^{k})+ \uuu_{k} -\EE\big[   F(\bx^{k+1})+ \uuu_{k+1}\,\big\vert\,i_0,\dots,i_{k-1}\big]
 \ \ \PP\text{-a.s.},
\end{align}
where $\displaystyle \uuu_k = \frac{L_{\mathrm{res}}^{\bVV}}{2\sqrt{\pp_{\max}}}
\sum_{h = k-\tau}^{k-1} (h-(k-\tau)+1)\norm{\bx^{h+1} - \bx^{h}}_{\bVV}^2$.
\end{proposition}

\begin{lemma}\label{lem:20210207} Let $(\bx^k)_{k \in \N}$ be the sequence generated by Algorithm \ref{algoAsymain}.
Then for every $k \in \N$, we have
\begin{align*}
    \langle\nabla f(\bx^k)  - \nabla f(\bhat{\bx}^k), 
    \bbar{\bx}^{k+1} - \bx^k\rangle_{\bVV}
    \leq \tau L_{\mathrm{res}}^{\bVV} \sqrt{\pp_{\max}} \sum_{i=0}^m \pp_i |\bar{x_i}^{k+1} - x_i^k|^2 + \uuu_k - \EE\big[\uuu_{k+1}\,\big\vert\,i_0,\dots,i_{k-1}\big],
\end{align*}
where $\alpha_k$ is defined in Proposition~\ref{prop:20211112a}.
\end{lemma}

The next two results extend \cite[Proposition 4.4, Proposition 4.5]{salzo2021parallel} to our more general setting.
\begin{lemma}
\label{p:20190313c}
Let $(\bx_k)_{k \in \N}$ be a sequence generated by Algorithm \ref{algoAsymain}. 
Let $k \in \N$ and let $\bx$ be an $\bHH$-valued  random variable
which is measurable w.r.t.~$i_1,\dots, i_{k-1}$. Then, 
\begin{equation}
\label{eq:20190313c}
\EE[\norm{\bx^{\iter+1} - \bx}_\bWW^2 \,\vert\, i_0,\dots, i_{k-1} ] -  \norm{\bx^\iter - \bx}^2_\bWW= 
\norm{\bar{\bx}^{\iter+1} - \bx}^2_{\bGammas^{-1}} 
 - \norm{\bx^\iter - \bx}^2_{\bGammas^{-1}}
\end{equation}
and
$\EE[\norm{\bx^{\iter+1} - \bx^\iter}_\bWW^2 \,\vert\, i_0,\dots, i_{k-1} ]
= \norm{\bar{\bx}^{\iter+1} - \bx^\iter}^2_{\bGammas^{-1}}$.
\end{lemma}

\begin{proposition}\label{prop:20210104} Let $(\bx_k)_{k \in \N}$ be a sequence generated by Algorithm \ref{algoAsymain}
and suppose that $\delta < 2$. Let $(\bbar{\bm{x}}^k)_{k \in \N}$ and $(\alpha_k)_{k \in \N}$ be defined as in \eqref{eq:20170921b} and in Proposition~\ref{prop:20211112a} respectively. Then, for every $k \in \N$, 
\begin{align*}
(\forall\, \bxs \in \bHH)\quad
\langle\bx^{k}-\bbar{\bx}^{k+1}, \bxs-\bx^{k}\rangle_{\bGammas^{-1}} &\leq \frac{1}{\pp_{\min }}
\big(F(\bx^k) + \uuu_k - \EE\big[ F(\bx^{k+1}) + \uuu_{k+1} \,\vert\, i_0,\dots, i_{k-1}\big]\big) \nonumber \\
&\qquad + F(\bxs) - F(\bx^k) 
+ \frac{\tau L_{\mathrm{res}}}{2}\sum_{h \in J(k)} \|\bx^h - \bx^{h+1}\|^2\nonumber \\
    &\qquad +\frac{\delta - 2}{2} \|\bx^k - \bm{\bar{x}}^{k+1}\|^2_{\bGammas^{-1}}.
\end{align*}
\end{proposition}
Next we state a proposition that we will use throughout the rest of this paper. It corresponds to  \cite[Proposition 4.6]{salzo2021parallel}.
\begin{proposition}\label{prop:20210519}
Let $(\bx_k)_{k \in \N}$ be a sequence generated by Algorithm \ref{algoAsymain}
and suppose that $\delta < 2$. 
Let $(\alpha_k)_{k \in \N}$ be defined as in Proposition~\ref{prop:20211112a}.
Then, for every $k \in \N$,
\begin{align}
(\forall\, \bxs \in \bHH)\quad    
\EE\big[&\norm{\bx^{k+1}-\bxs}^{2}_{\bWW}\,\vert\, i_0,\dots, i_{k-1}\big] \nonumber \\[1ex]
&\leq \norm{\bx^{k}-\bxs}^{2}_{\bWW} \nonumber \\
&\quad + \frac{2}{\pp_{\min }}\left(\frac{(\delta-1)_+}{2-\delta}+ 1\right)
\big( F(\bx^k)
+\uuu_k - \EE\big[F(\bx^{k+1})+\uuu_{k+1}\,\vert\, i_0,\dots, i_{k-1}\big] \big) \nonumber \\
&\quad + \tau L_{\mathrm{res}} \sum_{h \in J(k)} \|\bx^h - \bx^{h+1}\|^2 \nonumber\\
&\quad + 2(F(\bxs) - F(\bx^k)).
\end{align}
\end{proposition}

In the following, we show a general inequality from which we derive simultaneously the convergence 
of the iterates and the rate of convergence in expectation of the function values.
\begin{proposition}
\label{prop:20200126a}
Let $(\bx_k)_{k \in \N}$ be a sequence generated by Algorithm \ref{algoAsymain}
and suppose that $\delta < 2$. 
Let $(\alpha_k)_{k \in \N}$ be defined as in Proposition~\ref{prop:20211112a}.
Then, for all $\bxs \in \bHH$,
\begin{equation*}
\EE\big[\|\bx^{k+1}-\bxs\|^{2}_{\bWW}\,\vert\,
i_0,\dots, i_{k-1}\big] 
\leq \|\bx^{k}-\bxs\|^{2}_{\bWW} +
2\big( F(\bxs) -\EE\big[F(\bx^{k+1})+\uuu_{k+1}\,\vert\, i_0,\dots, i_{k-1}\big]\big)
  + \xi_k,
\end{equation*}
where $(\xi^k)_{k \in \mathbb{N}}$ is a  sequence of positive random variables
such that
\begin{equation}
\label{eq:20210620a}
\sum_{k \in \N}\EE[\xi_k] \leq 2 C(F(\bx^0)  - F^*),
\end{equation}
with $\displaystyle C = \frac{\max\left\{1,(2-\delta)^{-1}\right\}}{\pp_{\min}} -1
+ \frac{\tau}{\sqrt{\pp_{\min}}(2-\delta)}
\left( 1 + \frac{\pp_{\max}}{\sqrt{\pp_{\min}}}\right)$.
\end{proposition}

\begin{proposition}
\label{prop:20200126b}
Let $(\bx_k)_{k \in \N}$ be a sequence generated by Algorithm \ref{algoAsymain}
and suppose that $\delta < 2$. 
Let $(\bbar{\bm{x}}^k)_{k \in \N}$ be defined as in \eqref{eq:20170921b}. Then there exists a sequence of $\bHH$-valued random variables $(\bm{v}^k)_{k \in \N}$
such that the following assertions hold:
\begin{enumerate}[label={\rm (\roman*)}]
\item\label{prop:20200126b_i}  $\forall\, k \in \N\colon$ $\bm{v}^k \in \partial F(\bbar{\bm{x}}^{k+1})$ $\PP$-a.s.
\item\label{prop:20200126b_ii} $\bm{v}^k \to 0$ and $\bm{x}^{k} - \bbar{\bm{x}}^{k+1} \to 0$ $\PP$-a.s., as $k \to +\infty$.
\end{enumerate}
\end{proposition}

We are now ready to prove the main theorem.

\begin{proof}[Proof of Theorem \ref{thm:main1}]
\ref{thm:main1_i}: It follows from
Proposition \ref{prop:20200126a} that
\begin{equation*}
(\forall\, \bxs \in \argmin F)\quad
\EE\big[\|\bx^{k+1}-\bxs\|^{2}_{\bWW}\,\vert\,
i_0,\dots, i_{k-1}\big] 
\leq \|\bx^{k}-\bxs\|^{2}_{\bWW} + \xi_k,
\end{equation*}
where $(\xi_k)_{k \in \N}$ is a sequence of positive random variable which is $\PP$-a.s.~summable.
Thus, the sequence $(\bx^k)_{k\in \mathbb{N}}$ is stochastic quasi-Fej\'er with respect to $\argmin F$ in the norm $\norm{\cdot}_{\bWW}$ (which is equivalent to $\norm{\cdot}$). Then according to Fact~\ref{prop:stochfej} it is bounded $\PP$-a.s. We now prove that $\argmin F$ contains the weak cluster points of $(\bx^k)_{k\in \mathbb{N}}$ $\PP$-a.s. Indeed,
let $\Omega_1 \subset \Omega$ with $\PP(\Omega \setminus \Omega_1) = 0$ be such that
items \ref{prop:20200126b_i} and \ref{prop:20200126b_ii} of Proposition~\ref{prop:20200126b} hold.
Let $\omega \in \Omega_1$ and
let $\bxs$ be a weak cluster point of $(\bx^{k}(\omega))_{k \in \mathbb{N}}$. There exists a subsequence $(\bx^{k_q}(\omega))_{q \in \mathbb{N}}$ which weakly converges to $\bxs$. 
By Proposition~\ref{prop:20200126b}, we have
$\bbar{\bx}^{k_q+1}(\omega) \rightharpoonup \bxs$,
 $\bm{v}^{k_q+1}(\omega) \to 0$, and 
$\bm{v}^{k_q+1}(\omega) \in \partial (f+g)(\bbar{\bm{x}}^{k_q+1}(\omega))$.
Thus, \cite[Proposition~1.6 (demiclosedness of the graph of the subgradient)]{marcellin2006evolution} yields $0 \in \partial F(\bxs)$ and hence $\bxs \in \argmin F$.
Therefore, again by Fact~\ref{prop:stochfej} we conclude that the sequence 
$(\bx^k)_{k \in \N}$ weakly converges to a random variable that takes value in $\argmin F$ $\PP$-a.s.

\ref{thm:main1_ii}:
Choose $\bxs \in \argmin F$ in 
Proposition~\ref{prop:20200126a} and then take the expectation. Then we get
\begin{equation*}
\EE[F(\bx^{k+1})+\uuu_{k+1}]- F^* \leq \frac{1}{2} \big(
 \EE[\norm{\bx^{k}-\bxs}^{2}_{\bWW}] - \EE[\norm{\bx^{k+1}-\bxs}^{2}_{\bWW}] \big)
  + \frac 1 2 \EE[\xi_k].
\end{equation*}
Since $\sum_{k \in \N}(\EE[\norm{\bx^{k}-\bxs}^{2}_{\bWW}] - \EE[\norm{\bx^{k+1}-\bxs}^{2}_{\bWW}]) \leq \norm{\bx^{0}-\bxs}^{2}_{\bWW}$, and recalling the bound on $\sum_{k \in \N}\EE[\xi_k]$ in \eqref{eq:20210620a}, we have
\begin{equation*}
\sum_{k \in \N}\big(\EE[F(\bx^{k+1})+\uuu_{k+1}]- F^*\big) 
\leq \frac{\norm{\bx^{0}-\bxs}^{2}_{\bWW}}{2}
+C(F(\bx^0)  - F^*).
\end{equation*}
Thus, since, in virtue of equation~\ref{eq:20210106a}, 
$(\EE[F(\bx^{k+1})+\uuu_{k+1}]- F^*)_{k \in \N}$
is decreasing, 
the statement follows from Fact~\ref{fact:convmonseq},
considering that $\alpha_k \geq 0$.
\end{proof}
\section{Linear convergence under error bound condition}
\label{sec:errorbound}
In the previous section we get a sublinear rate of convergence. Here we show that 
with an additional assumption we can get a better convergence rate. Also, we derive a strong convergence of the iterates, improving the weak convergence proved in Theorem~\ref{thm:main1}.

We will assume that the following \emph{Luo-Tseng error bound condition} \cite{luo1993error} holds on a subset 
$\bXX \subset \bHH$ (containing the iterates $\bx^k$).
\begin{align}\label{eb}
(\forall \bxs \in \bXX) \quad \operatorname{dist}_{\bGammas^{-1}}\left(\bxs, \argmin F\right) \leq C_{\bXX, \bGammas^{-1}} \big\|\bxs-\prox_{g}^{\bGammas^{-1}}\big(\bxs-\nabla^{\bGammas^{-1}} f(\bxs)\big)\big\|_{\bGammas^{-1}}.
\end{align}
\begin{remark}
 We recall that the condition above is equivalent to the  Kurdyka-Lojasiewicz property and 
 the quadratic growth condition \cite{drusvyatskiy2018error, bolte2017error, salzo2021parallel}. Any of these conditions can be used to prove linear convergence rates for various algorithms.
\end{remark}
The following theorem is the main result of this section. Here, linear convergence of the function values and strong convergence of the iterates are ensured.
 

\begin{theorem}
\label{thm:main3} 
Let $(\bx^k)_{k \in \N}$ be generated by Algorithm \ref{algoAsymain} 
and suppose $\delta<2$ and that the error bound condition \eqref{eb} holds with $\bXX \supset 
\{\bx^k\,\vert\,k \in \N\}$ $\PP$-a.s.~for some $C_{\bXX, \bGammas^{-1}} > 0$.
Then for all $k \in \mathbb{N}$,
\begin{enumerate}[label={\rm (\roman*)}]
    \item \label{thm:scnd_i}
        $\displaystyle \EE\big[F(\bx^{k+1})-F^*\big] 
        \leq \left( 1 - \frac{\pp_{\min }}{\kappa+\theta}\right)^{\lfloor \frac{k+1}{\tau + 1} \rfloor} \EE\big[F(\bx^{0})-F^*\big],$
      
    where 
    \begin{align*}
    \kappa &= 1 + \frac{(2C_{\bXX, \bGammas^{-1}} + \delta-2)_+}{2-\delta}=\max\left\{1,\frac{2C_{\bXX, \bGammas^{-1}}}{2-\delta}\right\}\\
    \theta &= \frac{\tau L_{\mathrm{res}}\gamma_{\max}}{2-\delta} \left(\frac{\pp^2_{\max}}{\sqrt{\pp_{\min}}} + 1 \right) \leq  \frac{\sqrt{\pp_{\min}}}{\pp_{\max}(2-\delta)} \left(\frac{\pp^2_{\max}}{\sqrt{\pp_{\min}}} + 1 \right).
    \end{align*}
    \item \label{thm:scnd_ii} The sequence $(\bx^k)_{k\in \mathbb{N}}$
converges strongly $\PP$-a.s.~to a random variable $\bx^*$ that takes values in $\argmin F$ and  $ \EE\big[\norm{\bx^{k} - \bx^*}_{\bGammas^{-1}}\big] = \mathcal{O}\big(\big( 1 - \pp_{\min}/(\kappa+\theta)\big)^ {\lfloor \frac{k}{\tau + 1} \rfloor/2}\big)$.
\end{enumerate}
\end{theorem}
\begin{proof}
\ref{thm:scnd_i}: From Proposition~\ref{prop:20210104} we have
\begin{align*}
\frac{1}{\pp_{\min }}&\EE\big[F(\bx^{k+1}) 
+ \uuu_{k+1} - F(\bx^k) - \uuu_k\,\vert\, i_0,\dots, i_{k-1}\big] \nonumber\\
&\leq \|\bx^{k}-\bbar{\bx}^{k+1}\|_{\bGammas^{-1}}\|\bx^{k}-\bxs\|_{\bGammas^{-1}} \nonumber \\
&\qquad + F(\bxs) - F(\bx^k) + \frac{\tau L_{\mathrm{res}}}{2}
\sum_{h \in J(k)} \|\bx^h - \bx^{h+1}\|^2 \nonumber \\
&\qquad +\frac{\delta -2}{2}\|\bx^k - \bm{\bar{x}}^{k+1}\|^2_{\bGammas^{-1}},
\end{align*}
where $\uuu_k = (L_{\mathrm{res}}/(2\sqrt{\pp_{\min}}))
\sum_{h = k-\tau}^{k-1} (h-(k-\tau)+1)\norm{\bx^{h+1} - \bx^{h}}_{\bVV}^2$.
Now, taking $\bxs \in \argmin F$ and using the  error bound condition \ref{eb} 
and equation \ref{eq:20210106a}, we obtain
\begin{align}
\frac{1}{\pp_{\min }}\EE&\big[F(\bx^{k+1})+\uuu_{k+1}-F(\bx^k)) - \uuu_k\,\vert\, i_0,\dots, i_{k-1}\big]
\nonumber\\
&\leq \left(C_{\bXX, \bGammas^{-1}} 
+ \frac{\delta -2}{2}\right) \|\bx^k - \bbar{\bx}^{k+1}\|^2_{\bGammas^{-1}} \nonumber \\
&\qquad - (F(\bx^k) - F^*) + \frac{\tau L_{\mathrm{res}}}{2}
\sum_{h = k-\tau}^{k-1} \|\bx^h - \bx^{h+1}\|^2 \nonumber \\
&\leq  \frac{(2C_{\bXX, \bGammas^{-1}} + \delta -2)_+}{(2-\delta)\pp_{\min}} 
 \EE\big[F(\bx^k) + \uuu_k - F(\bx^{k+1}) - \uuu_{k+1}\,\vert\, i_0,\dots, i_{k-1}\big] 
\nonumber \\
&\qquad - (F(\bx^k) - F^*) + \frac{\tau L_{\mathrm{res}}}{2}
\sum_{h = k-\tau}^{k-1} \|\bx^h - \bbar{\bx}^{h+1}\|^2,
\end{align}
Adding and removing $F^*$ in both expectation yield
\begin{align}
\label{eq:20210119d}
\kappa\EE\big[F(\bx^{k+1})+\uuu_{k+1}-F^*\,\vert\, i_0,\dots, i_{k-1}\big] 
&\leq \kappa \EE\big[F(\bx^{k})+\uuu_{k}-F^*\,\vert\, i_0,\dots, i_{k-1}\big] \nonumber\\
&\qquad + \frac{\tau L_{\mathrm{res}}\gamma_{\max} \pp_{\min}}{2}
\sum_{h = k-\tau}^{k-1} \|\bx^h - \bbar{\bx}^{h+1}\|^2_{\bGammas^{-1}} \nonumber \\
&\qquad - \pp_{\min } (F(\bx^k) + \uuu_k - F^*) + \pp_{\min }\uuu_k,
\end{align}
where $\kappa = 1 + (2C_{\bXX, \bGammas^{-1}} + \delta -2 )_+/(2-\delta)$.
Now, since $\norm{\cdot}^2_{\bVV} \leq \gamma_{\max}\pp_{\max}^2 \norm{\cdot}^2_{\bWW}$
we have
\begin{align}
\label{eq:20210621a}
\EE[\uuu_k] &\leq \frac{\tau L_{\mathrm{res}}\gamma_{\max}\pp_{\max}^2}{2 \sqrt{\pp_{\min}}}
\sum_{h = k-\tau}^{k-1}\EE[\norm{\bx^{h+1} - \bx^{h}}^2_{\bWW}] \nonumber \\
&= \frac{\tau L_{\mathrm{res}}\gamma_{\max}\pp_{\max}^2}{2 \sqrt{\pp_{\min}}}
\sum_{h = k-\tau}^{k-1}\EE[\norm{\bbar{\bx}^{h+1} - \bx^{h}}^2_{\bGammas^{-1}}],
\end{align}
where in the last equality we used Lemma~\ref{p:20190313c}.
From \eqref{eq:20210106a}, we have, for $k$ such that $k-\tau \geq 0$,
\begin{align}
\label{eq:20210121a}
\sum_{h = k-\tau}^{k-1}\EE[\norm{\bbar{\bx}^{h+1} - \bx^{h}}^2_{\bGammas^{-1}}] 
&\leq 
\frac{2}{(2-\delta)\pp_{\min}}
\sum_{h = k-\tau}^{k-1} \EE\big[F(\bx^h)+\uuu_h\big] 
- \EE\big[F(\bx^{h+1})+\uuu_{h+1}\big] \nonumber \\
&= 
\frac{2}{(2-\delta)\pp_{\min}}
\left( \EE\big[F(\bx^{k - \tau})+\uuu_{k - \tau}\big] -  \EE\big[F(\bx^{k})+\uuu_{k}\big]\right) \nonumber\\
&\leq 
\frac{2}{(2-\delta)\pp_{\min}}
\left( \EE\big[F(\bx^{k - \tau})+\uuu_{k-\tau}\big] 
- \EE\big[F(\bx^{k+1})+\uuu_{k+1}\big]\right)\nonumber \\
&= 
\frac{2}{(2-\delta)\pp_{\min}}
\left( \EE\big[F(\bx^{k - \tau})+\uuu_{k-\tau}-F^*\big] 
-  \EE\big[F(\bx^{k+1})+\uuu_{k+1} - F^*\big]\right).
\end{align}
Because the sequence $\left(\EE\big[F(\bx^{k})+\uuu_{k}\big]\right)_{k \in \N}$ is decreasing, 
the transition from the second line to the third one is allowed. Using \eqref{eq:20210621a} and
\eqref{eq:20210121a}   in \eqref{eq:20210119d} with total expectation, we obtain
\begin{align}
\label{eq:20210119b}
(\kappa+\theta)\EE\big[F(\bx^{k+1})+\uuu_{k+1}-F^*\big] &\leq (\kappa - \pp_{\min })
\EE\big[F(\bx^{k})+\uuu_{k}-F^*\big] \nonumber \\
&\qquad + \theta \EE\big[F(\bx^{k - \tau})+\uuu_{k-\tau}-F^*\big] \nonumber \\
&\leq (\kappa - \pp_{\min }) \EE\big[F(\bx^{k-\tau})+\uuu_{k-\tau}-F^*\big] \nonumber \\
&\qquad + \theta \EE\big[F(\bx^{k - \tau})+\uuu_{k-\tau}-F^*\big] \nonumber \\
&=  (\kappa + \theta - \pp_{\min }) \EE\big[F(\bx^{k-\tau})+\uuu_{k-\tau}-F^*\big],
\end{align}
where $\displaystyle \theta = (2-\delta)^{-1} \left(\frac{\tau L_{\mathrm{res}}\gamma_{\max}\pp_{\max}^2}{\sqrt{\pp_{\min}}}+ \tau L_{\mathrm{res}}\gamma_{\max} \right) 
= \tau L_{\mathrm{res}}\gamma_{\max}(2-\delta)^{-1} \left(\frac{\pp^2_{\max}}{\sqrt{\pp_{\min}}} 
+ 1 \right)$. 
That means
\begin{align}\label{eq:20210526a}
\EE\big[F(\bx^{k+1})+\uuu_{k+1}-F^*\big] &\leq \left( 1 - \frac{\pp_{\min }}{\kappa+\theta}\right)
\EE\big[F(\bx^{k-\tau})+\uuu_{k-\tau}-F^*\big] \nonumber \\
&\leq \left( 1 - \frac{\pp_{\min }}{\kappa+\theta}\right)^{\lfloor \frac{k+1}{\tau + 1} \rfloor} 
\EE\big[F(\bx^{0})+\uuu_{0}-F^*\big].
\end{align}
Now for $k < \tau$, $\lfloor \frac{k+1}{\tau + 1} \rfloor = 0$. Because $\left(\EE\big[F(\bx^{k})
+\uuu_{k}\big]\right)_{k \in \N}$ is decreasing, we know that 
\begin{align*}
\EE\big[F(\bx^{k+1})+\uuu_{k+1}-F^*\big] &\leq \EE\big[F(\bx^{0})+\uuu_{0}-F^*\big] \\
&= \left( 1 - \frac{\pp_{\min }}{\kappa+\theta}\right)^{\lfloor \frac{k+1}{\tau + 1} \rfloor} 
\EE\big[F(\bx^{0})+\uuu_{0}-F^*\big].
\end{align*}
So \eqref{eq:20210526a} remains true.
Also from \eqref{eq:20211209a}, we have
\begin{equation*}
    \theta \leq  \frac{\sqrt{\pp_{\min}}}{\pp_{\max}}(2-\delta)^{-1} \left(\frac{\pp^2_{\max}}{\sqrt{\pp_{\min}}} 
+ 1 \right).
\end{equation*}

\ref{thm:scnd_ii}: From Jensen inequality, \eqref{eq:20210106a} and \eqref{eq:20210526a}, we have
\begin{align}
    \EE\big[\norm{\bx^{k+1} - \bx^k}_{\bGammas^{-1}}\big] &\leq \sqrt{\EE\big[\norm{\bx^{k+1} - \bx^k}^2_{\bGammas^{-1}}\big]} \nonumber \\
    & \leq \sqrt{\EE\big[\norm{\bbar{\bx}^{k+1} - \bx^k}^2_{\bGammas^{-1}}\big]} \nonumber \\
     &\leq 
      \sqrt{\frac{2}{\pp_{\min}(2-\delta)} \EE\big[ F(\bx^{k})+ \uuu_{k} - F^*\big]} \nonumber\\
      &\label{eq:20220111a} \leq \sqrt{\frac{2}{\pp_{\min}(2-\delta)}\left( 1 - \frac{\pp_{\min }}{\kappa+\theta}\right)^{\lfloor \frac{k}{\tau + 1} \rfloor} 
\EE\big[F(\bx^{0})+\uuu_{0}-F^*\big]}.
\end{align}
Since $ 1 - \pp_{\min }/(\kappa+\theta) < 1$,
\begin{equation*}
    \EE\bigg[\sum_{k \in \N} \norm{\bx^{k+1} - \bx^k}_{\bGammas^{-1}}\bigg] = \sum_{k \in \N} \EE\big[\norm{\bx^{k+1} - \bx^k}_{\bGammas^{-1}}\big] < \infty.
\end{equation*}
Therefore $\sum_{k \in \N} \norm{\bx^{k+1} - \bx^k}_{\bGammas^{-1}} < \infty$ $\PP$-a.s. This means the sequence $(\bx^k)_{k\in \mathbb{N}}$ is a Cauchy sequence $\PP$-a.s.~By Theorem \ref{thm:main1} \ref{thm:main1_i}, this sequence has accumulation points that take values in $\argmin F$. So it converges strongly $\PP$-a.s.~to a random variable that takes values in $\argmin F$.

Now let $\rho =  1 - \pp_{\min }/(\kappa+\theta)$. For all $n \in \N$,
\begin{equation*}
    \norm{\bx^{k+n} - \bx^{k}}_{\bGammas^{-1}} \leq \sum_{i=0}^{n-1} \norm{\bx^{k+i+1} - \bx^{k+i}}_{\bGammas^{-1}} \leq \sum_{i=0}^{\infty} \norm{\bx^{k+i+1} - \bx^{k+i}}_{\bGammas^{-1}}.
\end{equation*}
Letting $n \rightarrow \infty$ and using \eqref{eq:20220111a}, we get
\begin{align*}
    \EE\big[\norm{\bx^{k} - \bx^*}_{\bGammas^{-1}}\big] &\leq \left(\frac{2}{\pp_{\min}(2-\delta)} 
\EE\big[F(\bx^{0})+\uuu_{0}-F^*\big]\right)^{1/2} \sum_{i=0}^{\infty} \rho^{\lfloor \frac{k+i}{\tau + 1} \rfloor/2} \\
    &\leq \left(\frac{2}{\pp_{\min}(2-\delta)} 
\EE\big[F(\bx^{0})+\uuu_{0}-F^*\big]\right)^{1/2} \rho^{\lfloor \frac{k}{\tau + 1} \rfloor/2} \sum_{i=0}^{\infty} \rho^{\lfloor \frac{i}{\tau + 1}\rfloor/2} \\
    &= \rho^{\lfloor \frac{k}{\tau + 1} \rfloor/2}\left(\frac{2}{\pp_{\min}(2-\delta)} 
\EE\big[F(\bx^{0})+\uuu_{0}-F^*\big]\right)^{1/2} \frac{\tau  + 1 }{1 - \rho^{1/2}}.
\qedhere
\end{align*}
\end{proof}

\newpage
\begin{remark}\
\begin{enumerate}[label={\rm (\roman*)}]
\item A linear convergence rate is also given in \cite[Theorem 4.1]{liu2015asynchronous} by assuming a quadratic growth condition instead of the error bound condition \eqref{eb}. Their rate depend on the stepsize
which in general can be very small, as explained earlier in point \ref{item:20211216a} of Remark \ref{rmk:20211216a}.
\item The error bound condition \eqref{eb} is sometimes satisfied globally, meaning 
on $\bXX = \dom F$, so that the condition $\bXX \supset \{\bx^k\,\vert\,k \in \N\}$ $\PP$-a.s.~required in Theorem~\ref{thm:main3} is clearly fulfilled. This is the case when
 $F$ is strongly convex or when $f$ is quadratic and $g$ is the indicator function of a polytope (see Remark~4.17(iv) in \cite{salzo2021parallel}).
 More often, for general convex objectives, 
 the error bound condition \eqref{eb} is satisfied on sublevel sets of $F$ (see \cite[Remark~4.18]{salzo2021parallel}).
 Therefore, it is important to find conditions ensuring that the sequence $(\bx^k)_{k \in \N}$
 remains in a sublevel set.
 The next results address this issue.
\end{enumerate}
\end{remark}

We first give an analogue of Lemma~\ref{lem:20210207}.

\begin{lemma}\label{lem:20210623a}
Let $(\bx^k)_{k \in \mathbb{N}}$ be the sequence generated by Algorithm \ref{algoAsymain}. Then, for every $k \in \N$,
\begin{equation*}
    \langle\nabla f(\bx^k)  - \nabla f(\bhat{\bx}^k), 
    {x}^{k+1} - x^k\rangle
    \leq \tau L_{\mathrm{res}}  \|{\bx}^{k+1} - \bx^k\|^2 + \tilde{\uuu}_k - \tilde{\uuu}_{k+1},
\end{equation*}
with $\tilde{\uuu}_k = (L_{\mathrm{res}}/2)\sum_{h = k-\tau}^{k-1} (h-(k-\tau)+1)\|\bx^{h+1} - \bx^{h}\|^2$.
\end{lemma}
\begin{proof}
Let $k \in \N$.
We have, from Cauchy-Schwarz inequality, the
Young inequality and Lemma~\ref{eq:20170925a}, that
\begin{align*}
\langle\nabla f(\bx^k)&-\nabla f(\bhat{\bx}^k),{\bx}^{k+1}-\bx^k \rangle \\ 
&\leq L_{\mathrm{res}} \sum_{h \in J(k)} \|\bx^{h+1} - \bx^{h}\| \|{\bx}^{k+1} - \bx^k\| \\
&\leq \frac{1}{2}\left[\frac{L_{\mathrm{res}}^2}{s}\bigg(\sum_{h \in J(k)}
\|\bx^{h+1} - \bx^{h}\|\bigg)^2 
+s\| {\bx}^{k+1} - \bx^k\|^2\right] \\
&\leq \frac{1}{2}\left[\frac{\tau L_{\mathrm{res}}^2}{s}\left(\sum_{h = k-\tau}^{k-1}
\|\bx^{h+1} - \bx^{h}\|^2\right)
+s\| {\bx}^{k+1} - \bx^k\|^2\right] \\
&= \frac{s}{2}\|{\bx}^{k+1} - \bx^k\|^2 
+ \frac{\tau L_{\mathrm{res}}^2}{2s}\sum_{h = k-\tau}^{k-1} 
\|\bx^{h+1} - \bx^{h}\|^2.
\end{align*}
Using the same decomposition of the last term as in Lemma \ref{lem:20210207} , we get 
\begin{align*}
\langle\nabla f(\bx^k)& - \nabla f(\bhat{\bx}^k), {\bx}^{k+1} - \bx^k\rangle\\
&\leq \frac{s}{2}\|{\bx}^{k+1} - \bx^k\|^2 
+ \frac{\tau L_{\mathrm{res}}^2}{2s} \sum_{h = k-\tau}^{k-1} 
(h- (k-\tau)+1)\|\bx^{h+1} - \bx^{h}\|^2 \\
&\qquad - \frac{\tau L_{\mathrm{res}}^2}{2s} \sum_{h=k-\tau+1}^{k}
(h-(k-\tau))\|\bx^{h+1} - \bx^{h}\|^2 \\
&\qquad + \frac{\tau^2 L_{\mathrm{res}}^2}{2s}
\|\bx^{k+1} - \bx^{k}\|^2.
\end{align*}
So taking 
$$\displaystyle \tilde{\uuu}_k = \frac{\tau L_{\mathrm{res}}^2}{2s}\sum_{h = k-\tau}^{k-1} (h-(k-\tau)+1)\|\bx^{h+1} - \bx^{h}\|^2,$$ 
we get
\begin{equation*}
    \langle\nabla f(\bx^k)  - \nabla f(\bhat{\bx}^k), 
    \bar{x}^{k+1} - x^k\rangle 
    \leq \left(\frac{s}{2} + \frac{\tau^2 L_{\mathrm{res}}^2}{2s} \right) \|\bar{\bx}^{k+1} - \bx^k\|^2 + \tilde{\uuu}_k - \tilde{\uuu}_{k+1}.
\end{equation*}
By minimizing $s \mapsto (s/2 + \tau^2 L_{\mathrm{res}}^2/(2s))$, we find $s=\tau L_{\mathrm{res}}$. We then obtain
\begin{equation*}
    \langle\nabla f(\bx^k)  - \nabla f(\bhat{\bx}^k), 
    {x}^{k+1} - x^k\rangle 
    \leq \tau L_{\mathrm{res}}  \|{\bx}^{k+1} - \bx^k\|^2 + \tilde{\uuu}_k - \tilde{\uuu}_{k+1},
\end{equation*}
and the statement follows.
\end{proof}

\begin{proposition}\label{prop:20210623a}
Let $(\bx^k)_{k \in \mathbb{N}}$ be the sequence generated by Algorithm \ref{algoAsymain}. Then, for every $k \in \N$,
\begin{equation}
\label{eq:20210626a}
\bigg(\frac{1}{\gamma_{i_{k}}}-\frac{L_{i_{k}}}{2} 
-\tau L_{\mathrm{res}}\bigg)\|\bx^{k+1} - \bx^k\|^2 
\leq F(\bx^k)+ \tilde{\uuu}_k - \big(F(\bx^{k+1}) +\tilde{\uuu}_{k+1} \big)
 \qquad \PP\text{-a.s.},
\end{equation}
where  
$\tilde{\uuu}_k = (L_{\mathrm{res}}/2)\sum_{h = k-\tau}^{k-1} (h-(k-\tau)+1)\|\bx^{h+1} - \bx^{h}\|^2$.
\end{proposition}
\begin{proof}
Using Lemma \ref{lem:20210623a} in equation \eqref{eq:20210623a}, we have
\begin{align*}
F(\bx^{k+1})&\leq F(\bx^k) + \langle \nabla_{i_{k}} f(\bx^k) - \nabla_{i_{k}} f(\bhat{\bx}^k), \bar{x}^{k+1}_{i_{k}} - x^k_{i_{k}} \rangle - \bigg(\frac{1}{\gamma_{i_{k}}}-\frac{L_{i_{k}}}{2}\bigg)|\bar{x}^{k+1}_{i_{k}} - x^k_{i_{k}}|^2 \\
&= F(\bx^k) + \langle \nabla f(\bx^k) - \nabla f(\bhat{\bx}^k), \bx^{k+1} - \bx^k \rangle 
- \bigg(\frac{1}{\gamma_{i_{k}}}-\frac{L_{i_{k}}}{2}\bigg)\|\bx^{k+1} - \bx^k\|^2 \\
&\leq F(\bx^k) + \tilde{\uuu}_k - \tilde{\uuu}_{k+1}
- \bigg(\frac{1}{\gamma_{i_{k}}}-\frac{L_{i_{k}}}{2} - \tau L_{\mathrm{res}}\bigg)\|\bx^{k+1} - \bx^k\|^2.
\end{align*}
So the statement follows.
\end{proof}

\begin{corollary}\label{cor:20210907a}
Let $(\bx^k)_{k \in \N}$ be generated by Algorithm~\ref{algoAsymain}
with the $\gamma_i$'s satisfying the following stepsize rule
\begin{equation}
\label{eq:stepsize2}
(\forall\, i \in [m])\quad
\gamma_i < \frac{2}{L_{i} + 2\tau L_{\mathrm{res}}}.
\end{equation}
Then
\begin{equation}
(\forall\, k \in \N)\quad    F(\bx^k) \leq F(\bxx^0) \quad \PP\text{-a.s.}
\end{equation}
So if the error bound condition \eqref{eb} holds on the sublevel set $\bXX = \{F \leq F(\bxx^0)\}$, then 
the assumptions of Theorem~\ref{thm:main3} are met.
\end{corollary}
\begin{proof}
The left hand side in \eqref{eq:20210626a} is positive and hence
$(F(\bx_k) + \tilde{\uuu}_k)_{k \in \N}$ is decreasing $\PP$-a.s. 
Therefore, we have, for every $k \in \N$ 
\begin{equation*}
F(\bx^{k}) \leq F(\bx^{k}) +\tilde{\uuu}_{k} \leq F(\bx^0) + \tilde{\uuu}_0 = F(\bxx^0).
\qedhere
\end{equation*}
\end{proof}

\begin{remark}
 The rule \eqref{eq:stepsize2} yields stepsizes possibly smaller
than the ones given in Theorem~\ref{thm:main1}, which requires 
$\gamma_i< 2/(L_{i} + 2\tau L_{\mathrm{res}}\pp_{\max}/\sqrt{\pp_{\min}})$. 
Indeed this happens when $\pp_{\max}/\sqrt{\pp_{\min}} < 1$. For instance if the
distribution is uniform, we have $\pp_{\max}/\sqrt{\pp_{\min}} 
= 1/\sqrt{m} < 1$ whenever $m \geq 2$. On the bright side, there may 
exist distributions for which $\pp_{\max}/\sqrt{\pp_{\min}} > 1$.
\end{remark}

\section{Applications}
\label{sec:applications}
Here we present two problems where Algorithm \ref{algoAsymain} can be useful.

\subsection{The Lasso problem}\label{sec:lasso}
We start with the Lasso problem \cite{tibshirani1996regression}, also known as basis pursuit
\cite{chen2001atomic}. It is a least-squares regression problem with an $\ell_1$ regularizer 
which favors sparse solutions. More precisely, given $\As \in \R^{n \times m}$ and 
$\bbs \in \R^n$, one aims at solving the following problem
\begin{equation}\label{app:lasso}
\underset{\bxx \in \R^m}{\text{minimize }} \frac{1}{2}\|\As\bxx - \bbs\|_2^2 
+ \lambda \|\bxx\|_1 \qquad \left(\lambda > 0\right).
\end{equation}
We clearly fall in the framework of problem \eqref{mainprob} with 
$f(\bxx) = (1/2)\norm{\As\bxx - \bbs}_2^2$ and $g_i(\xx_i) = \lambda|\xx_i|$. The assumptions
\ref{eq:A1}, \ref{eq:A2}, \ref{eq:A3} and \ref{eq:A4} are also satisfied. In particular, here 
$L_i = \norm{a_i}^2$, where $a_i$ is the $i$-th column of $\As$, $L_{\mathrm{res}} = \max_{i}\|\As^{\intercal}\As_{\cdot i}\|_{2}$, with $\As^{\intercal}\As_{\cdot i}$ the $i$-th column of $\As^{\intercal}\As$,
and $F = f + g$ attains its minimum.

The Lasso technique is used in many fields, especially for high-dimensional problems -- 
among others it is worth mentioning statistics, signal processing, and inverse problems; see \cite{beck2009fast, tropp2006just, kim2007interior, 10.1214/14-AOS1204, 1614066, sun2013sparse} and references therein. 
Since there is no closed form solution for this problem, many iterative algorithms have 
been proposed to solve it: forward-backward, accelerated (proximal) gradient descent, (proximal)
block coordinate descent, etc. \cite{combettes2005signal,beck2009fast, nesterov2013gradient,fu1998penalized, tseng2001convergence, friedman2010regularization}.  In the same vein, 
applying Algorithm \ref{algoAsymain} to the Lasso problem \eqref{app:lasso} yields the 
iterative scheme:

\vspace{-2.5ex}
\begin{equation}
\label{eq:algoPRCD_lasso}
\begin{array}{l}
\text{for}\;n=0,1,\ldots\\
\left\lfloor
\begin{array}{l}
\text{for}\;i=1,\dots, m\\[0.7ex]
\left\lfloor
\begin{array}{l}
x^{k+1}_i = 
\begin{cases}
\mathsf{\mathsf{soft}}_{\lambda \gamma_{i_{k}}} \big(x^k_{i_{k}} - \gamma_{i_{k}}  a_{i_k}^\intercal( \As\bx^{k-\bds^k} - \bbs) \big) &\text{if } i=i_{k}\\
x^k_i &\text{if } i \neq i_{k},
\end{cases}
\end{array}
\right.
\end{array}
\right.
\end{array}
\end{equation}
where, for every $\rho>0$, $\mathsf{soft}_{\rho}\colon \R \to \R$ is the soft thresholding operator (with threshold $\rho$) \cite{salzo2021book}.
Thanks to Theorem \ref{thm:main1} we know that the iterates $(\bx^k)_{k \in \N}$ generated are weakly convergent and the function values have a convergence rate of $o(1/k)$.
On top of that the cost function of the Lasso problem \eqref{app:lasso} satisfies the error bound condition \eqref{eb} on its sublevel sets  \cite[Theorem 2]{tseng2010approximation}. So, following Corollary \ref{cor:20210907a} and Theorem \ref{thm:main3}, the iterates converge strongly (a.s.) and linearly in mean, whenever $\gamma_i < 2/\left(L_{i} + 2\tau L_{\mathrm{res}}\right)$, for all $i \in [m]$.
\subsection{Linear convergence of dual proximal gradient method }
We consider the problem
\begin{equation}\label{prob:20210909a}
\underset{\xx \in \HH}{\text{minimize }}\sum_{i=1}^{m} \phi_{i}\left(\As_i\xx\right) + h(\xx),
\end{equation}
where, for all $i \in [m], \As_i\colon \HH \to \GG_i$ is a linear operator between Hilbert spaces,
$\phi_{i}\colon \GG_i \to\left]-\infty,+\infty\right]$ is proper convex and lower semicontinuous, 
and $h\colon \HH \to\left]-\infty,+\infty\right]$ is proper lower semicontinuous and
$\sigma$-strongly convex $(\sigma>0)$. The first term of the objective function may
represent the empirical data loss and the second term the regularizer. This problem arises in
many applications in machine learning, signal processing and statistical estimation, and is commonly called regularized empirical risk minimization \cite{shalev-shwartz_ben-david_2014}. It includes, for instance, ridge regression and (soft margin) support vector machines \cite{shalev-shwartz_ben-david_2014}, more generally Tikhonov regularization \cite[Section 5.3]{kress1998ill}.

In the following we apply Algorithm~\ref{algoAsymain} to the dual of problem
\eqref{prob:20210909a}. Below we provide details. 
Set $\bm{\GG}=\bigoplus_{i=1}^m \GG_i$ and  $\buu = (\uu_{1}, \uu_{2}, \ldots, \uu_{m})$.
Then, the dual of problem \eqref{prob:20210909a} is
\begin{equation}\label{prob:20210909b}
\underset{\buu \in \mathbf{\GG}}{\text{minimize }} F(\buu) 
= h^*\bigg(-\sum_{i=1}^{m} \As^*_i\uu_{i}\bigg) + \sum_{i=1}^m \phi^*_i(\uu_i),
\end{equation}
where, $\As_i^*$ is the adjoint operator of $\As_i$
 $h^*$ and $\phi_i^*$ are 
the Fenchel conjugates of $h$ and $\phi_i$ respectively.
The link between the dual variable $\buu$ and the primal variable $\xx$
is given by the rule $\buu \mapsto \nabla h^*(-\sum_{i=1}^{m} \As^*_i\uu_{i})$.
Since $h^*$ is $(1/\sigma)$-Lipschitz smooth, the dual problem above 
is in the form of problem \eqref{mainprob}. Thus,
Algorithm \eqref{algoAsymain} applied to 
the dual problem \eqref{prob:20210909b} gives
\begin{equation}
\label{eq:algoPRCD5}
\begin{array}{l}
\text{for}\;k=0,1,\ldots\\
\left\lfloor
\begin{array}{l}
\text{for}\;i=1,\dots, m\\[0.7ex]
\left\lfloor
\begin{array}{l}
u^{k+1}_i = 
\begin{cases}
\prox_{\gamma_{i_{k}} \phi_{i_{k}}^*}\big(u_{i_{k}}^{k}+\gamma_{i_{k}}
\As_{i_k}\nabla  h^{*}( -
\sum_{j=1}^{m} \As^*_j u^{k - \ds^k_j}_{j})
\big) &\text{if } i=i_{k}\\
u^k_i &\text{if } i \neq i_{k},
\end{cases}
\end{array}
\right.
\end{array}
\right.
\end{array}
\end{equation}
Suppose that $\nabla  h^{*} = \Bs$ 
is a linear operator and that the delay vector  $\bds^k= (\ds_1^k,\cdots, \ds_m^k)$ is
uniform, that is, $\ds^k_i = \ds_j^k = \ds^k \in \N$. Then, using the primal variable,
the KKT condition $x^k = \nabla  h^{*}( -\sum_{j=1}^{m} \As^*_j u^{k}_{j}) =
- \sum_{j=1}^{m} \Bs \As^*_j u^{k}_{j}$, and the fact that $\bm{u}^{k+1}$ and $\bm{u}^k$
differ only on the $i_k$-component, the algorithm becomes
\begin{equation}
\label{eq:algoPRCD6}
\begin{array}{l}
\text{for}\;k=0,1,\ldots\\
\left\lfloor
\begin{array}{l}
\text{for}\;i=1,\dots, m\\[0.7ex]
\left\lfloor
\begin{array}{l}
u^{k+1}_i = 
\begin{cases}
\prox_{\gamma_{i_{k}} \phi_{i_{k}}^*}\big(u_{i_{k}}^{k}+\gamma_{i_{k}} \As_{i_k} \bm{x}^{k-\ds^k}\big) &\text{if } i=i_{k}\\
u^k_i &\text{if } i \neq i_{k}.
\end{cases}\\
\text{ }\\
\bm{x}^{k+1} = \bm{x}^k - \Bs \As^*_{i_k} (u^{k+1}_{i_k} - u^{k}_{i_k}).
\end{array}
\right.
\end{array}
\right.
\end{array}
\end{equation}
The above algorithm requires a lock during the update of the primal variable $\xx$. On the contrary, the update of the dual variable $\buu$ is completely asynchronous without any lock as in the setting we studied in this paper. To get a better understanding of this aspect, we will expose a concrete example: the ridge regression.


\subsubsection{Example: Ridge regression}
The ridge regression is the following regularized least squares problem.
\begin{equation}
\label{prob:ridgereg}
    \underset{\ww \in \HH}{\text{minimize}}\, \frac{1}{\lambda m} \sum_{i=1}^{m} \left(\yy_{i} - \left\langle\ww, \xx_{i}\right\rangle\right)^2 +\frac{1}{2}\|\ww\|^{2}.
\end{equation}
Its dual problem is
\begin{align*}
\underset{\buu \in \R^m}{\text{minimize}}\, 
\frac{1}{2} \scalarp{(\KK+\lambda m \Id_m) \buu, \buu} -\scalarp{\byy,\buu},
\end{align*}
where $\KK = \XX\XX^{*}$ and $\XX \colon \HH \to \R^m$, with 
$\XX \ww = (\scalarp{\ww,\xx_i})_{1 \leq i\leq m}$. We remark that, in this situation,
$A_i = \scalarp{\cdot, \xx_i}$, $A_i^* = \xx_i$ and $B=\Id$. Let 
$\bds^k= (\ds^k, \ds^k, \cdots, \ds^k)$. With $\ww^k = \XX^{*}\buu^k$ and
considering that the non smooth part $g$ is null, the algorithm is given by
\begin{equation}
\label{eq:ridgereg}
\begin{array}{l}
\text{for}\;k=0,1,\ldots\\
\left\lfloor
\begin{array}{l}
\text{for}\;i=1,\dots, m\\[0.7ex]
\left\lfloor
\begin{array}{l}
u^{k+1}_i = 
\begin{cases}
u_{i_{k}}^{k} - \gamma_{i_{k}} \big(\langle \xx_{i_k}, \bm{w}^{k-\ds^k}\rangle+\lambda m u_{i_k}^{k-\ds^{k}}-\yy_{i_k}\big) &\text{if } i=i_{k}\\
u^k_i &\text{if } i \neq i_{k}.
\end{cases}\\
\text{ }\\
\bm{w}^{k+1} = \bm{w}^k - \gamma_{i_{k}} \xx_{i_k}\big(\uu^{k+1}_{i_k} - \uu^k_{i_k}\big).
\end{array}
\right.
\end{array}
\right.
\end{array}
\end{equation}
\begin{remark}
Now we will compare the above dual asynchronous algorithm to the asynchronous 
stochastic gradient descent (ASGD) \cite{niu2011hogwild, agarwal2011distributed}. 
We note that \eqref{eq:ridgereg} yields
\begin{align*}
    \bm{w}^{k+1} &= \bm{w}^k - \gamma_{i_{k}} \xx_{i_k}\big(u^{k+1}_{i_k} - u^k_{i_k}\big) \\
    &= \bm{w}^k - \gamma_{i_{k}} \big(\langle \xx_{i_k}, \bm{w}^{k-\ds^k}\rangle\xx_{i_k} +\lambda m u_{i_k}^{k-\ds^k}\xx_{i_k} -\yy_{i_k}\xx_{i_k}\big).
\end{align*}
Instead, applying asynchronous SGD to 
the primal problem \eqref{prob:ridgereg} multiply by $\lambda m$, we get
\begin{align*}
\bm{w}^{k+1} = \bm{w}^k - \gamma_{k}^{\prime} \big(\langle \xx_{i_k},
\bm{w}^{k-\ds^k}\rangle\xx_{i_k} +\lambda m \bm{w}^{k-\ds^k}-\yy_{i_k}\xx_{i_k}\big).
\end{align*}
We see that the only difference is the second term inside the parentheses in
both updates. Indeed the term $\bm{w}^{k-\ds^k} = \XX^{*}\bm{u}^{k-\ds^k} = 
\sum_{i=1}^m u_{i}^{k-\ds^k}\xx_{i}$ in ASGD is replaced by only one summand
$u_{i_k}^{k-\ds^k}\xx_{i_k}$ in our algorithm. However,
a major difference between the two approaches lies in the way the stepsize is
set. Indeed, in ASGD, the stepsize $\gamma_k^{\prime}$  is chosen with respect
to the operator norm of $\KK + \lambda m \Id$ i.e., the Lipschitz constant of
the full gradient of the primal objective function, see 
\cite[Theorem 1]{agarwal2011distributed}. By contrast, in algorithm
\eqref{eq:ridgereg}, for all $i \in [m]$, the stepsizes $\gamma_i^k$  are chosen with respect to the Lipschitz constant of the partial
derivatives of the dual objective function i.e., $\KK_{i,i} + \lambda m$. Not
only the latter are easier to compute, they also allow for possibly longer
steps along the coordinates.
\end{remark}
\section{Experiments}\label{sect:20220825a}
In this section, we will present some experiments with the purpose of assessing our theoretical findings and making comparison with related results in the literature. All the codes are available on GitHub\footnote{https://github.com/cheiktraore/Codes\_Paper\_Asc\_Coord\_Desc}.

We coded the mathematical model of asynchronicity in \eqref{eq:algoPRCD2}. At each iteration we compute the forward step using gradients that are possibly outdated. The delay vector components are a priori chosen according to a uniform distribution on $\{0,1,\ldots,\tau\}$.
The block coordinates are updated with a uniform distribution independent from the delay vector.  We considered three kinds of experiments: in the first one we did a speedup test for our algorithm on the Lasso problem. This allows to check whether the speed of convergence increases linearly with the number of machines used. Then, we considered a comparison with the synchronous version of the algorithm in order to show the advantage of the asynchronous implementation. Finally, in the third group of experiments we compared our algorithm with those by Liu et al \cite{liu2015asynchronous} and Cannelli et al \cite{cannelli2019asynchronous}.

\subsection{Speedup test} \label{sec:exp_speedup}

In this section we consider the Lasso problem \eqref{app:lasso} with $n=100$ and $m \in \{500, 1000, 2000, 8000\}$. $\lambda$ is chosen small enough so that the minimizer $\bx^*$ has some non zero components. For more flexibility,  we used synthetic data, which were generated using the function \textsc{make\_correlated\_data} of the python library \textsc{celer}. 
 This function creates a matrix $\As$ with columns generated according to the Autoregressive (AR) model\footnote{The code is available at https://github.com/mathurinm/celer/blob/501788e/celer/datasets/simulated.py\#L10}. Then $\bs$ is generated as $\bs = \As\bw + \epsilon$, where $\epsilon$ is a Gaussian random vector, with zero mean and variance equal to the identity, such that the signal to noise ration (SNR) is $3$ and $\bw$ is a vector with $1\%$ of nonzero entries. The nonzero blocks of $\bw$ are chosen uniformly and their entries are generated according to the standard normal distribution.
 As in \cite{liu2015asynchronous, cannelli2019asynchronous}, we make the assumption that $\tau$ is proportional to the number of machines. Since we use 10 cores, we fix $\tau = 10$ like in \cite{lian2015asynchronous}.
For a fixed data, we run the algorithm 10 times and average it.
 Similarly to \cite{liu2015asynchronous, cannelli2019asynchronous}, in our experiment the speedup gets better when we increase the number of blocks, see Figure \ref{fig:speedup}. This can be explained by the fact that the algorithm has to run long enough in order to minimize the cost of parallelization --- the initialization cost, the mandatory locks in order to avoid data racing, etc. Also, if there are more blocks, the probability of two machines having to write to the same block at the same time is reduced and so is the number of locks. All these observations align with the known fact that the more there are cores, the more the problem should be complex to see good speedup.

\begin{figure}[t]
     \centering
     \begin{subfigure}[b]{0.475\textwidth}
         \centering
         \includegraphics[width=\textwidth]{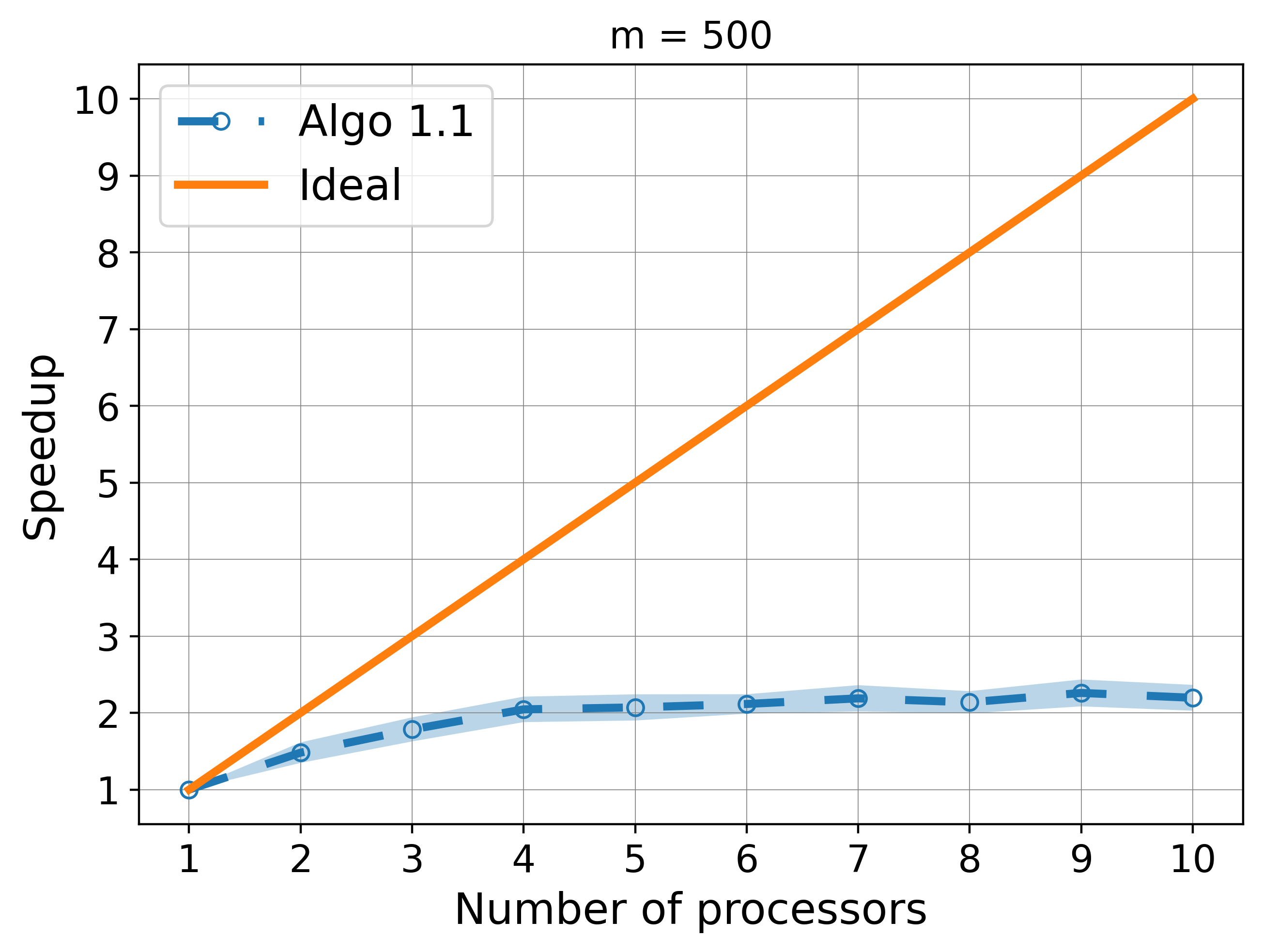}
     \end{subfigure}
     \hfill
     \begin{subfigure}[b]{0.475\textwidth}
         \centering
         \includegraphics[width=\textwidth]{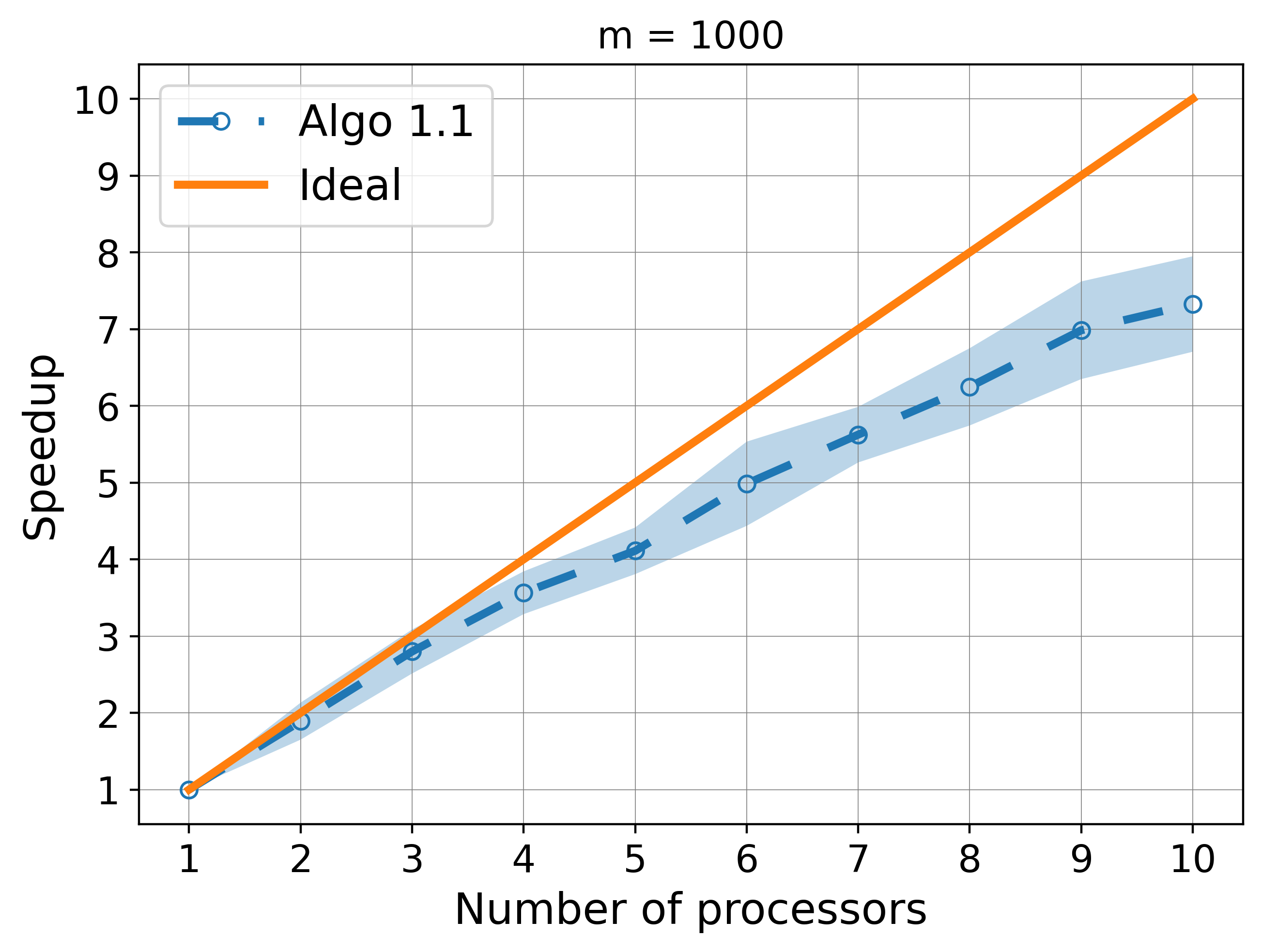}
     \end{subfigure}
     \begin{subfigure}[b]{0.475\textwidth}
         \centering
         \includegraphics[width=\textwidth]{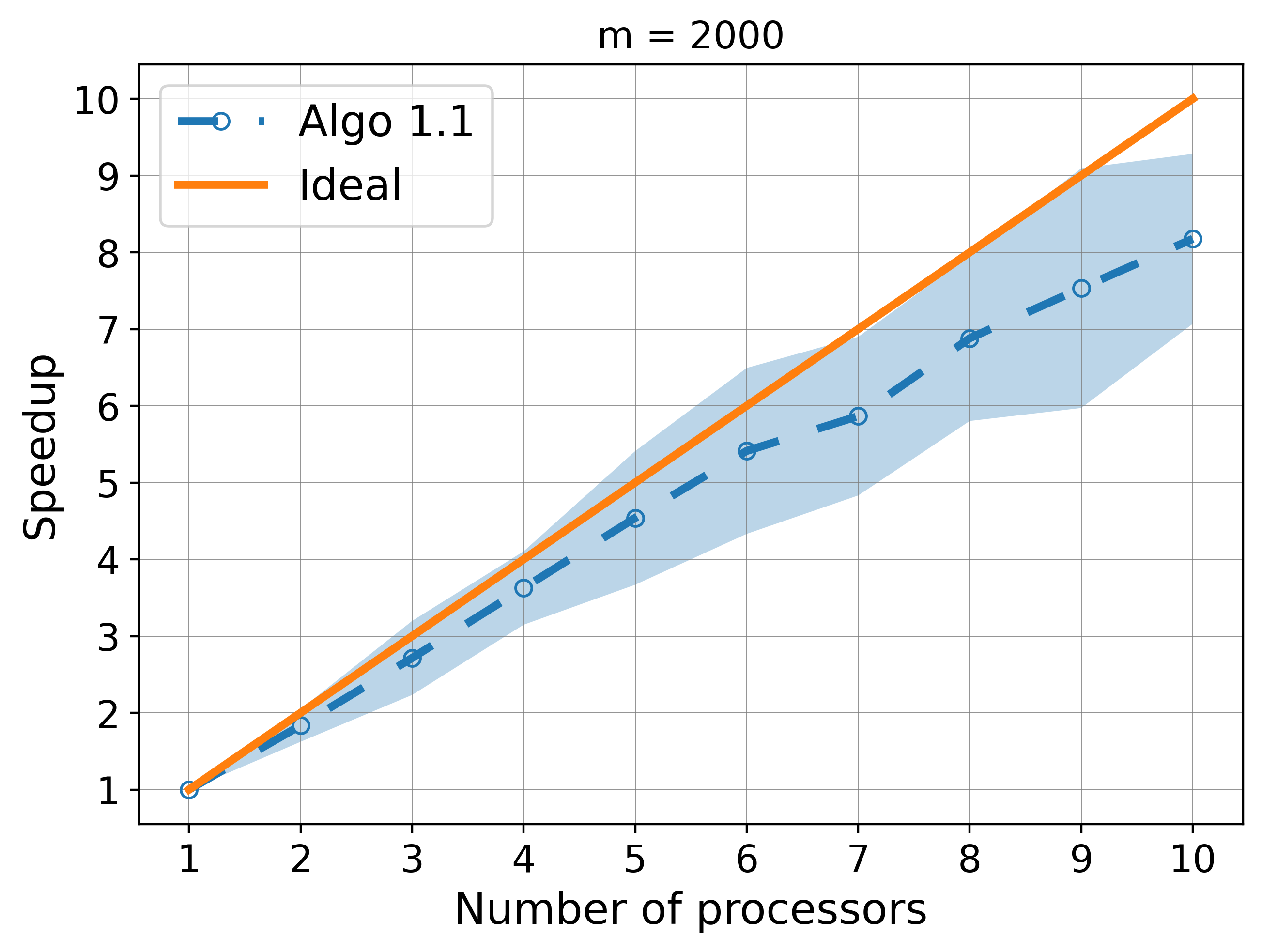}
     \end{subfigure}
     \hfill
     \begin{subfigure}[b]{0.475\textwidth}
         \centering
         \includegraphics[width=\textwidth]{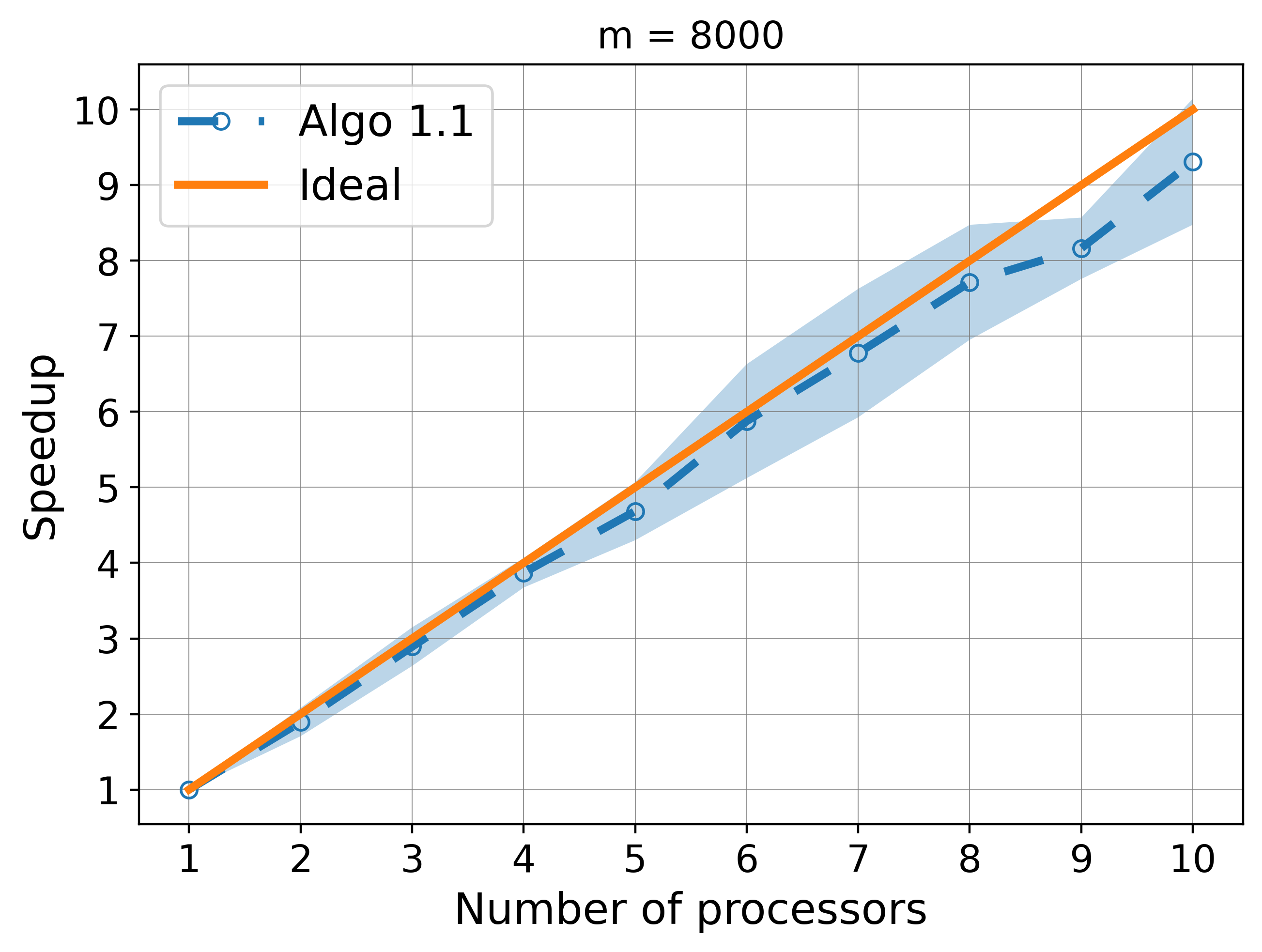}
     \end{subfigure}
        \caption{The plots showed the speedup obtain by Algorithm 1.1 compared to the ideal speedup for different number of blocks. The shaded zones illustrate the standard deviation of the results over 10 trials.}
        \label{fig:speedup}
\end{figure}

\begin{figure}[ht!]
     \centering
     \includegraphics[width=.5\textwidth]{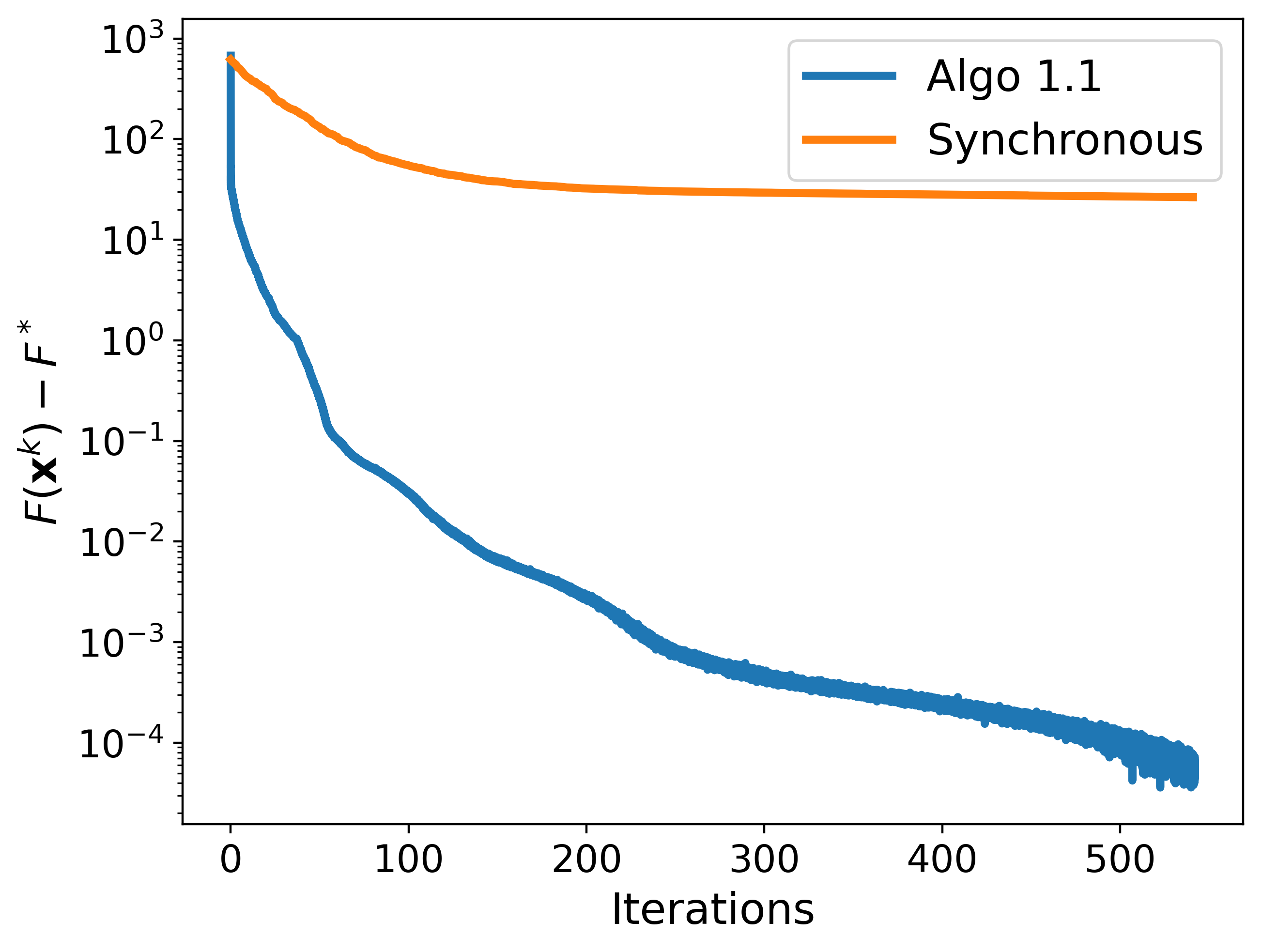}
    \caption{Comparison of Algorithm \ref{algoAsymain} to its synchronous counterpart.}
    \label{fig:synch}
\end{figure}

\subsection{Comparison with the synchronous version} \label{sec:exp_synchronous}

We compared Algorithm \ref{algoAsymain} to its synchronous counterpart in the Lasso case. The data, as well as the parameters, is generated as in the speedup experiment. The step size of the synchronous algorithm is set as suggested in \cite{salzo2021parallel} for a non sparse matrix $\As$. We run both algorithms for $120$ seconds and compare the distances of their function values to the minimum. As expected, Algorithm \ref{algoAsymain} is faster; see Figure \ref{fig:synch}.

\subsection{Comparison with other asynchronous algorithms}
In this section we illustrate the results of the comparison with the algorithms proposed in \cite{liu2015asynchronous} and \cite{cannelli2019asynchronous}.
As for \cite{cannelli2019asynchronous},  we set (in the notation of the paper) the relaxation parameter $\gamma = 1$ and $c_{\tilde{f}}=2\beta$ so that
\begin{equation*}
x^{k+1}_i = 
\begin{cases}
\prox_{ (1/2\beta) g_{i}} \big(x^{k-d^k_{i}}_{i} - (1/2\beta) \nabla_{i} f (\bx^{k-\bds^k})\big) &\text{if } i=i_{k}\\
x^k_i &\text{if } i \neq i_{k}.
\end{cases}
\end{equation*}
Then, according to Theorem~1 in \cite{cannelli2019asynchronous}, we choose $2\beta > L_f(1+\delta^2/2)$ where $\delta = \tau$ is the maximum delay. We note that this model is slightly different from ours since the delay is present not only in the gradient.

 In \cite{liu2015asynchronous}, the same algorithm as \eqref{eq:algoPRCD2} is considered, but with a stepsize $\gamma$ which is the same for all the blocks. In our comparisons, we choose the step according to the conditions required by the main Theorem $4.1$ in \cite{liu2015asynchronous}, since the hypotheses of Corollary 4.2 are not satisfied for our datasets\footnote{For the two datasets we used, \textsc{YearPredictionMSD.t} and \textsc{Splice.t}, 
 we have that $\sqrt{m}/(4 e \Lambda)$
 is equal to $0.62084123$ $0.00459623$ respectively, so that
 condition \eqref{eq:20220902a} is never satisfied
 by any nonnegative integer $\tau$.}, see the discussion in  Remark \ref{rmk:20211216a} \ref{item:20211216a}.
If $\tau$ is the maximum delay, Theorem $4.1$ in \cite{liu2015asynchronous} requires the following conditions on the stepsize:
\begin{equation*}
\gamma < \frac{\sqrt{n}(1-\rho^{-1})-4}{4(1+\theta)L_{\text{res}}/L'_{\max}}
\quad
\text{with}\ 
\theta=\frac{\rho^{(\tau+1) / 2}-\rho^{1 / 2}}{\rho^{1 / 2}-1},
\end{equation*}
which only makes sense if the right hand side is strictly positive, so when $n > 16$ and $\rho > \frac{1 + 4/\sqrt{n}}{1 - 16/n}$ (instead of $\rho > 1 + 4/\sqrt{n}$ as claimed in \cite{liu2015asynchronous}). So, in the experiments, we set $\rho > \frac{1 + 4/\sqrt{n}}{1 - 16/n}$. This leads in general to very small stepsizes, as we will further discuss in the next section. 

\begin{figure}[t!]
     \centering
     \begin{subfigure}[b]{0.475\textwidth}
         \centering
         \includegraphics[width=\textwidth]{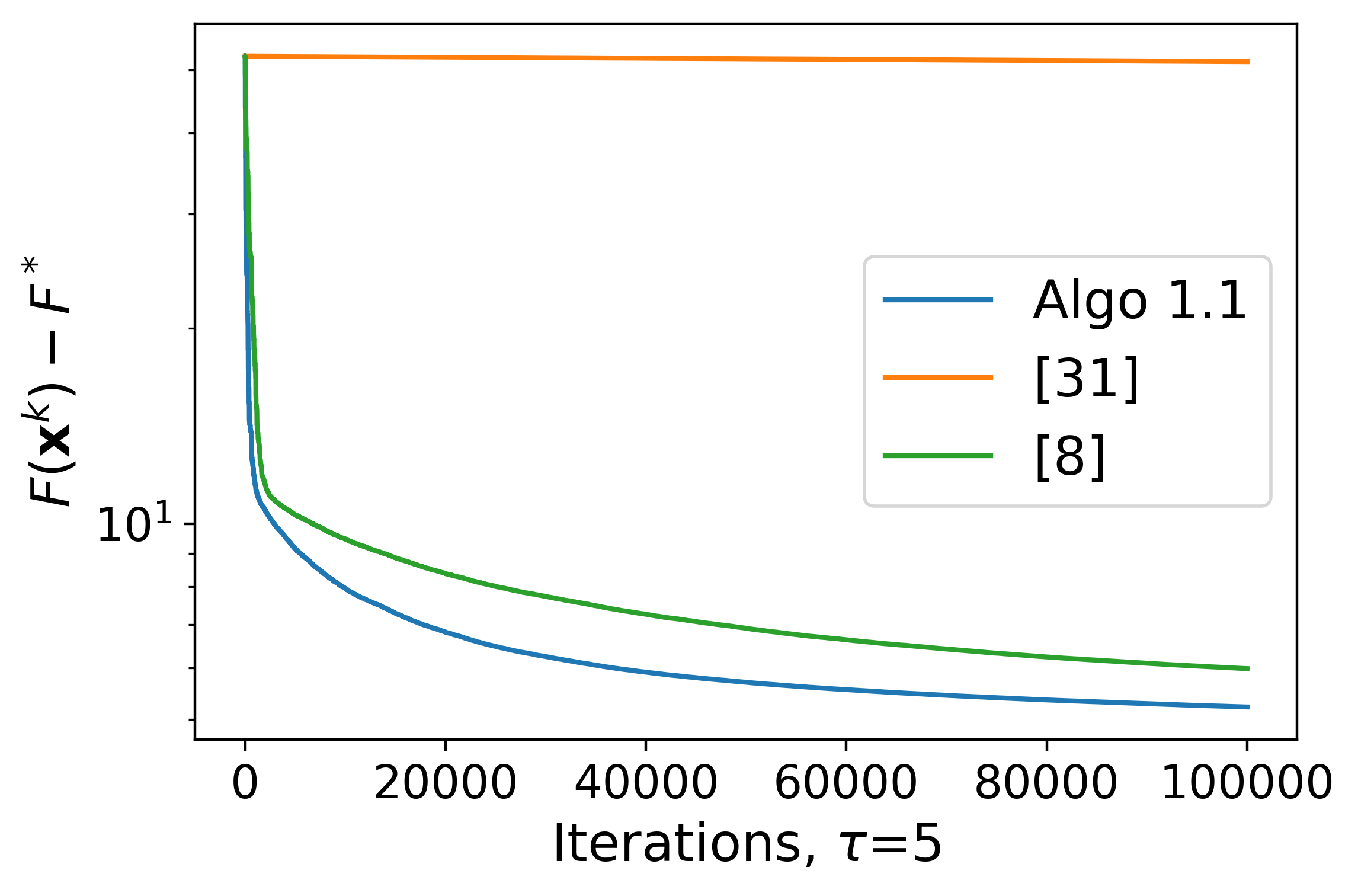}
     \end{subfigure}
     \hfill
     \begin{subfigure}[b]{0.475\textwidth}
         \centering
         \includegraphics[width=\textwidth]{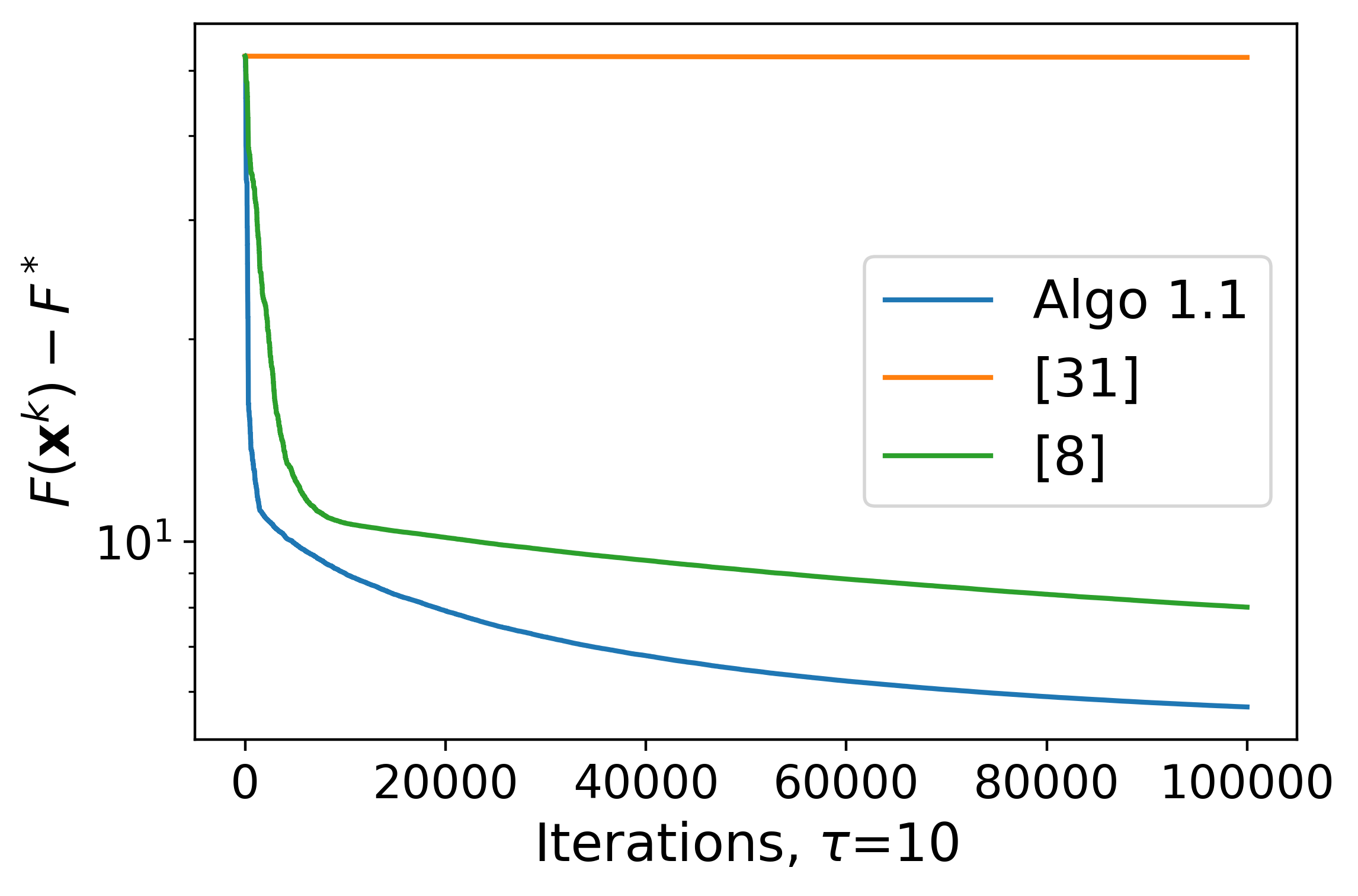}
     \end{subfigure}
     \begin{subfigure}[b]{0.475\textwidth}
         \centering
         \includegraphics[width=\textwidth]{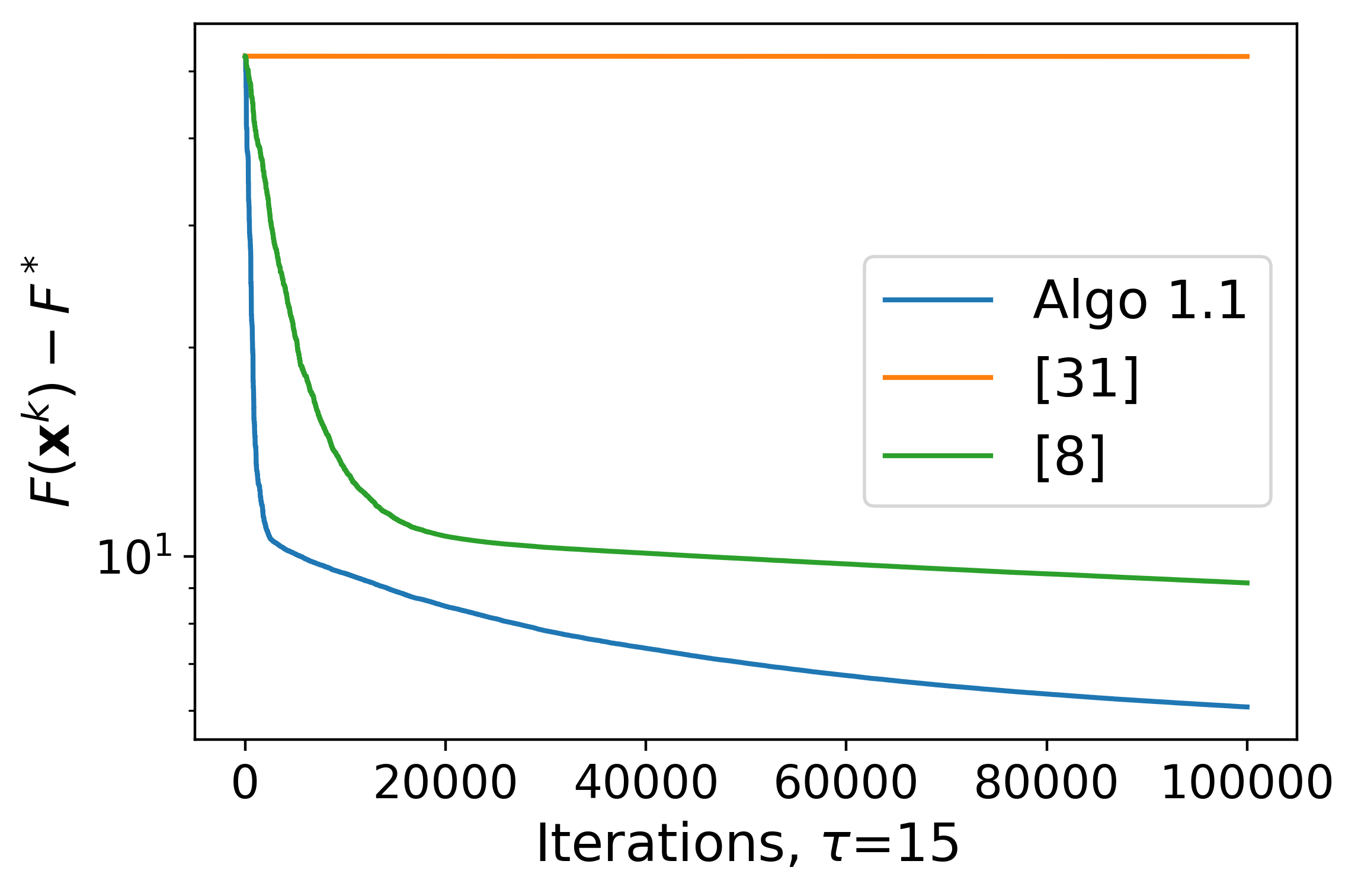}
     \end{subfigure}
     \hfill
     \begin{subfigure}[b]{0.475\textwidth}
         \centering
         \includegraphics[width=\textwidth]{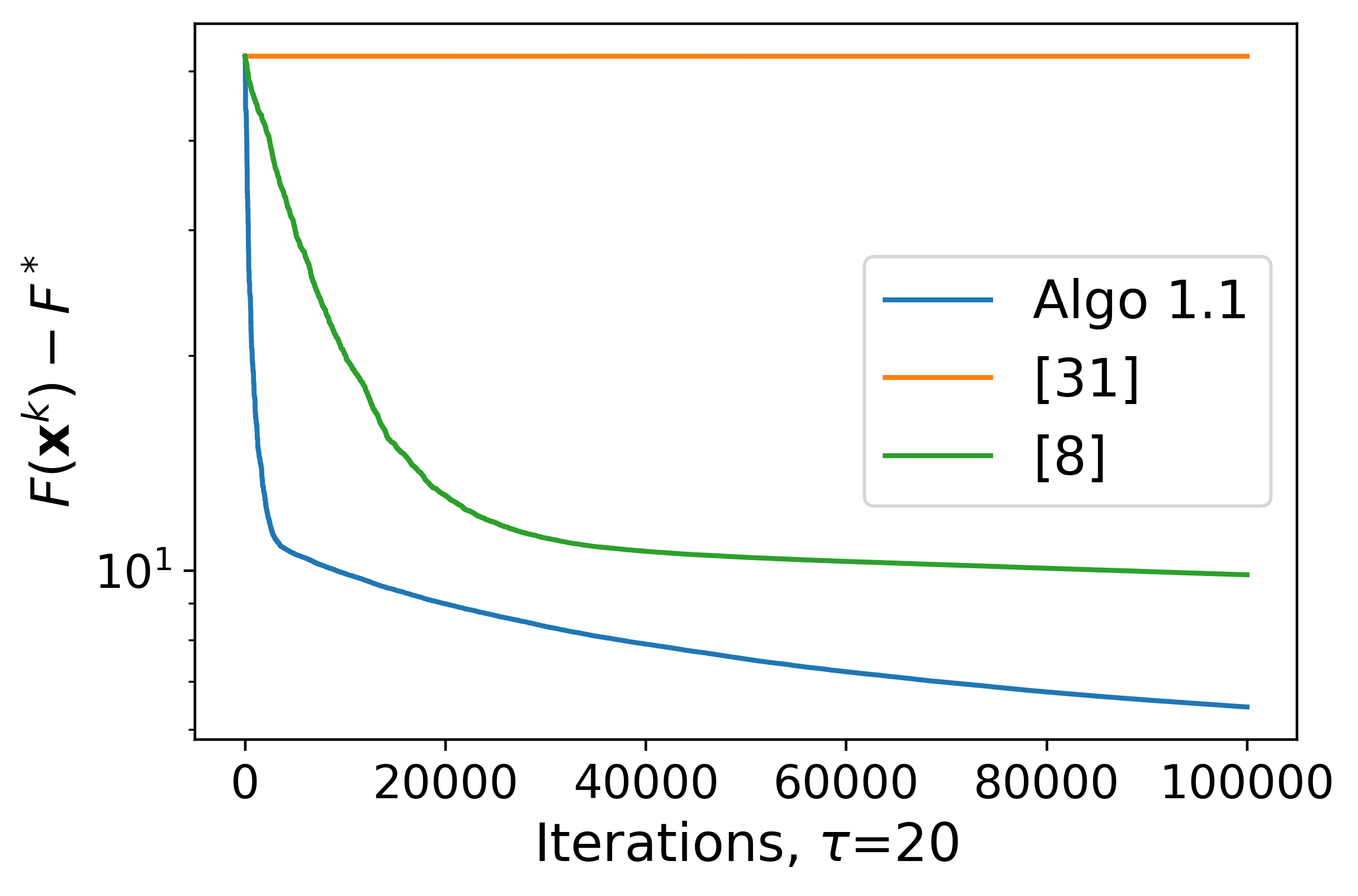}
     \end{subfigure}
        \caption{The plots show the behavior of $F(\bx_k)-  F^*$ for the 3 algorithms applied to a lasso loss with different values of $\tau$: 5, 10, 15, 20.}
        \label{fig:lasso}
\end{figure}

\subsubsection{Lasso problem}
In this section we consider the Lasso problem \eqref{app:lasso} with $m=90$, $n=51630$, and $\lambda=0.01$. We use the data \textsc{YearPredictionMSD.t} from \textsc{libsvm}\footnote{\label{ftnte:1} https://www.csie.ntu.edu.tw/{\raise.17ex\hbox{$\scriptstyle\sim$}}cjlin/libsvmtools/datasets/} to generate the matrix $\As$.
Before showing the results, we briefly comment on the experimental set-up.
As shown in Section~\ref{sec:lasso}, in this case $L_i=\|\As_{\cdot i}\|_2^2$ and $L_{\mathrm{res}}=\max_{i}\|\As^{\intercal}\As_{\cdot i}\|_{2}$. 
In $\cite{cannelli2019asynchronous}$,  $L_f=L_{\mathrm{res}}$ and in \cite{liu2015asynchronous}  $L'_{\max}=\max_{i}\|\As^{\intercal}\As_{\cdot i}\|_{\infty}$.

Looking at the results, we see that our algorithm outperforms those in \cite{liu2015asynchronous} and \cite{cannelli2019asynchronous}, see Figure \ref{fig:lasso}. This difference is due to the fact that our stepsize is bigger than the other two. Indeed, in \cite{liu2015asynchronous}
and \cite{cannelli2019asynchronous} the stepsizes
have a worse dependence on the maximum delay $\tau$ 
(inverse quadratically in \cite{cannelli2019asynchronous}
and exponentially in \cite{liu2015asynchronous}),
which ultimately shorten the stepsizes.
Also, in both \cite{liu2015asynchronous} and \cite{cannelli2019asynchronous} the stepsize is the same for all the blocks, so the algorithm is more sensitive to the conditioning of the problem.
An overall comparison of the effect of $\tau$
on the stepsize is shown in Figure \ref{fig:dpdcetau}.

\begin{figure}[t]
     \centering
     \begin{subfigure}[b]{0.475\textwidth}
         \centering
         \includegraphics[width=\textwidth]{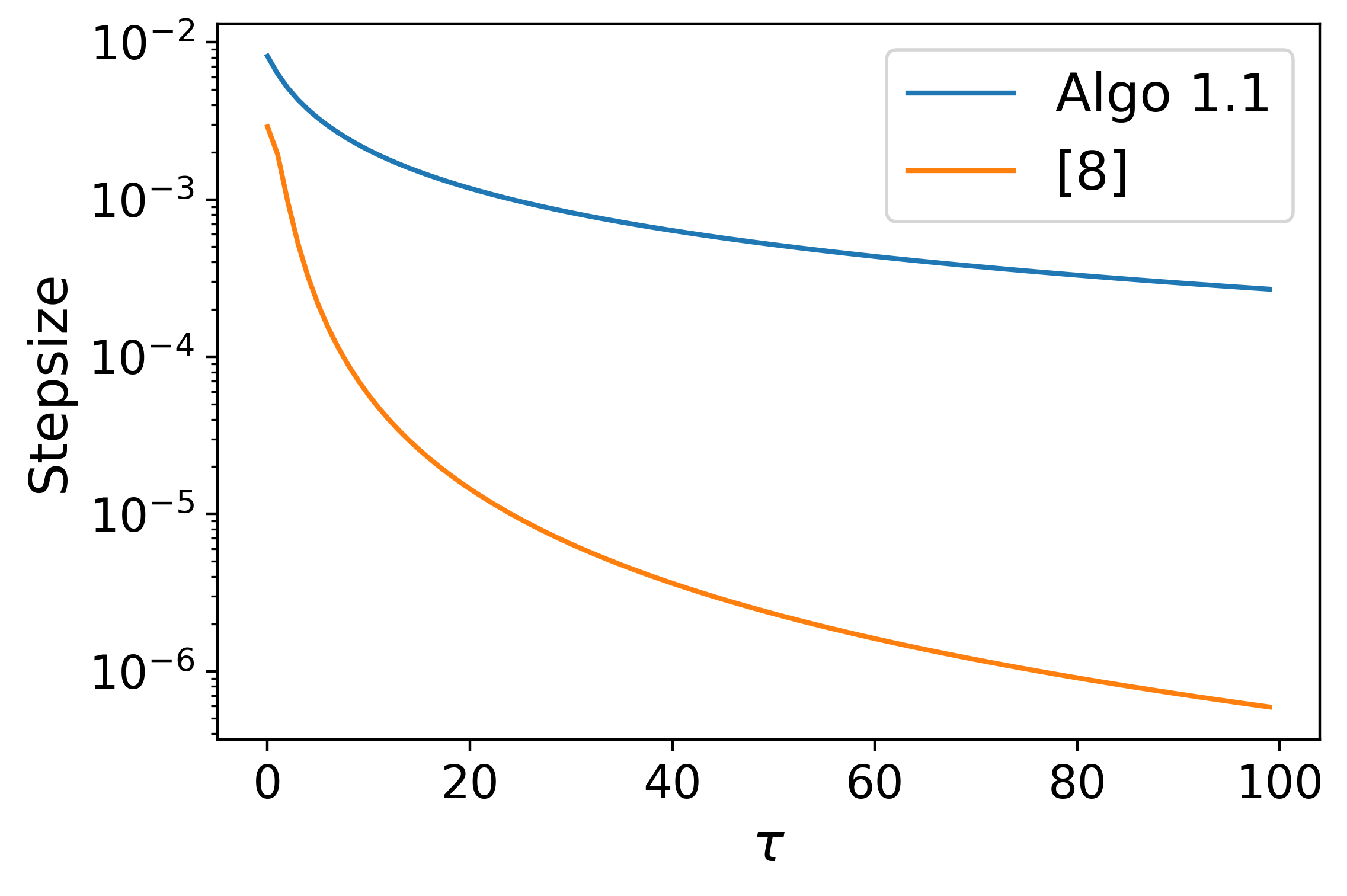}
     \end{subfigure}
     \hfill
     \begin{subfigure}[b]{0.475\textwidth}
         \centering
         \includegraphics[width=\textwidth]{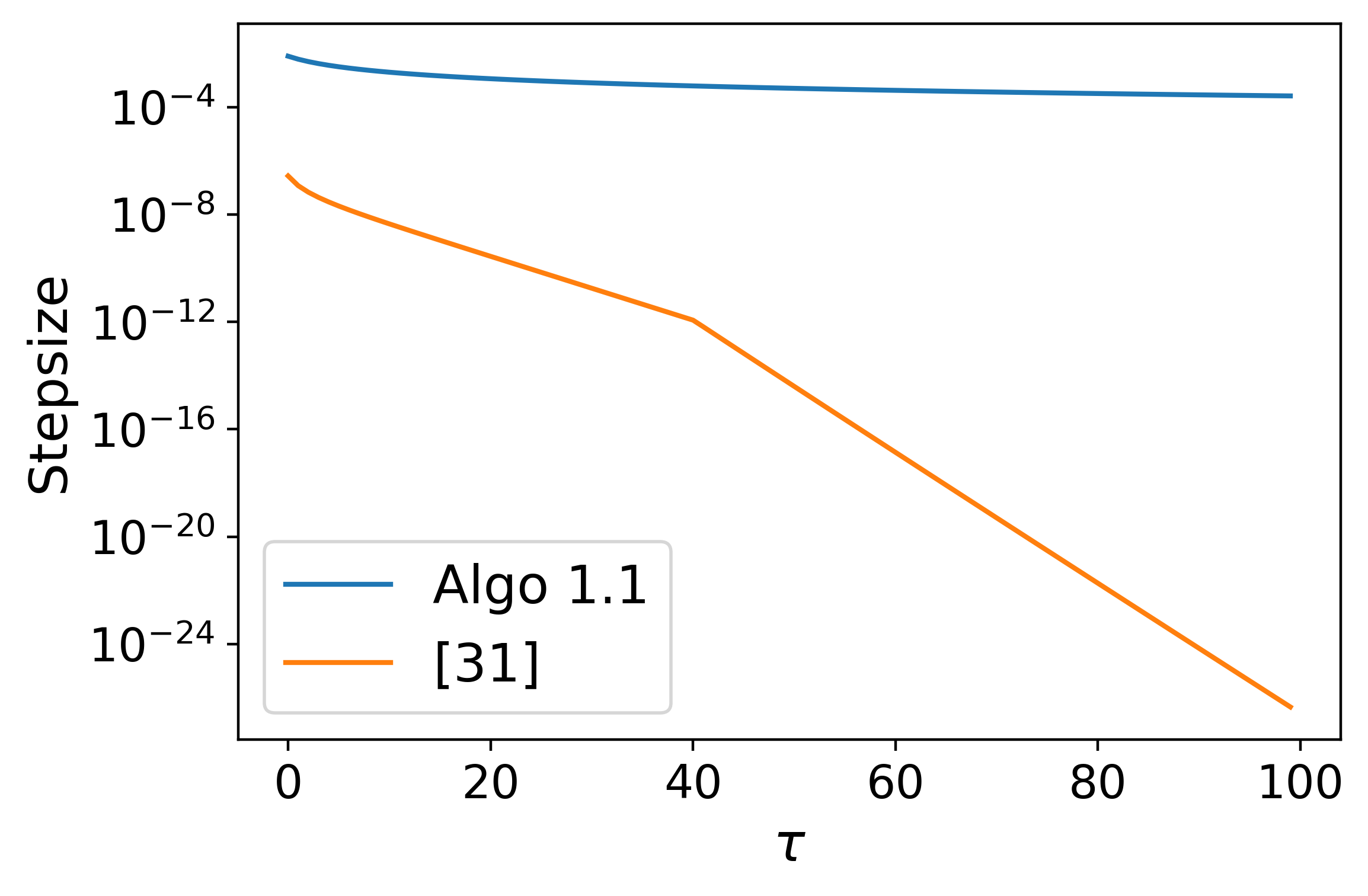}
     \end{subfigure}
        \caption{This figure shows how the minimum of our stepsizes fares against the two others when $\tau$ increases on a lasso problem.}
        \label{fig:dpdcetau}
\end{figure}

\subsubsection{Logistic regression}
For another comparison, next we consider the $\ell_1$ regularized logistic loss:
\begin{equation}\label{app:logistic}
    F(x) = \frac{1}{n} \sum_{i=1}^n \log(1+\exp\{-b_i\langle a_i, \bx\rangle\}) + \lambda \|\bx\|_{1}.
\end{equation}
For this experiment we use the data \textsc{Splice.t} from \textsc{libsvm}\footnote{\url{https://www.csie.ntu.edu.tw/\~cjlin/libsvmtools/datasets/}} with $m=60$, $n=2175$, and $\lambda=0.01$.
Let $\As \in \R^{m \times n}$ be the matrix with columns the $a_i$'s 
($i \in [n]$). We denote by $\|\cdot\|$, $\|\cdot\|_{\infty}$,
$\|\cdot\|_{F}$, the spectral norm, the infinity norm, and the Frobenius norm of matrices, respectively. 
The relevant constants for the stepsizes are 
\begin{itemize}
    \item $L_{\mathrm{res}}= \frac{1}{n}\|\As\| \max_j \|\As_{j \cdot}\|_2$ for our algorithm and \cite{liu2015asynchronous},
    \item $L^\prime_{\max } = \frac{1}{n} \|\As\|_{\infty} \max_j \|\As_{j \cdot}\|_{\infty}$ for \cite{liu2015asynchronous},
    \item $L_j = \frac{1}{n}\|\As_{j \cdot}\|_2^2$, $j \in [m]$, for our algorithm, where $\As_{j \cdot}$ is the $j$-th row of $\As$.
    \item $L_f = \frac{1}{n}\|\As\|_{F} \max_j \|\As_{j \cdot}\|_2$ for \cite{cannelli2019asynchronous}.
\end{itemize}

So, the stepsizes range from about 
$1.1191*10^{-3} \text{ to } 7.5164*10^{-3}$ for \cite{cannelli2019asynchronous}, $5.6537*10^{-8} \text{ to } 2.1571*10^{-10}$ for \cite{liu2015asynchronous}, and $2.2605*10^{-2} \text{ to } 6.1590*10^{-3}$ for our algorithm.
The results show the same trend as in the Lasso case, actually 
with even larger differences, see Figure \ref{fig:logistic}.

\begin{figure}
     \centering
     \begin{subfigure}[b]{0.475\textwidth}
         \centering
         \includegraphics[width=\textwidth]{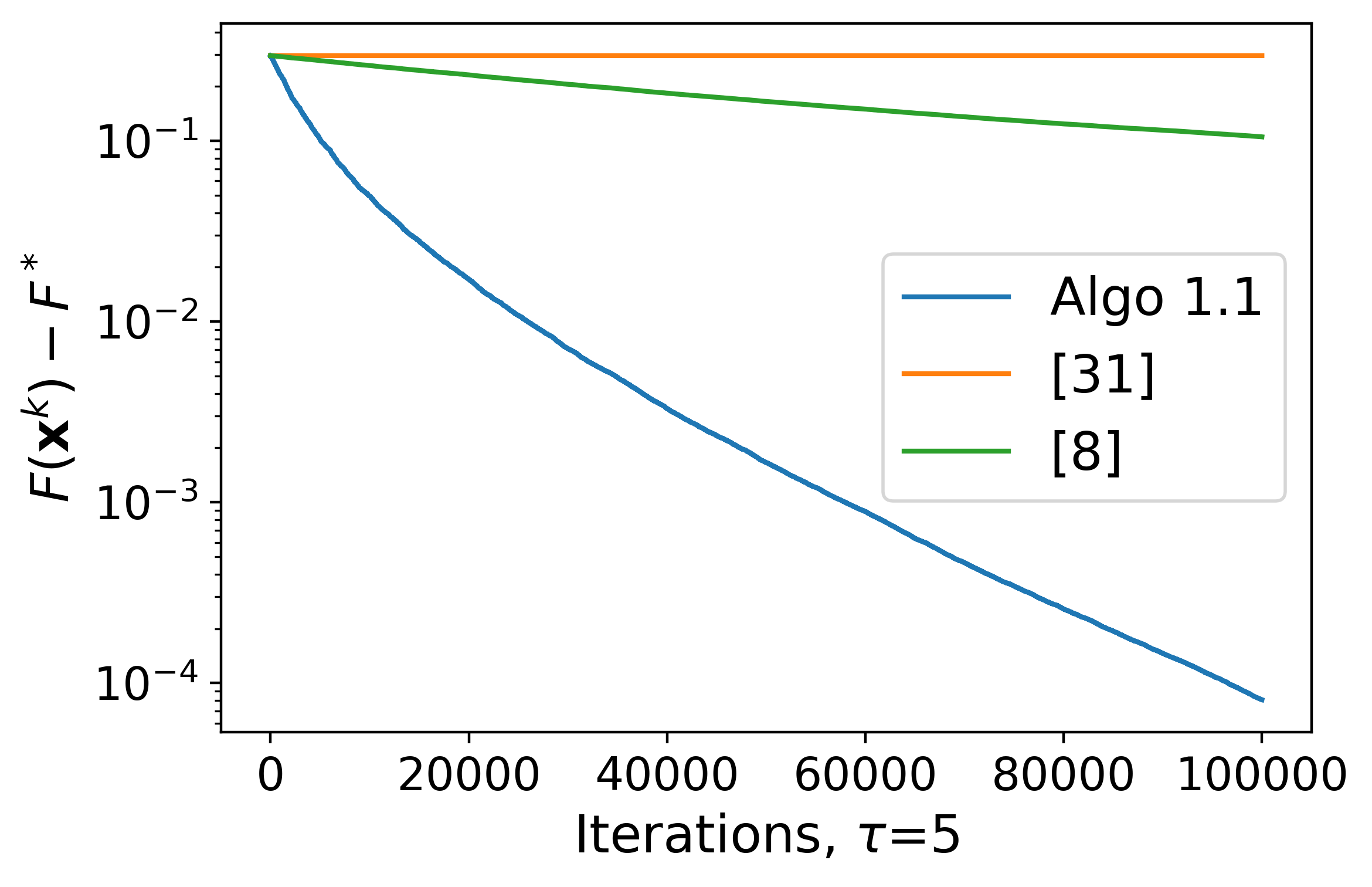}
     \end{subfigure}
     \hfill
     \begin{subfigure}[b]{0.475\textwidth}
         \centering
         \includegraphics[width=\textwidth]{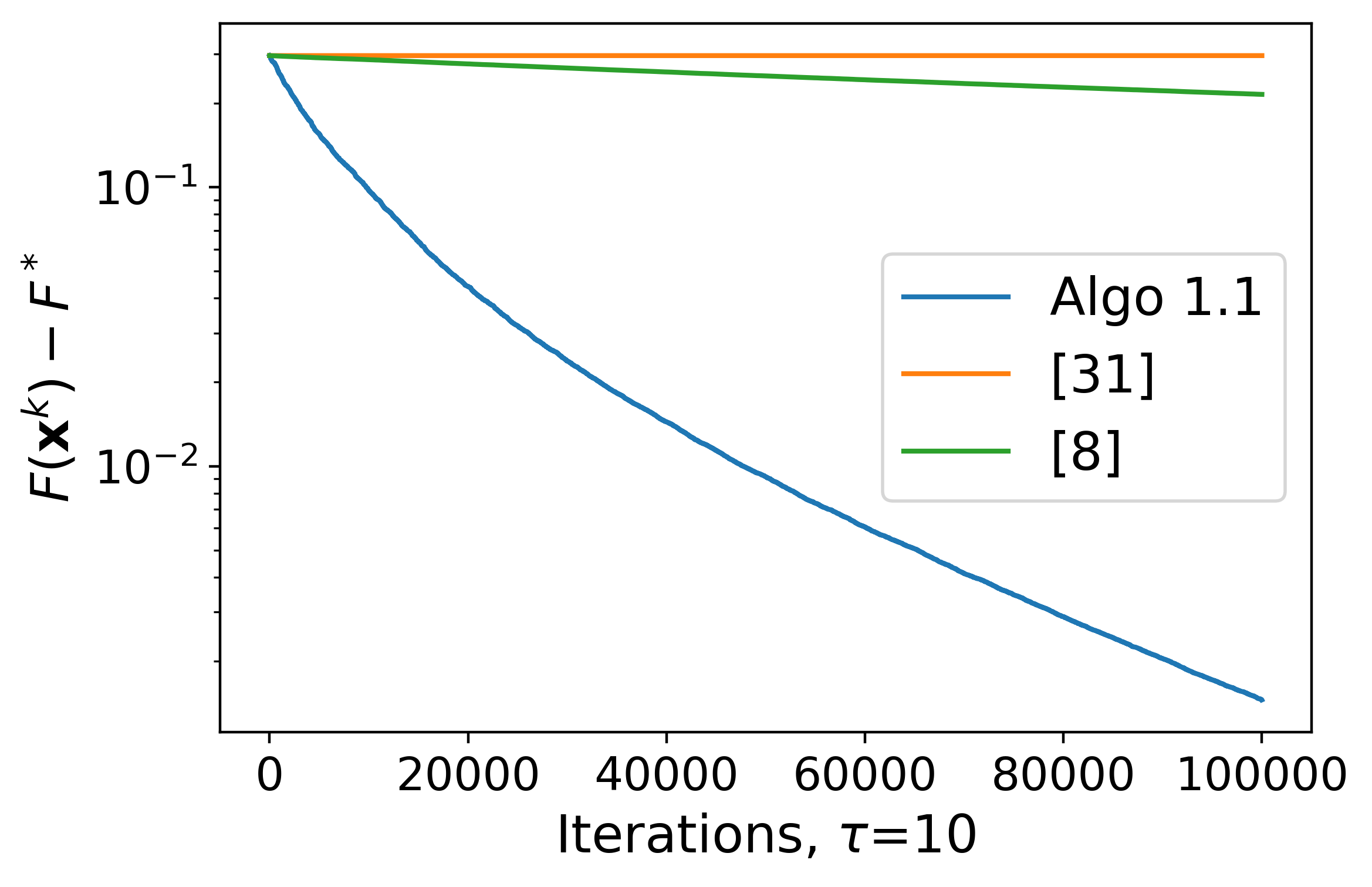}
     \end{subfigure}
     \begin{subfigure}[b]{0.475\textwidth}
         \centering
         \includegraphics[width=\textwidth]{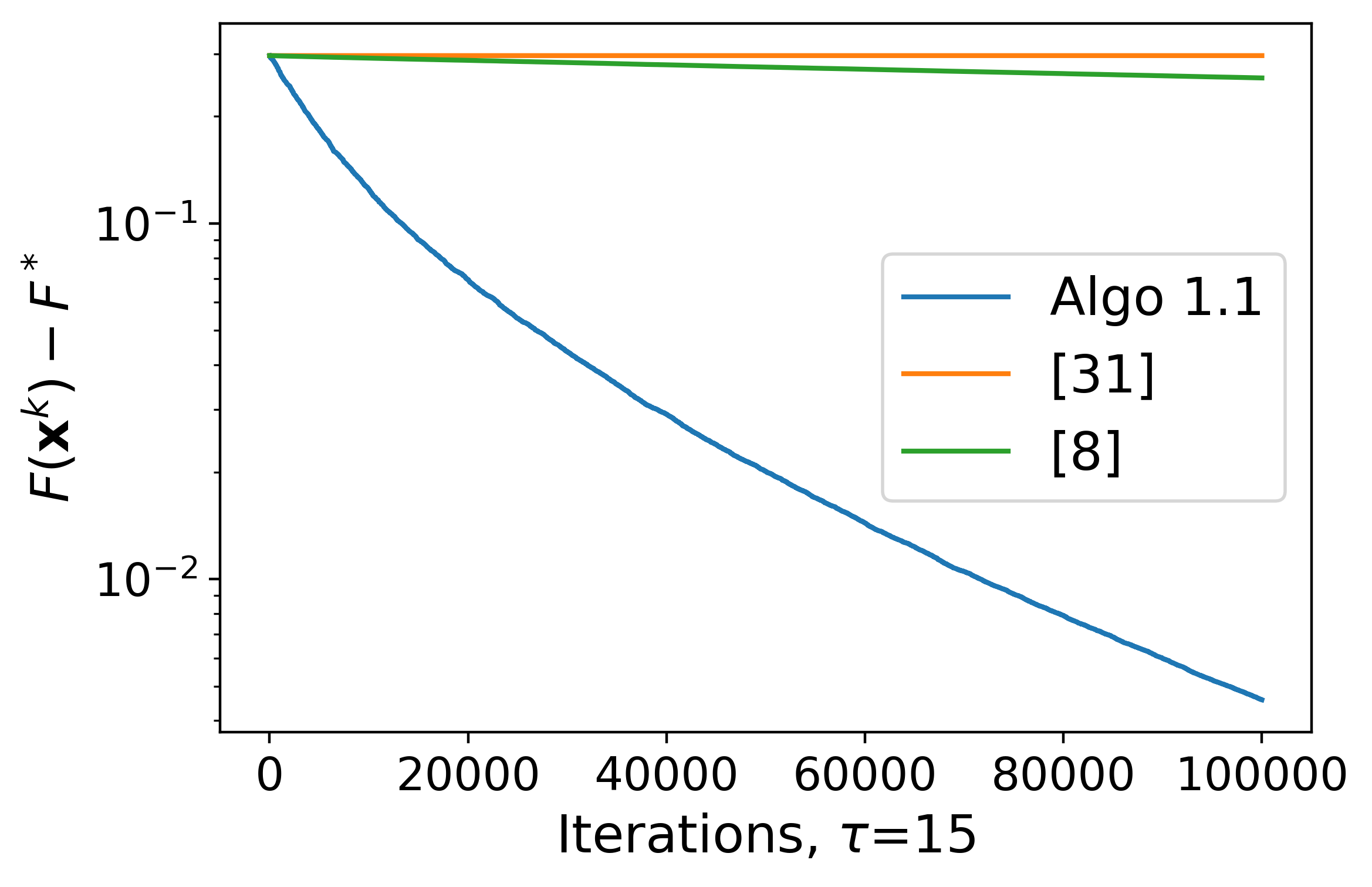}
     \end{subfigure}
     \hfill
     \begin{subfigure}[b]{0.475\textwidth}
         \centering
         \includegraphics[width=\textwidth]{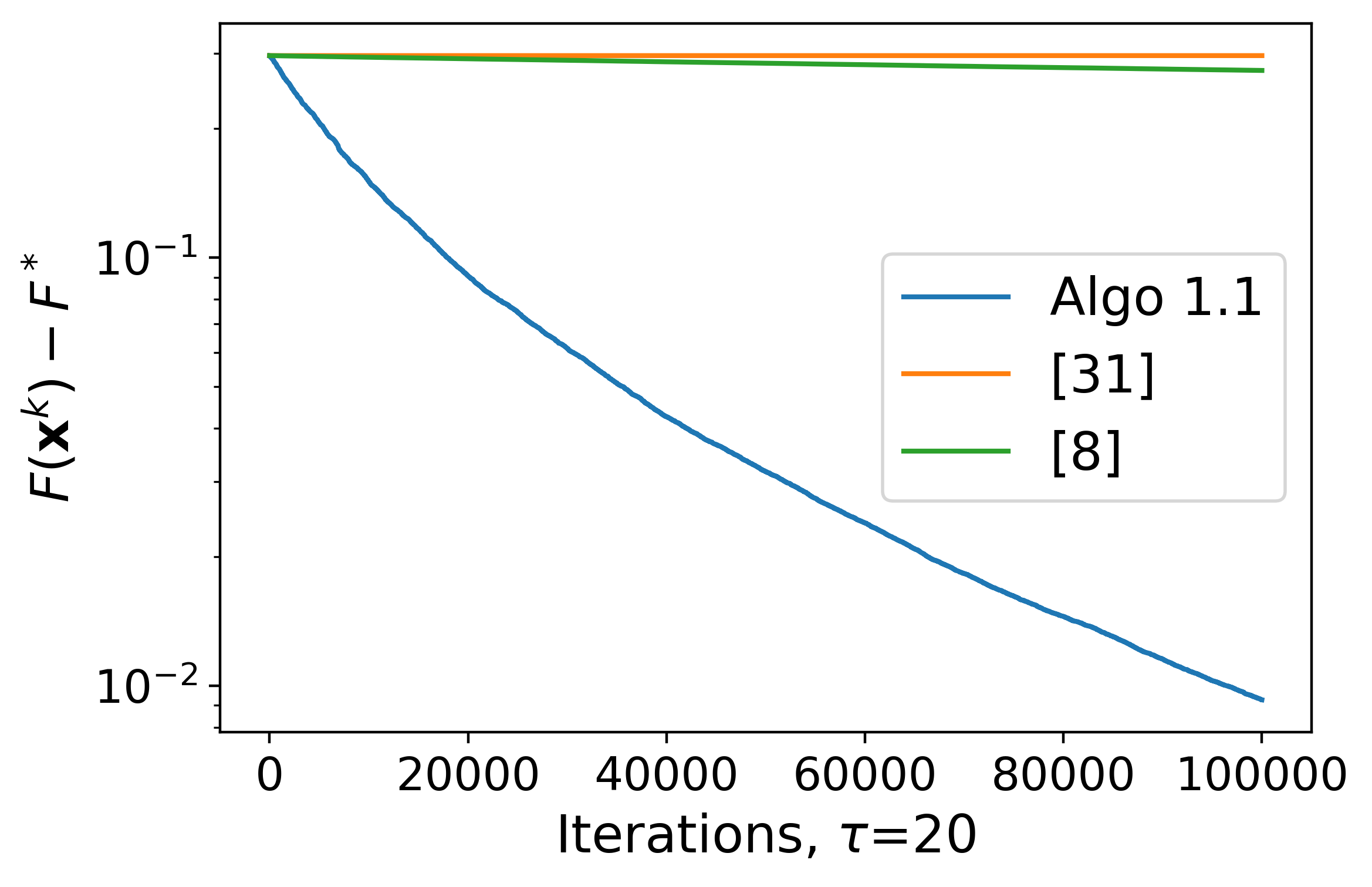}
     \end{subfigure}
        \caption{The plots show the behavior of $F(\bx_k)-  F^*$ for the 3 algorithms applied to a regularized logistic loss for different values of $\tau$: 5, 10, 15, 20.}
        \label{fig:logistic}
\end{figure}


\appendix

\noindent {\huge \textbf{Appendices}}
\section{Proofs of the auxiliary Lemmas in Section~\ref{sec:preliminaries}} 
\label{app:proof}
\label{AppendixA} 

In this section, for reader's convenience, we provide detailed proofs of the Lemmas presented in Section~\ref{sec:preliminaries}, even though they are mostly not original. They are adapted from or can be found, e.g., in \cite{liu2015asynchronous,salzo2021parallel}.

\vspace{2ex}
\noindent
\textbf{Proof of Lemma \ref{p:20170918c}.}
Let $k \in \N$.
Since, for every $i \in [m]$, $\ds_i^k \leq \min\{k,\tau\}$, we have 
\begin{align}
\nonumber\bx^{k- \bds^k} - \bx^k  &= \sum_{i=1}^m \Js_i(x^{k - \ds^k_i}_i - x^k_i)\\
\nonumber&=  \sum_{i=1}^m \Js_i\bigg( \sum_{h=k-\ds^k_i}^{k-1} (x^{h}_i - x^{h+1}_i)\bigg)\\
\nonumber&=  \sum_{i=1}^m \Js_i\bigg( \sum_{h=k-\tau}^{k-1} \delta_{h,i}(x^{h}_i - x^{h+1}_i)\bigg)\\
\label{eq:20170913a}&=   \sum_{h=k-\tau}^{k-1} \sum_{i=1}^m  \Js_i\big(\delta_{h,i}(x^{h}_i - x^{h+1}_i)\big).
\end{align}
where $\delta_{h,i}=1$ if $h \geq k - \ds^k_{i}$ and $\delta_{h,i} = 0$ if $h<k - \ds^k_{i}$. 
Note that for any $h \in \{k-\tau, \dots, k-1\}$, in the sum
\begin{equation*}
\sum_{i=1}^m  \Js_i\big(\delta_{h,i}(x^{h}_i - x^{h+1}_i)\big)
\end{equation*}
at most one summand is different from zero, because the difference between $\bx^h$ 
and $\bx^{h+1}$ is only in the $i_h$-th component. So
\begin{equation*}
\sum_{i=1}^m  \Js_i\big(\delta_{h,i}(x^{h}_i - x^{h+1}_i)\big)=
\begin{cases}
\Js_{i_h}(x^{h}_{i_h} - x^{h+1}_{i_h})\ = \bx^h - \bx^{h+1} &\text{if } h \geq k - \ds^k_{i_h}\\
0 &\text{if } h < k - \ds^k_{i_h}.
\end{cases}
\end{equation*}
Therefore setting $J(k) = \big\{h \in \{k-\tau,\dots, k-1\} \,\vert\, h \geq k - \ds^k_{i_h} \big\}$, 
 \eqref{eq:20170913a} yields \eqref{eq:20170918d}.
Note that, since $i_h$ is a random variable,
$J(k)$ is a random set in the sense that
$J(k)(\omega) = \big\{h \in \{k-\tau,\dots, k-1\} \,\vert\, h \geq k - \ds^k_{i_h(\omega)} \big\}$.
\qed

\vspace{2ex}
\noindent
\textbf{Proof of Lemma \ref{p:20170921n}.}
Let $k \in \N$,
let $ p = \mathrm{card} (J(k))$, and let $(h_j)_{1 \leq j \leq p} $ be the elements of $ J(k) $
ordered in (strictly) increasing order. Then, from Lemma \ref{p:20170918c} we have
\begin{equation}
\label{eq:20170918e}
\bx^k - \bhat{\bx}^{k} = \sum_{j = 1}^p  (\bx^{h_j+1} - \bx^{h_j}).
\end{equation}
Let's set, for each $t \in \{0,\dots, p\}$
\begin{equation*}
\bhat{\bx}^{k,t} = \bhat{\bx}^{k} + \sum_{j = 1}^t  (\bx^{h_j+1} - \bx^{h_j}). 
\end{equation*}
Then it follows
\begin{equation*}
\bhat{\bx}^{k,0} = \bhat{\bx}^{k}, \quad
\bhat{\bx}^{k,p} = \bx^k,\quad\text{and}\quad
\forall\, t \geq1\quad\bhat{\bx}^{k,t} - \bhat{\bx}^{k,t-1} 
= \bx^{h_t + 1} - \bx^{h_t}.
\end{equation*}
Therefore
\begin{equation*}
\bx^k - \bhat{\bx}^{k} = \sum_{t= 1}^p  (\bhat{\bx}^{k,t} - \bhat{\bx}^{k,t-1})
\end{equation*}
and $\bhat{\bx}^{k,t}, \bhat{\bx}^{k,t-1} $ differ only in the value of a component. Thus
\begin{align*}
\norm{\nabla f (\bx^k) - \nabla f(\bhat{\bx}^{k})}
&= \Big\lVert \sum_{t=1}^p  \nabla f (\bhat{\bx}^{k,t}) -  \nabla f (\bhat{\bx}^{k,t-1}) \Big\rVert\\
&\leq \sum_{t=1}^p \norm{\nabla f (\bhat{\bx}^{k,t}) - \nabla f(\bhat{\bx}^{k,t-1})}\\
&\leq L_{\mathrm{res}} 
\sum_{t=1}^p \norm{\bhat{\bx}^{k,t} - \bhat{\bx}^{k,t-1}}\\
&= L_{\mathrm{res}} 
\sum_{t=1}^p \norm{\bx^{h_t + 1} - \bx^{h_t}}\\
&= L_{\mathrm{res}} 
\sum_{h \in J(k)} \norm{\bx^{h + 1} - \bx^{h}}.
\end{align*}
from which the result follows.
\qed

\vspace{2ex}
\noindent
\textbf{Proof of Lemma \ref{eq:20210308a}.}
Let $k \in \N$ and $\bxs \in \bHH$. Then
\begin{align}
\scalarp{\nabla f(\bhat{\bx}^k), \bxs - \bx^k} &=  \scalarp{\nabla f(\bhat{\bx}^k), \bxs - \bhat{\bx}^k}  
+ \scalarp{\nabla f(\bhat{\bx}^k),  \bhat{\bx}^k - \bx^k}\nonumber\\
& = \scalarp{\nabla f(\bhat{\bx}^k), \bxs - \bhat{\bx}^k} + 
\sum_{t=0}^{p-1} \scalarp{\nabla f(\bhat{\bx}^k),  \bhat{\bx}^{k,t} - \bhat{\bx}^{k,t+1}}\nonumber\\
& =  \scalarp{\nabla f(\bhat{\bx}^k), \bxs - \bhat{\bx}^k}\nonumber\\
&\qquad + \sum_{t=0}^{p-1} \scalarp{\nabla f(\bhat{\bx}^{k,t}),  \bhat{\bx}^{k,t} - \bhat{\bx}^{k,t+1}}
+ \scalarp{\nabla f(\bhat{\bx}^k) - \nabla f(\bhat{\bx}^{k,t}),  \bhat{\bx}^{k,t} - \bhat{\bx}^{k,t+1}}. \nonumber
\end{align}
Thanks to the convexity of $f$ and \eqref{eq:20170925d}, it follows
\begingroup
\allowdisplaybreaks
\begin{align}
\scalarp{\nabla f(\bhat{\bx}^k), \bxs - \bx^k}
& \leq   f(\bxs) - f(\bhat{\bx}^k)  +
\sum_{t=0}^{p-1} f(\bhat{\bx}^{k,t}) - f(\bhat{\bx}^{k,t+1}) 
+ \frac{L_{\max}}{2} \norm{\bhat{\bx}^{k,t} - \bhat{\bx}^{k,t+1}}^2 \nonumber\\
&\qquad +  \sum_{t =0}^{p-1}
\scalarp{\nabla f(\bhat{\bx}^k) - 
\nabla f(\bhat{\bx}^{k,t}),  \bhat{\bx}^{k,t} - \bhat{\bx}^{k,t+1}}\nonumber\\
& =   f(\bxs) - f(\bx^k) 
+ \frac{L_{\max}}{2}  \sum_{t=0}^{p-1}
\norm{\bhat{\bx}^{k,t} - \bhat{\bx}^{k,t+1}}^2 \nonumber \\
&\qquad+  \sum_{t =0}^{p-1} \sum_{s=0}^{t-1}
\scalarp{\nabla f(\bhat{\bx}^{k,s}) 
- \nabla f(\bhat{\bx}^{k,s+1}),  \bhat{\bx}^{k,t} - \bhat{\bx}^{k,t+1}}\nonumber\\
& \leq    f(\bxs) - f(\bx^k) 
+ \frac{L_{\max}}{2} \sum_{t=0}^{p-1}\norm{\bhat{\bx}^{k,t} - \bhat{\bx}^{k,t+1}}^2\nonumber\\
 &\qquad+ L_{\mathrm{res}} \sum_{t =0}^{p-1} \sum_{s=0}^{t-1}
 \norm{\bhat{\bx}^{k,s} - \bhat{\bx}^{k,s+1}} \norm{\bhat{\bx}^{k,t} - \bhat{\bx}^{k,t+1}} \nonumber.
\end{align}
\endgroup
Using the equality of the square of sum, Holder inequality and $L_{\max} \leq L_{\mathrm{res}}$, we finally get
\begin{align}
\scalarp{\nabla f(\bhat{\bx}^k), \bxs - \bx^k}
& \leq   f(\bxs) - f(\bx^k) 
+ \frac{L_{\max}}{2} \sum_{t=0}^{p-1}\norm{\bhat{\bx}^{k,t} - \bhat{\bx}^{k,t+1}}^2 \nonumber\\
&\qquad+ \frac{L_{\mathrm{res}}}{2} \bigg[
\bigg(\sum_{t =0}^{p-1} \norm{\bhat{\bx}^{k,t} - \bhat{\bx}^{k,t+1}} \bigg)^2
- \sum_{t =0}^{p-1} \norm{\bhat{\bx}^{k,t} - \bhat{\bx}^{k,t+1}}^2 \bigg]\nonumber\\
& =   f(\bxs) - f(\bx^k) 
+ \frac{L_{\mathrm{res}}}{2}  \bigg(\sum_{t =0}^{p-1} 
\norm{\bhat{\bx}^{k,t} - \bhat{\bx}^{k,t+1}} \bigg)^2 \nonumber \\
&\qquad+ \left(\frac{L_{\max}}{2} - \frac{L_{\mathrm{res}}}{2} \right) \sum_{t=0}^{p-1}\norm{\bhat{\bx}^{k,t} - \bhat{\bx}^{k,t+1}}^2\nonumber\\
&\leq  f(\bxs) - f(\bx^k) 
+ \frac{\tau L_{\mathrm{res}}}{2} \sum_{h \in J(k)} \norm{\bx^{h} - \bx^{h+1}}^2. \nonumber
\end{align}
The statement follows. \qed

\vspace{2ex}
\noindent
\textbf{Proof of Lemma \ref{lem:20200126a}.}
Let $\bzz \in \bHH$. It follows from the definition of $\bxx^{+}$ that $\bxx-\bxx^{+}-\nabla \varphi(\hat{\bxx}) \in \partial \psi\left(\bxx^{+}\right) .$ Therefore, $\psi(\bzz) \geq \psi\left(\bxx^{+}\right)+\left\langle \bxx-\bxx^{+}-\nabla \varphi(\hat{\bxx}), \bzz -\bxx^{+}\right\rangle,$ hence
\begin{equation*}
\left\langle \bxx-\bxx^{+}, \bzz-\bxx^{+}\right\rangle \leq \psi(\bzz)-\psi\left(\bxx^{+}\right)+\left\langle\nabla \varphi(\hat{\bxx}), \bzz-\bxx^{+}\right\rangle.
\end{equation*}
Then,
\begin{equation*}
\begin{array}{r}
\left\langle \bxx-\bxx^{+}, \bzz-\bxx\right\rangle+\left\langle \bxx-\bxx^{+}, \bxx-\bxx^{+}\right\rangle \leq \psi(\bzz)-\psi\left(\bxx^{+}\right)+\langle\nabla \varphi(\hat{\bxx}), \bzz-\bxx\rangle+\left\langle\nabla \varphi(\hat{\bxx}), \bxx-\bxx^{+}\right\rangle.
\end{array}
\end{equation*}
Rearranging the terms the statement follows.
\qed

\section{Proofs of Section~\ref{sec:convergence}}
\label{sec:appB}

\vspace{2ex}
\noindent
\textbf{Proof of Lemma \ref{lem:20210207}.}
Let $k \in \N$.
We have, from Cauchy-Schwarz inequality, the
Young inequality and Remark~\ref{rmk:20210618a}, that
\begingroup
\allowdisplaybreaks
\begin{align*}
\langle\nabla f(\bx^k)&-\nabla f(\bhat{\bx}^k),\bbar{\bx}^{k+1}-\bx^k \rangle_{\bVV} \\[1ex] 
&\leq L_{\mathrm{res}}^{\bVV} \sum_{h \in J(k)} \|\bx^{h+1} - \bx^{h}\|_{\bVV} \|\bbar{\bx}^{k+1} - \bx^k\|_{\bVV} \\
&\leq \frac{1}{2}\left[\frac{(L_{\mathrm{res}}^{\bVV})^2}{s}\bigg(\sum_{h \in J(k)}
\|\bx^{h+1} - \bx^{h}\|_{\bVV}\bigg)^2 
+s\| \bbar{\bx}^{k+1} - \bx^k\|_{\bVV}^2\right] \\
&\leq \frac{1}{2}\left[\frac{\tau (L_{\mathrm{res}}^{\bVV})^2}{s}\left(\sum_{h = k-\tau}^{k-1}
\|\bx^{h+1} - \bx^{h}\|_{\bVV}^2\right)
+s\| \bbar{\bx}^{k+1} - \bx^k\|_{\bVV}^2\right] \\
&= \frac{s}{2}\|\bbar{\bx}^{k+1} - \bx^k\|_{\bVV}^2 
+ \frac{\tau (L_{\mathrm{res}}^{\bVV})^2}{2s}\sum_{h = k-\tau}^{k-1} 
\|\bx^{h+1} - \bx^{h}\|_{\bVV}^2, 
\end{align*}
\endgroup
Now, thanks to a decomposition of the last term by Fact \ref{fact:decomp}, we obtain 
\begin{align*}
\langle\nabla f(\bx^k)& - \nabla f(\bhat{\bx}^k), \bbar{\bx}^{k+1} - \bx^k\rangle_{\bVV}\\
&\leq \frac{s}{2}\|\bbar{\bx}^{k+1} - \bx^k\|_{\bVV}^2 
+ \frac{\tau (L_{\mathrm{res}}^{\bVV})^2}{2s} \sum_{h = k-\tau}^{k-1} 
(h- (k-\tau)+1)\|\bx^{h+1} - \bx^{h}\|_{\bVV}^2 \\
&\qquad - \frac{\tau (L_{\mathrm{res}}^{\bVV})^2}{2s} \sum_{h=k-\tau+1}^{k}
(h-(k-\tau))\|\bx^{h+1} - \bx^{h}\|_{\bVV}^2 \\
&\qquad + \frac{\tau^2 (L_{\mathrm{res}}^{\bVV})^2}{2s}
\|\bx^{k+1} - \bx^{k}\|_{\bVV}^2.
\end{align*}
We recall that $\|\bx^{k+1} - \bx^{k}\|_{\bVV}^2 = \pp_{i_{k}}|\bar{x}^{k+1}_{i_{k}} - x^k_{i_{k}}|^2$. So taking 
$$\displaystyle \uuu_k = \frac{\tau (L_{\mathrm{res}}^{\bVV})^2}{2s}\sum_{h = k-\tau}^{k-1} (h-(k-\tau)+1)\|\bx^{h+1} - \bx^{h}\|_{\bVV}^2,$$ 
we get
\begin{align*}
    \EE\big[ \langle\nabla f(\bx^k) & - \nabla f(\bhat{\bx}^k), 
    \bar{x}^{k+1} - x^k\rangle_{\bVV} \,\big\vert\,i_0,\dots,i_{k-1}\big] \\
    &\leq \frac{s}{2}\|\bbar{\bx}^{k+1} - \bx^k\|_{\bVV}^2 + \frac{\tau^2 (L_{\mathrm{res}}^{\bVV})^2}{2s}\sum_{i=0}^m \pp_i^2 |\bar{x_i}^{k+1} - x_i^k|^2 + \uuu_k - \EE\big[\uuu_{k+1}\,\big\vert\,i_0,\dots,i_{k-1}\big] 
\end{align*}
Meaning
\begin{align*}
    \langle\nabla f(\bx^k) & - \nabla f(\bhat{\bx}^k), 
    \bar{x}^{k+1} - x^k\rangle_{\bVV} \\
    &\leq \sum_{i=0}^m \pp_i\left(\frac{s}{2} + \frac{\tau^2 (L_{\mathrm{res}}^{\bVV})^2}{2s}\pp_i \right) |\bar{x_i}^{k+1} - x_i^k|^2 + \uuu_k - \EE\big[\uuu_{k+1}\,\big\vert\,i_0,\dots,i_{k-1}\big]\\
    &\leq \sum_{i=0}^m \pp_i\left(\frac{s}{2} + \frac{\tau^2 (L_{\mathrm{res}}^{\bVV})^2}{2s}\pp_{\max} \right) |\bar{x_i}^{k+1} - x_i^k|^2 + \uuu_k - \EE\big[\uuu_{k+1}\,\big\vert\,i_0,\dots,i_{k-1}\big].
\end{align*}
By minimizing $\displaystyle s \mapsto \left(\frac{s}{2} + \frac{\tau^2 (L_{\mathrm{res}}^{\bVV})^2}{2s}\pp_{\max} \right)$, we find $s=\tau L_{\mathrm{res}}^{\bVV} \sqrt{\pp_{\max}}$. We then get
\begin{align*}
    \langle\nabla f(\bx^k) & - \nabla f(\bhat{\bx}^k), 
    \bar{x}^{k+1} - x^k\rangle_{\bVV} \\
    &\leq \tau L_{\mathrm{res}}^{\bVV} \sqrt{\pp_{\max}} \sum_{i=0}^m \pp_i |\bar{x_i}^{k+1} - x_i^k|^2 + \uuu_k - \EE\big[\uuu_{k+1}\,\big\vert\,i_0,\dots,i_{k-1}\big],
\end{align*}
and $\displaystyle \uuu_k = \frac{L_{\mathrm{res}}^{\bVV}}{2\sqrt{\pp_{\max}}}\sum_{h = k-\tau}^{k-1} (h-(k-\tau)+1)\|\bx^{h+1} - \bx^{h}\|_{\bVV}^2$.
\qed

\vspace{2ex}
\noindent
\textbf{Proof of Lemma \ref{p:20190313c}.}
We have
\begin{equation}
\norm{\bx^{\iter+1} - \bx}_\bWW^2
= \sum_{i=1}^m \frac{1}{\pp_i \gamma_i} \abs{x_i^{k+1}- x_i}^2
= \frac{1}{\pp_{i_k} \gamma_{i_k}} \abs{\bbar{x}_{i_k}^{k+1}- x_{i_k}}^2
+ \norm{\bx^{\iter} - \bx}_\bWW^2 - \frac{1}{\pp_{i_k} \gamma_{i_k}} \abs{x_{i_k}^{k}- x_{i_k}}^2.
\end{equation}
Thus, taking the conditional expectation we have
\begin{equation}
\EE[\norm{\bx^{\iter+1} - \bx}_\bWW^2 \,\vert\, i_0,\dots, i_{k-1}]
= \norm{\bbar{\bx}^{\iter+1} - \bx}_{\bGammas^{-1}}^2 + \norm{\bx^{\iter} - \bx}_\bWW^2
- \norm{\bx^{\iter} - \bx}_{\bGammas^{-1}}^2
\end{equation}
and \eqref{eq:20190313c} follows.
The second equation follows from \eqref{eq:20190313c}, by choosing $\bx = \bx^\iter$.
\qed

\vspace{2ex}
\noindent
\textbf{Proof of Proposition~\ref{prop:20211112a}.}
Let $k \in \N$.
We have from the descent lemma along the $i_{k}$-th block-coordinate,
\begin{align*}
 F(\bx^{k+1})&\leq  f(\bx^k) + \langle \nabla_{i_{k}} f(\bx^k), \bar{x}^{k+1}_{i_{k}} - x^k_{i_{k}} \rangle
+ \frac{L_{i_{k}}}{2} \abs{\bar{x}^{k+1}_{i_{k}} - x^k_{i_{k}}}^2 +\sum_{i=1}^n g_i(x^{k+1}_i)\\
&= f(\bx^k) + \langle \nabla_{i_{k}} f(\bx^k), \bar{x}^{k+1}_{i_{k}} - x^k_{i_{k}} \rangle
+ \frac{L_{i_{k}}}{2} \abs{\bar{x}^{k+1}_{i_{k}} - x^k_{i_{k}}}^2 + \Big( g_{i_{k}}(x^{k+1}_{i_{k}}) + \sum_{i\neq i_{k}}^n g_i(x^{k}_i) \Big) \\
&= f(\bx^k) + \langle \nabla_{i_{k}} f(\bx^k), \bar{x}^{k+1}_{i_{k}} - x^k_{i_{k}} \rangle
+ \frac{L_{i_{k}}}{2} \abs{\bar{x}^{k+1}_{i_{k}} - x^k_{i_{k}}}^2 \\
&\qquad + \Big( g_{i_{k}}(x^{k+1}_{i_{k}}) - g_{i_{k}}(x^k_{i_{k}}) + g(\bx^k) \Big) \\
&= F(\bx^k) + \langle \nabla_{i_{k}} f(\bx^k), \bar{x}^{k+1}_{i_{k}} - x^k_{i_{k}} \rangle
+ \frac{L_{i_{k}}}{2} \abs{\bar{x}^{k+1}_{i_{k}} - x^k_{i_{k}}}^2 + \Big( g_{i_{k}}(\bar{x}^{k+1}_{i_{k}}) - g_{i_{k}}(x^k_{i_{k}}) \Big) \\
&= F(\bx^k) + \langle \nabla_{i_{k}} f(\bx^k) - \nabla_{i_{k}} f(\bhat{\bx}^k), \bar{x}^{k+1}_{i_{k}} - x^k_{i_{k}} \rangle + \frac{L_{i_{k}}}{2} \abs{\bar{x}^{k+1}_{i_{k}} - x^k_{i_{k}}}^2 \\
&\qquad + \Big( \langle \nabla_{i_{k}} f(\bhat{\bx}^k), \bar{x}^{k+1}_{i_{k}} - x^k_{i_{k}} \rangle + g_{i_{k}}(\bar{x}^{k+1}_{i_{k}}) - g_{i_{k}}(x^k_{i_{k}}) \Big).
\end{align*}
From \eqref{eq:20170925a}, we can write that
\begin{align}\label{eq:20210623a}
F(\bx^{k+1})&\leq F(\bx^k) + \langle \nabla_{i_{k}} f(\bx^k) - \nabla_{i_{k}} f(\bhat{\bx}^k), \bar{x}^{k+1}_{i_{k}} - x^k_{i_{k}} \rangle - \bigg(\frac{1}{\gamma_{i_{k}}}-\frac{L_{i_{k}}}{2}\bigg)|\bar{x}^{k+1}_{i_{k}} - x^k_{i_{k}}|^2 
\end{align}
By taking the conditional expectation and using Fact \ref{fact:durrett}, it follows:
\begingroup
\allowdisplaybreaks
\begin{align}
    \EE\big[   F(\bx^{k+1})\,&\big\vert\,i_0,\dots,i_{k-1}\big] \nonumber \\
    &\leq F(\bx^k) + \EE\big[ \langle \nabla_{i_{k}} f(\bx^k) - \nabla_{i_{k}} f(\bhat{\bx}^k), \bar{x}^{k+1}_{i_{k}} - x^k_{i_{k}} \rangle\,\big\vert\,i_0,\dots,i_{k-1}\big] \nonumber\\
    &\qquad - \sum_{i=0}^m \pp_i \Big(\frac{1}{\gamma_{i}}-\frac{L_{i}}{2}\Big)|\bar{x}^{k+1}_{i} - x^k_{i}|^2 \nonumber\\
    &= F(\bx^k) + \sum_{i=0}^m \pp_i \langle \nabla_{i} f(\bx^k) - \nabla_{i} f(\bhat{\bx}^k), \bar{x}^{k+1}_{i} - x^k_{i} \rangle \nonumber\\
    &\qquad - \sum_{i=0}^m \pp_i \Big(\frac{1}{\gamma_{i}}-\frac{L_{i}}{2}\Big)|\bar{x}^{k+1}_{i} - x^k_{i}|^2 \nonumber\\
    &= F(\bx^k) + \langle \nabla f(\bx^k) - \nabla f(\bhat{\bx}^k), \bbar{\bx}^{k+1} - \bx^k \rangle_{\bVV} \nonumber\\
    &\label{eq:20201223c}\qquad - \sum_{i=0}^m \pp_i \Big(\frac{1}{\gamma_{i}}-\frac{L_{i}}{2}\Big)|\bar{x}^{k+1}_{i} - x^k_{i}|^2.
\end{align}
\endgroup
From Lemma \ref{lem:20210207}, we have
\begin{align*}
    \langle\nabla f(\bx^k) & - \nabla f(\bhat{\bx}^k), 
    \bar{x}^{k+1} - x^k\rangle_{\bVV} \\
    &\leq \tau L_{\mathrm{res}}^{\bVV} \sqrt{\pp_{\max}} \sum_{i=0}^m \pp_i |\bar{x_i}^{k+1} - x_i^k|^2 + \uuu_k - \EE\big[\uuu_{k+1}\,\big\vert\,i_0,\dots,i_{k-1}\big],
\end{align*}
with $\displaystyle \uuu_k = \frac{L_{\mathrm{res}}^{\bVV}}{2\sqrt{\pp_{\max}}}\sum_{h = k-\tau}^{k-1} (h-(k-\tau)+1)\norm{\bx^{h+1} - \bx^{h}}_{\bVV}^2$.
We then plug this result in \eqref{eq:20201223c} obtaining
\begin{align*}
    \sum_{i=0}^m \pp_i \Big(\frac{1}{\gamma_{i}}-\frac{L_{i}}{2}\Big)
    \abs{\bar{x}^{k+1}_{i} - x^k_{i}}^2 \nonumber
    &\leq F(x^k) + \uuu_k +  \tau L_{\mathrm{res}}^{\bVV} \sqrt{\pp_{\max}}\sum_{i=0}^m \pp_i |\bar{x_i}^{k+1} - x_i^k|^2 \\
    &\qquad - \EE\big[F(\bx^{k+1})+ \uuu_{k+1}\,\big\vert\,i_0,\dots,i_{k-1}\big] .
\end{align*}
Hence
\begin{align*}
    \sum_{i=0}^m \pp_i\left(\frac{1}{\gamma_{i}}-\frac{L_{i}}{2} - \tau L_{\mathrm{res}}^{\bVV} \sqrt{\pp_{\max}} \right)|\bar{x}^{k+1}_{i} - x^k_{i}|^2
    \leq F(\bx^k) + \uuu_k - \EE\big[ F(\bx^{k+1})+ \uuu_{k+1}\,\big\vert\,i_0,\dots,i_{k-1}\big].
\end{align*}
Since $\delta < 2$, recalling \eqref{eq:20210618b}, we have, for all $i \in [m]$, 
\begin{equation*}
\left(\frac{1}{\gamma_{i}}-\frac{L_{i}}{2} - \tau L_{\mathrm{res}}^{\bVV} \sqrt{\pp_{\max}} \right)=
\frac{1}{2\gamma_i}(2 -L_{i}\gamma_i - 2\gamma_i \tau L_{\mathrm{res}}^{\bVV} \sqrt{\pp_{\max}}) 
\geq \frac{1}{2\gamma_i}(2 - \delta) > 0. 
\end{equation*}
Therefore the statement follows.
\qed

\vspace{2ex}
\noindent
\textbf{Proof of Proposition~\ref{prop:20210104}.}
Let $k \in \N$ and $\bxs \in \bHH$.
Since $\langle\nabla f(\bhat{\bx}^k), \bxs-\bx^{k}\rangle 
= \langle\nabla^{\bGammas^{-1}} f(\bhat{\bx}^k), \bxs-\bx^{k}\rangle_{\bGammas^{-1}}$ and
$\bar{\bx}^{k+1} = \prox^{\bGammas^{-1}}_g \big( \bx^{k}
-  \nabla^{\bGammas^{-1}} f(\bhat{\bx}^k) \big)$, 
we derive from Lemma \ref{lem:20200126a} above written in weighted norm that
\begin{align}
\label{eq:20210103a}
\langle\bx^{k}-\bbar{\bx}^{k+1}, \bxs - \bx^{k}\rangle_{\bGammas^{-1}} &\leq
g(\bxs)-g(\bx^{k})+\langle\nabla f(\bhat{\bx}^k), \bxs -\bx^{k}\rangle \nonumber \\
&\qquad +g(\bx^{k}) -g(\bbar{\bx}^{k+1}) 
+\langle\nabla f(\bhat{\bx}^k), \bx^{k}-\bbar{\bx}^{k+1}\rangle \nonumber \\
&\qquad -\lVert \bx^{k}-\bbar{\bx}^{k+1}\rVert^{2}_{\bGammas^{-1}}.
\end{align}
From Lemma \ref{eq:20210308a}, we have
\begin{align*}
\langle\nabla f(\bhat{\bx}^k), \bxs-\bx^{k}\rangle &\leq f(\bxs) - f(\bx^k) 
+ \frac{\tau L_{\mathrm{res}}}{2}\sum_{h \in J(k)} \|\bx^h - \bx^{h+1}\|^2.
\end{align*}
So \eqref{eq:20210103a} becomes
\begin{align}
\label{eq:20210103b}
\langle\bx^{k}-\bbar{\bx}^{k+1}, \bxs - \bx^{k}\rangle_{\bGammas^{-1}} 
&\leq F(\bxs)-F(\bx^{k})+ \frac{\tau L_{\mathrm{res}}}{2}\sum_{h \in J(k)} \|\bx^h - \bx^{h+1}\|^2 
\nonumber \\
&\qquad + g(\bx^{k}) - g(\bbar{\bx}^{k+1}) 
+ \langle\nabla f(\bhat{\bx}^k), \bx^{k}-\bbar{\bx}^{k+1}\rangle \nonumber \\[1ex]
&\qquad -\|\bx^{k}-\bbar{\bx}^{k+1}\|^{2}_{\bGammas^{-1}}.
\end{align}
Next, recalling that $x^k$ and $x^{k+1}$ differs only in the $i_k$-th component, we have
\begin{align*}
g(\bx^{k})-g(\bbar{\bx}^{k+1}) &+\langle\nabla f(\bhat{\bx}^{k}), \bx^{k}-\bbar{\bx}^{k+1}\rangle \\
&=\EE\left[\sum_{i=1}^{m} \frac{1}{\pp_{i}}\big(g_{i}(x_{i}^{k})-g_{i}(x_{i}^{k+1})+\langle\nabla_{i}f(\bhat{\bx}^{k}), x_{i}^{k}-x_{i}^{k+1}\rangle\big)\,\vert\, i_0,\dots, i_{k-1}\right]
\end{align*}
Moreover,
\begin{align*}
\sum_{i=1}^{m}& \frac{1}{\pp_{i}}\big(g_{i}(x_{i}^{k})-g_{i}(x_{i}^{k+1})+\langle\nabla_{i} f(\bhat{\bx}^{k}), x_{i}^{k}-x_{i}^{k+1}\rangle\big) \\
&=\frac{1}{\pp_{\min} }\big(g(\bx^{k})-g(\bx^{k+1})+\langle\nabla f(\bhat{\bx}^{k}), \bx^{k}-\bx^{k+1}\rangle\big) \\
&\quad-\sum_{i=1}^{m}(\underbrace{\frac{1}{\pp_{\min}}-\frac{1}{\pp_{i}}}_{\geq 0})
\big(g_{i}(x_{i}^{k})-g_{i}(x_{i}^{k+1})+\langle\nabla_{i} f(\bhat{\bx}^{k}), x_{i}^{k}-x_{i}^{k+1}\rangle\big) \\
&\leq \frac{1}{\pp_{\min} }\big(g(\bx^{k})-g(\bx^{k+1})+\langle\nabla f(\bhat{\bx}^{k}), \bx^{k}-\bx^{k+1}\rangle\big) \\
&\quad-\left(\frac{1}{\pp_{\min }}-\frac{1}{\pp_{i_k}}\right) \frac{1}{\gamma_{i_k}}|\Delta_{i_k}^{k}|^{2}
\end{align*}
where in the last inequality we used that 
\begin{equation*}
-\left(g_{i_k}(x_{i_k}^{k})-g_{i_k}(x_{i_k}^{k+1})+\langle\nabla_{i_k} f(\bhat{\bx}^{k}), x_{i_k}^{k}-x_{i_k}^{k+1}\rangle\right) \leq-\frac{1}{\gamma_{i_k}}|\Delta_{i_k}^{k}|^{2},
\end{equation*}
which was derived from \eqref{eq:20170925a}.
So
\begin{align*}
g(\bx^{k})-g(\bbar{\bx}^{k+1}) &+ \langle\nabla f(\bhat{\bx}^{k}), \bx^{k}-\bbar{\bx}^{k+1}\rangle \\
&\leq \frac{1}{\pp_{\min}} \EE\big[g(\bx^{k})-g(\bx^{k+1})+\langle\nabla f(\bhat{\bx}^{k}), \bx^{k}-\bx^{k+1}\rangle\,\big\vert\, i_0,\dots, i_{k-1}\big] \\
&\quad - \frac{1}{\pp_{\min}}\sum_{i=1}^{m}\frac{\pp_i}{\gamma_i} |\Delta_{i}^{k}|^{2} + \|\bx^k - \bar{\bx}^{k+1}\|^2_{\bGammas^{-1}} .
\end{align*}
Now, by Lemma \ref{lem:20210207} and the 
block-coordinate descent lemma \eqref{eq:20170921m}, we have
\begin{align*}
\EE[&\langle\nabla f(\bhat{\bx}^{k}), \bx^{k}-\bx^{k+1}\rangle\,\vert\, i_0,\dots, i_{k-1}]\\[1ex]
&\leq \EE\big[\langle\nabla f(\bhat{\bx}^{k}) 
- \nabla f(\bx^k), \bx^{k}-\bx^{k+1}\rangle\,\big\vert\, i_0,\dots, i_{k-1}\big] 
+ \EE\big[\langle\nabla f(\bx^{k}), \bx^{k}-\bx^{k+1}\rangle\,\big\vert\, i_0,\dots, i_{k-1}\big]\\[1ex]
&= \langle\nabla f(\bhat{\bx}^{k}) - \nabla f(\bx^k), \bx^{k}-\bbar{\bx}^{k+1}\rangle_{\bVV} 
+ \EE[\langle\nabla f(\bx^{k}), \bx^{k}-\bx^{k+1}\rangle\,\vert\, i_0,\dots, i_{k-1}]\\
&\leq \tau L_{\mathrm{res}}^{\bVV} \sqrt{\pp_{\max}} \sum_{i=0}^m \pp_i \lvert\bar{x_i}^{k+1} 
- x_i^k\rvert^2 + \uuu_k - \EE\big[\uuu_{k+1}\,\big\vert\,i_0,\dots,i_{k-1}\big] \\ 
&\qquad + \EE\Big[f(\bx^{k})-f(\bx^{k+1})
+ \frac{L_{i_k}}{2}|\Delta^k_{i_k}|^2\,\Big\vert\, i_0,\dots, i_{k-1}\Big],
\end{align*}
where $\uuu_k = L_{\mathrm{res}}^{\bVV}/(2\sqrt{\pp_{\max}})\sum_{h = k-\tau}^{k-1} (h-(k-\tau)+1)\|\bx^{h+1} - \bx^{h}\|_{\bVV}^2$ for all $k \in \mathbb{N}$.
Therefore
\begin{align}
\nonumber
g(\bx^{k})-g(\bbar{\bx}^{k+1}) &+ \langle\nabla f(\bhat{\bx}^{k}), \bx^{k}-\bbar{\bx}^{k+1}\rangle \\
\nonumber&\leq \frac{1}{\pp_{\min}} \EE[F(\bx^{k})+\uuu_k-F(x^{k+1})-\uuu_{k+1}\,\vert\, i_0,\dots, i_{k-1}] \\
\label{eq:20210618c}&\quad + \frac{1}{\pp_{\min}}\sum_{i=1}^{m}\pp_i\left(\frac{L_i}{2} + \tau L_{\mathrm{res}}^{\bVV} \sqrt{\pp_{\max}} - \frac{1}{\gamma_i} \right) |\Delta_{i}^{k}|^{2} + \|\bx^k - \bbar{\bx}^{k+1}\|^2_{\bGammas^{-1}}.
\end{align}
Since $\gamma_i L_i + 2 \gamma_i \tau L_{\mathrm{res}}^{\bVV} \sqrt{\pp_{\max}} \leq \delta < 2$, 
we have 
\begin{equation*}
\frac{L_i}{2} + \tau L_{\mathrm{res}}^{\bVV} \sqrt{\pp_{\max}} - \frac{1}{\gamma_i} 
= \frac{1}{2 \gamma_i}
(\gamma_i L_i + 2 \gamma_i \tau L_{\mathrm{res}}^{\bVV} \sqrt{\pp_{\max}} - 2)< 0,
\end{equation*}
 and hence \eqref{eq:20210618c} yields
\begin{align*}
g(\bx^{k})-g(\bbar{\bx}^{k+1}) &+ \langle\nabla f(\bhat{\bx}^{k}), \bx^{k}-\bbar{\bx}^{k+1}\rangle \\
&\leq \frac{1}{\pp_{\min}} \EE[F(\bx^{k})+\uuu_k-F(x^{k+1})-\uuu_{k+1}\,\vert\, i_0,\dots, i_{k-1}] \\
&\quad +\frac{\delta - 2 }{2} \sum_{i=1}^{m}\frac{1}{\gamma_i}|\Delta_{i}^{k}|^{2} + \|\bx^k - \bbar{\bx}^{k+1}\|^2_{\bGammas^{-1}}.
\end{align*}
The statement follows from \eqref{eq:20210103b}.
\qed

\vspace{2ex}
\noindent
\textbf{Proof of Proposition~\ref{prop:20210519}.}
We know that
\begin{align*}
\norm{\bx^{k}-\bxs}^{2}_{\bGammas^{-1}} - \norm{\bbar{\bx}^{k+1}-\bxs}^{2}_{\bGammas^{-1}} 
&=-\norm{\bx^{k}-\bbar{\bx}^{k+1}}^{2}_{\bGammas^{-1}} 
+2 \langle \bx^{k}-\bbar{\bx}^{k+1}, \bx^{k}-\bxs\rangle_{\bGammas^{-1}}.
\end{align*}
We derive from Proposition~\ref{prop:20210104}, multiplied by 2, that
\begin{align}\label{eq:20210105c}
\norm{\bbar{\bx}^{k+1}-\bxs}^{2}_{\bGammas^{-1}} 
&\leq \norm{\bx^{k}-\bxs}^{2}_{\bGammas^{-1}} \nonumber \\
&\qquad + \frac{2}{\pp_{\min }}
\EE\big[F(\bx^k) + \uuu_k -F(\bx^{k+1}) - \uuu_{k+1} \,\vert\, i_0,\dots, i_{k-1}\big] \nonumber \\
&\qquad + 2(F(\bxs) - F(\bx^k)) 
+ \tau L_{\mathrm{res}}\sum_{h \in J(k)} \|\bx^h - \bx^{h+1}\|^2 \nonumber \\
&\qquad (\delta-1)\|\bx^k - \bbar{\bx}^{k+1}\|^2_{\bGammas^{-1}}.
\end{align}
where $\uuu_k = L_{\mathrm{res}}^{\bVV}/(2\sqrt{\pp_{\max}})
\sum_{h = k-\tau}^{k-1} (h-(k-\tau)+1)\norm{\bx^{h+1} - \bx^{h}}_{\bVV}^2$. 
It follows from Lemma~\ref{p:20190313c}  that 
\begin{align}\label{eq:20210401b}
\EE\big[\|\bx^{k+1}-\bxs\|^{2}_{\bWW}\,&\vert\, i_0,\dots, i_{k-1}\big] \nonumber\\
&\leq \|\bx^{k}-\bxs\|^{2}_{\bWW} \nonumber \\
&\quad+(\delta-1) \|\bx^k - \bbar{\bx}^{k+1}\|^2_{\bGammas^{-1}} \nonumber \\
&\quad + \frac{2}{\pp_{\min }}
\EE\big[F(\bx^k)+\uuu_k-F(\bx^{k+1})-\uuu_{k+1}\,\vert\, i_0,\dots, i_{k-1}\big] \nonumber \\
&\quad + 2(F(\bxs) - F(\bx^k)) + \tau L_{\mathrm{res}}\sum_{h \in J(k)} \|\bx^h - \bx^{h+1}\|^2.
\end{align}
Plugging \eqref{eq:20210106a} in \eqref{eq:20210401b} the statement follows.
\qed

\vspace{2ex}
\noindent
\textbf{Proof of Proposition~\ref{prop:20200126a}.}
Let $k \in \N$ and $\bxs \in \bHH$.
From Proposition \ref{prop:20210519}, we have
\begin{align*}
\EE\big[&\|\bx^{k+1}-\bxs\|^{2}_{\bWW}\,\vert\, i_0,\dots, i_{k-1}\big]  \\
&\leq \|\bx^{k}-\bxs\|^{2}_{\bWW} \\
&\quad + \frac{2}{\pp_{\min }}\left(\frac{(\delta-1)_+}{2-\delta}+ 1\right) 
\EE\big[F(\bx^k)+\uuu_k-F(\bx^{k+1})-\uuu_{k+1}\,\vert\, i_0,\dots, i_{k-1}\big]  \\
&\quad + \tau L_{\mathrm{res}} \sum_{h \in J(k)} \|\bx^h - \bx^{h+1}\|^2 \\
&\quad + 2(F(\bxs) - \EE\big[F(\bx^{k+1})+\uuu_{k+1}\,\vert\, i_0,\dots, i_{k-1}\big]). \\[1ex]
&\quad - 2(\EE\big[ F(\bx^{k})+\uuu_{k} - F(\bx^{k+1})-\uuu_{k+1}\,\vert\, i_0,\dots, i_{k-1}\big])
+ 2 \uuu_k
\end{align*}
Set for all $k \in \mathbb{N}$,
\begin{align*}
    \xi_k &= 2\left(\frac{\max\{1,(2-\delta)^{-1}\}}{\pp_{\min }}- 1\right) \EE\big[F(\bx^k)+\uuu_k-F(\bx^{k+1})-\uuu_{k+1}\,\vert\, i_0,\dots, i_{k-1}\big] \\
    &\quad + \tau L_{\mathrm{res}}\sum_{h \in J(k)} \|\bx^h - \bx^{h+1}\|^2 + 2 \uuu_k.
\end{align*}
Now, on the one hand, recalling \eqref{eq:20201223m}, \eqref{eq:20210106a}
and Lemma~\ref{p:20190313c},
we have
\begin{align*}
\EE\bigg[\sum_{k \in \mathbb{N}} \sum_{h \in J(k)} \|\bx^h - \bx^{h+1}\|^2 \bigg]
&\leq \tau \gamma_{\max} \pp_{\max}  \sum_{k \in \mathbb{N}} \EE [\|\bx^k - \bx^{k+1}\|^2_{\bWW}]\\
& \leq \frac{2 \tau \gamma_{\max} \pp_{\max}}{(2 - \delta) \pp_{\min}} \sum_{k \in \N}
 \big( \EE[ F(\bx^{k})+ \uuu_{k}] - \EE[F(\bx^{k+1})- \uuu_{k+1}] \big)\\
& \leq \frac{2 \tau \gamma_{\max} \pp_{\max}}{(2 - \delta) \pp_{\min}} (F(\bx^0) + \uuu_0 - F^*)<+\infty
\end{align*}
Recalling the definition of $\uuu_k$ in Proposition~\ref{prop:20210104} and of 
$L_{\mathrm{res}}^{\bVV}$ in Remark~\ref{rmk:20210618a}, this also yields
\begin{align*}
\EE\bigg[\sum_{k \in \N} \uuu_k\bigg] &\leq \frac{\tau L_{\mathrm{res}}^{\bVV}}{2\sqrt{\pp_{\max}}}
\EE\bigg[\sum_{k \in \mathbb{N}} \sum_{h = k-\tau}^{k-1} \norm{\bx^h - \bx^{h+1}}_{\bVV}^2 \bigg] \\
&\leq \frac{\tau L_{\mathrm{res}}^{\bVV} \pp_{\max}}{2\sqrt{\pp_{\max}}}
\EE\bigg[\sum_{k \in \mathbb{N}} \sum_{h = k-\tau}^{k-1} \norm{\bx^h - \bx^{h+1}}^2 \bigg]\\
&\leq \frac{\tau L_{\mathrm{res}} \pp_{\max}^2}{\sqrt{\pp_{\min}}}
\frac{\tau \gamma_{\max}}{(2 - \delta) \pp_{\min}} (F(\bx^0) + \uuu_0 - F^*).
\end{align*}
On the other hand, setting $\eta_k = F(\bx^k)+\uuu_k - \EE\big[F(\bx^{k+1})-\uuu_{k+1}\,\vert\, i_0,\dots, i_{k-1}\big]$, which in virtue of \eqref{eq:20210106a} is positive $\PP$-a.s., we have
\begin{equation*}
\EE\bigg[\sum_{k \in \N} \eta_k\bigg] = \sum_{k \in \N} \EE[\eta_k]
= \sup_{n \in \N} \sum_{k=0}^n \EE[F(\bx^k)+\uuu_k] - \EE[F(\bx^{k+1})-\uuu_{k+1}]
\leq F(\bx^0)+\uuu_0 - F^*<+\infty.
\end{equation*}
Let $\displaystyle C = \frac{\max\left\{1,(2-\delta)^{-1}\right\}}{\pp_{\min}} -1
+ \tau^2\frac{L_{\mathrm{res}}\gamma_{\max}\pp_{\max}}{\pp_{\min}(2-\delta)}
\left( 1 + \frac{\pp_{\max}}{\sqrt{\pp_{\min}}}\right)$. We then get
\begin{equation*}
    \sum_{k \in \N}\EE[\xi_k] \leq 2 C(F(\bx^0)  - F^*).
\end{equation*}
We remark that $(\forall\, i \in [m])\quad\gamma_i( L_{i} + 2\tau L_{\mathrm{res}}\pp_{\max}/ \sqrt{\pp_{\min}})< 2$. So $\gamma_i \tau L_{\mathrm{res}}< \frac{2-\gamma_i L_{i}}{2} \frac{\sqrt{\pp_{\min}}}{\pp_{\max}}$. This implies $\tau \gamma_{\max} L_{\mathrm{res}}< \frac{2-\gamma_{\max} L_{i_0}}{2} \frac{\sqrt{\pp_{\min}}}{\pp_{\max}}$, where $i_0 \in [m]$ such that $\gamma_{i_0} = \gamma_{\max}$. Thus
\begin{equation}\label{eq:20211209a}
    \tau \gamma_{\max} L_{\mathrm{res}} < \frac{2-\gamma_{\max} L_{\min}}{2} \frac{\sqrt{\pp_{\min}}}{\pp_{\max}}.
\end{equation}
Using this in $C$, we get
 \begin{align*}
     C &\leq \frac{\max\left\{1,(2-\delta)^{-1}\right\}}{\pp_{\min}} -1
+ \tau\frac{2-\gamma_{\max} L_{\min}}{2\sqrt{\pp_{\min}}(2-\delta)}
\left( 1 + \frac{\pp_{\max}}{\sqrt{\pp_{\min}}}\right) \\
    &\leq \frac{\max\left\{1,(2-\delta)^{-1}\right\}}{\pp_{\min}} -1
+ \tau\frac{1}{\sqrt{\pp_{\min}}(2-\delta)}
\left( 1 + \frac{\pp_{\max}}{\sqrt{\pp_{\min}}}\right).
 \end{align*}
 The statement follows.
\qed

\vspace{2ex}
\noindent
\textbf{Proof of Proposition~\ref{prop:20200126b}.}
It follows from \eqref{eq:20210106a} that
\begin{align*}
(2-\delta)\frac{\pp_{\min}}{2}\EE\big[\norm{\bbar{\bx}^{k+1} - \bx^k}^2_{\bGammas^{-1}}\big] 
\leq \EE\big[ F(\bx^{k})+ \uuu_{k}\big] - \EE\big[ F(\bx^{k+1})+ \uuu_{k+1}\big].
\end{align*}
This means that $\displaystyle \big(\EE[F(\bx^{k})+ \uuu_{k}]\big)_{k \in \mathbb{N}}$ is a nonincreasing sequence and 
\begin{align*}
(2-\delta) \frac{\pp_{\min}}{2} \EE\bigg[\sum_{k \in \N}\norm{\bbar{\bx}^{k+1} - \bx^k}^2_{\bGammas^{-1}}\bigg] 
&= (2-\delta) \frac{\pp_{\min}}{2} \sup_{k \in \N}\sum_{h=0}^k  \EE\big[\|\bbar{\bx}^{h+1} - \bx^h\|^2_{\bGammas^{-1}}\big]\\
&\leq \sup_{k \in \N} \EE\big[ F(\bx^{0})+ \uuu_{0}\big] - \EE\big[ F(\bx^{k+1})+ \uuu_{k+1}\big]\\[1ex]
&\leq F(\bx^{0}) + \uuu_0 - F^*<+\infty.
\end{align*}
Therefore, since $\norm{\cdot}^2 \leq (\max_{i} \gamma_i)\norm{\cdot}_{\bGammas^{-1}}^2$, we derive that
\begin{align}
\label{eq:20201223h}\sum_{k \in \mathbb{N}} \norm{\bbar{\bx}^{k+1} - \bx^k}^2 
< \infty \quad \PP\text{-a.s}.
\end{align}
So, it follows that
\begin{align}
\label{eq:20201223i}\norm{\bbar{\bx}^{k+1} - \bx^k} \rightarrow 0 \quad \PP\text{-a.s},
\end{align}
and, since $\norm{\bx^{k+1} - \bx^k} \leq \norm{\bbar{\bx}^{k+1} - \bx^k}$ for all $k \in \mathbb{N}$, we have also
\begin{align}
     \label{eq:20201223j} \sum_{k \in \mathbb{N}} \norm{\bx^{k+1} - \bx^k}^2 < \infty \text{ and } \norm{\bx^{k+1} - \bx^k} \rightarrow 0 \quad \PP\text{-a.s}.
\end{align}
Now, by Lemma~\ref{p:20170918c}, we have $\norm{\bhat{\bx}^k - \bx^k}^2 \leq \tau\sum_{h \in J(k)} \norm{\bx^h - \bx^{h+1}}^2$ and, moreover,
\begin{equation}\label{eq:20201223m}
  \sum_{k \in \mathbb{N}} \sum_{h \in J(k)} \|\bx^h - \bx^{h+1}\|^2 
  \leq \sum_{k \in \N} \tau \norm{\bm{x}^{k} - \bm{x}^{k+1}}^2
   < \infty\ \ \PP\text{-a.s.},
\end{equation}
so that
\begin{align}
\label{eq:20201223l} 
\norm{\bbar{\bx}^{k+1} 
- \bhat{\bx}^k} \leq \norm{\bbar{\bx}^{k+1} - \bx^k} + \norm{\bx^k - \bhat{\bx}^k} \to 0
\ \ \PP\text{-a.s.}
\end{align}
Define, for all $i \in [m]$,
\begin{equation}
\label{eq:20210619a}
v^k_i = \nabla_i f(\bbar{\bx}^{k+1}) - \nabla_i f(\bhat{\bx}^k) + \frac{\Delta_i^{\!k}}{\gamma_i}.
\end{equation}
Then, thanks to the second equation in \eqref{eq:20170921l}, we have
\begin{equation}
\bm{v}^k = (v^k_1, \cdots, v^k_m) \in \nabla f(\bbar{\bx}^{k+1}) + \partial g(\bbar{\bx}^{k+1})
=\partial \left(f + g\right)(\bbar{\bx}^{k+1}).
\end{equation}
Moreover, since $\nabla f$ is Lipschitz continuous,
definition \eqref{eq:20210619a} and equations \eqref{eq:20201223i}, \eqref{eq:20201223l} 
yield $\bm{v}^k \to 0$ $\PP$-a.s. 
\qed
\bibliographystyle{abbrv} 
\bibliography{manuscript}
\end{document}